\documentclass[11pt,reqno,a4paper]{amsart}
\usepackage{amsmath,amsthm,verbatim,amscd,amssymb,setspace,enumitem,hyperref}
\usepackage{exscale,color}

\usepackage{enumitem}
\usepackage{mathrsfs}

\renewcommand{\epsilon}{\varepsilon}
\newcommand{\N}{\mathbb{N}}

\newcommand{\R}{\mathbb{R}}
\newcommand{\C}{\mathbb{C}}

\renewcommand{\Re}{\operatorname{Re}}

\newcounter{mtheorem}
\newtheorem{mtheorem}[mtheorem]{Theorem}

\newtheorem{mcor}[mtheorem]{Corollary}

\newcommand{{\vol}}{\rm vol}
\newcommand{\p}{\partial}

\newcommand{\Ric}{\operatorname{Ric}}

\def\tr{\operatorname{tr}}

\def\Div{\operatorname{div}}

\newtheoremstyle{fancy}{}{}{\itshape}{}{\textbf\bgroup}{.\egroup}{ }{}
\newtheoremstyle{fancy2}{}{}{\rm}{}{\textbf\bgroup}{.\egroup}{ }{}

\theoremstyle{fancy}
\newtheorem{theorem}{Theorem}[section]
\newtheorem{lemma}[theorem]{Lemma}
\newtheorem{corollary}[theorem]{Corollary}

\newtheorem{prop}[theorem]{Proposition}
\newtheorem{conj}[theorem]{Conjecture}

\theoremstyle{fancy2}
\newtheorem{definition}[theorem]{Definition}
\newtheorem{example}[theorem]{Example}
\newtheorem{remark}[theorem]{Remark}

\newtheorem{claim}[theorem]{Claim}

\setlist{leftmargin=*}

\numberwithin{equation}{section}

\begin{document}
\title{Classification results for expanding and shrinking gradient K\"ahler-Ricci solitons}
\date{\today}
\author{Ronan J.~Conlon}
\address{Department of Mathematics and Statistics, Florida International University, Miami, FL 33199, USA}
\email{rconlon@fiu.edu}
\author{Alix Deruelle}
\address{Institut de math\'ematiques de Jussieu, 4, place Jussieu, Boite Courrier 247 - 75252 Paris}
\email{alix.deruelle@imj-prg.fr}
\author{Song Sun}
\address{Department of Mathematics, University of California, Berkeley, CA 94720 }
\email{sosun@berkeley.edu}

\date{\today}

\begin{abstract}
We first show that a K\"ahler cone appears as the tangent cone of
a complete expanding gradient K\"ahler-Ricci soliton
with quadratic curvature decay with derivatives if and only if it has a smooth canonical model
(on which the soliton lives). This allows us to classify two-dimensional complete expanding gradient K\"ahler-Ricci solitons
with quadratic curvature decay with derivatives. We
then show that any two-dimensional complete shrinking gradient K\"ahler-Ricci soliton
whose scalar curvature tends to zero at infinity is, up to pullback by an element of $GL(2,\,\mathbb{C})$, either the
flat Gaussian shrinking soliton on $\mathbb{C}^{2}$ or the $U(2)$-invariant shrinking gradient K\"ahler-Ricci soliton of Feldman-Ilmanen-Knopf on the
blowup of $\mathbb{C}^{2}$ at one point. Finally, we show that up to pullback by an element of $GL(n,\,\mathbb{C})$, the only complete shrinking gradient K\"ahler-Ricci soliton with bounded Ricci curvature on $\mathbb{C}^{n}$ is the flat Gaussian shrinking soliton and on
the total space of $\mathcal{O}(-k)\to\mathbb{P}^{n-1}$ for $0<k<n$ is the $U(n)$-invariant example of Feldman-Ilmanen-Knopf.
In the course of the proof, we establish the uniqueness of the soliton vector field of a complete shrinking gradient K\"ahler-Ricci soliton with bounded Ricci curvature in the Lie algebra of a torus. A key tool used to achieve this result is the Duistermaat-Heckman theorem from symplectic geometry. This provides the first step towards understanding the relationship between complete shrinking gradient K\"ahler-Ricci solitons and algebraic geometry.
\end{abstract}

\maketitle

\markboth{Ronan J.~Conlon, Alix Deruelle, and Song Sun}{Classification results for expanding and shrinking gradient K\"ahler-Ricci solitons}

\section{Introduction}

\subsection{Overview}
 A \emph{Ricci soliton} is a triple $(M,\,g,\,X)$, where $M$ is a Riemannian manifold with a complete Riemannian metric $g$
and a complete vector field $X$ satisfying the equation
\begin{equation}\label{soliton111}
\Ric(g)+\frac{1}{2}\mathcal{L}_{X}g=\frac{\lambda}{2}g
\end{equation}
for some $\lambda\in\{-1,\,0,\,1\}$. If $X=\nabla^{g} f$ for some real-valued smooth function $f$ on $M$,
then we say that $(M,\,g,\,X)$ is \emph{gradient}. In this case, the soliton equation \eqref{soliton111}
reduces to $$\Ric(g)+\operatorname{Hess}(f)=\frac{\lambda}{2}g.$$

If $g$ is complete and K\"ahler with K\"ahler form $\omega$, then we say that $(M,\,g,\,X)$ (or $(M,\,\omega,\,X)$) is a \emph{K\"ahler-Ricci soliton} if
the vector field $X$ is complete and real holomorphic and the pair $(g,\,X)$ satisfies the equation
 \begin{equation}\label{soliton13}
\Ric(g)+\frac{1}{2}\mathcal{L}_{X}g=\lambda g
\end{equation}
for $\lambda$ as above. If $g$ is a K\"ahler-Ricci soliton and if $X=\nabla^{g} f$ for some real-valued smooth function $f$ on $M$, then we say that $(M,\,g,\,X)$ is \emph{gradient}. In this case, one can rewrite the soliton equation \eqref{soliton13} as
\begin{equation*}
\rho_{\omega}+i\partial\bar{\partial}f=\lambda\omega,
\end{equation*}
where $\rho_{\omega}$ is the Ricci form of $\omega$.

For Ricci solitons and K\"ahler-Ricci solitons $(M,\,g,\,X)$, the vector field $X$ is called the
\emph{soliton vector field}. Its completeness is guaranteed by the completeness of $g$
\cite{Zhang-Com-Ricci-Sol}. If the soliton is gradient, then
the smooth real-valued function $f$ satisfying $X=\nabla^g f$ is called the \emph{soliton potential}. It is unique up to a constant.
Finally, a Ricci soliton and a K\"ahler-Ricci soliton are called \emph{steady} if $\lambda=0$, \emph{expanding}
if $\lambda=-1$, and \emph{shrinking} if $\lambda=1$ in equations \eqref{soliton111} and \eqref{soliton13} respectively.

The study of Ricci solitons and their classification is important in the context of Riemannian geometry. For example, they provide a
natural generalisation of Einstein manifolds. Also, to each soliton, one may associate a self-similar solution of the Ricci flow
\cite[Lemma 2.4]{Chowchow} which are candidates for singularity models of the flow. The difference in
normalisations between \eqref{soliton111} and \eqref{soliton13} reflects the difference between the constants preceding the Ricci term
in the Ricci flow and the K\"ahler-Ricci flow when one takes this dynamic point of view.

In this article, we are concerned with complete expanding and shrinking gradient K\"ahler-Ricci solitons.
We consider primarily complete shrinking (respectively expanding) gradient K\"ahler-Ricci solitons with
quadratic curvature decay (resp.~with derivatives) as this assumption greatly simplifies the situation and already
imposes some constraints on the solitons in question. Indeed, it is known that any complete
shrinking (resp.~expanding) gradient Ricci soliton
whose curvature decays quadratically (resp.~with derivatives) along an end must be asymptotic to a cone with a smooth link along that end \cite{Che-Der, Chow, wangl, Siepmann}. Furthermore, any complete shrinking (resp.~expanding) Ricci soliton whose Ricci curvature decays
to zero (resp.~quadratically with derivatives) at infinity must have quadratic curvature decay (resp.~with derivatives) at infinity
\cite{Der-Asy-Com-Egs, munty} and consequently must be asymptotically conical along each of its ends. Kotschwar and Wang \cite{wangl2} have then shown that any two complete shrinking gradient Ricci solitons asymptotic along some end of each to the same cone with a smooth link must in fact be isometric. Thus, at least for shrinking gradient Ricci solitons, classifying those that are complete with Ricci curvature decaying to zero at infinity reduces to classifying their possible asymptotic cone models. Here we are principally concerned with the classification of complete shrinking (resp.~expanding) gradient K\"ahler-Ricci solitons whose curvature tensor has quadratic decay (resp.~with derivatives). Such examples of expanding type have been constructed in \cite{con-der} on certain equivariant resolutions of K\"ahler cones, whereas such examples of shrinking type
include the flat Gaussian shrinking soliton on $\mathbb{C}^{n}$
and those constructed by Feldman-Ilmanen-Knopf \cite{FIK} on the total space of the holomorphic line bundles $\mathcal{O}(-k)$ over $\mathbb{P}^{n-1}$ for $0<k<n$. These shrinking solitons are $U(n)$-invariant and in complex dimension two, they yield
two known examples of complete shrinking gradient K\"ahler-Ricci solitons with scalar curvature tending to zero at infinity;
the flat Gaussian shrinking soliton on $\mathbb{C}^{2}$ and the aforementioned $U(2)$-invariant example
of Feldman-Ilmanen-Knopf on the blowup of $\mathbb{C}^{2}$ at the origin. One of our main results is that in fact up to pullback by an element of $GL(2,\,\mathbb{C})$, these are the only two examples of complete shrinking gradient K\"ahler-Ricci solitons with
scalar curvature tending to zero at infinity in two complex dimensions. Other examples of complete (and indeed incomplete) shrinking gradient K\"ahler-Ricci solitons with quadratic curvature decay have been constructed on the total space of certain holomorphic vector bundles; see for example \cite{Wang, futaki2, lii, Boo}.

\subsection{Main results}

\subsubsection{General structure theorem}
Our first result concerns the structure of complete shrinking (respectively expanding) gradient K\"ahler-Ricci solitons $(M,\,g,\,X)$
with quadratic curvature decay (resp.~with derivatives). By ``quadratic curvature decay with derivatives'', we mean that the curvature $\operatorname{Rm}(g)$ of the K\"ahler-Ricci soliton $g$ satisfies
\begin{equation*}
A_{k}(g):=\sup_{x\in M}|(\nabla^{g})^{k}\operatorname{Rm}(g)|_{g}(x)d_{g}(p,\,x)^{2+k}<\infty\quad\textrm{for all $k\in\mathbb{N}_{0}$},
\end{equation*}
where $d_{g}(p,\,\cdot)$ denotes the distance to a fixed point $p\in M$ with respect to $g$.
\begin{mtheorem}[General structure theorem]\label{main}
Let $(M,\,g,\,X)$ be a complete expanding (respectively shrinking) gradient K\"ahler-Ricci soliton with complex structure $J$
whose curvature $\operatorname{Rm}(g)$ satisfies
\begin{equation*}
A_{k}(g):=\sup_{x\in M}|(\nabla^{g})^{k}\operatorname{Rm}(g)|_{g}(x)d_{g}(p,\,x)^{2+k}<\infty\quad\textrm{for all $k\in\mathbb{N}_{0}$ (resp.~for $k=0$)},
\end{equation*}
where $d_{g}(p,\,\cdot)$ denotes the distance to a fixed point $p\in M$ with respect to $g$. Then:
\begin{enumerate}[label=\textnormal{(\alph{*})}, ref=(\alph{*})]
  \item $(M,\,g)$ has a unique tangent cone at infinity $(C_{0},\,g_{0})$;
  \item there exists a K\"ahler resolution $\pi:M\to C_{0}$ of $C_{0}$ with exceptional set $E$
  with $g_{0}$ K\"ahler with respect to $J_{0}:=\pi_{*}J$ such that
 \begin{enumerate}[label=\textnormal{(\roman{*})}, ref=(\roman{*})]
  \item the K\"ahler form $\omega$ of $g$ and the curvature form $\Theta$ of the hermitian metric on $K_{M}$ (resp.~$-K_{M}$) induced by $\omega$ satisfy
\begin{equation}\label{condition-main}
\int_{V}(i\Theta)^{k}\wedge\omega^{\dim_{\mathbb{C}}V-k}>0
\end{equation}
for all positive-dimensional irreducible analytic subvarieties $V\subset E$ and for all integers $k$ such that $1\leq k\leq \dim_{\C}V$;
\item the real torus action on $C_{0}$ generated by $J_{0}r\partial_{r}$ extends to a holomorphic isometric torus action of $(M,\,g,\,J)$,
where $r$ denotes the radial coordinate of $g_{0}$;
\item  $d\pi(X)=r\partial_{r}$;
\end{enumerate}
\item with respect to $\pi$, we have that
\begin{equation*}
|(\nabla^{g_{0}})^k(\pi_{*}g-g_{0}-\operatorname{Ric}(g_{0}))|_{g_0} \leq C_{k}r^{-4-k}\quad\textrm{for all $k\in\mathbb{N}_{0}$}.
\end{equation*}
\end{enumerate}
\end{mtheorem}
In the expanding case, this theorem provides a converse to \cite[Theorem A]{con-der}.
Also notice that this theorem rules out the existence of shrinking (resp.~expanding) gradient K\"ahler-Ricci solitons with quadratic curvature decay (resp.~with derivatives) on smoothings of K\"ahler cones in contrast to the behaviour in the Calabi-Yau case \cite{Conlon}. This degree of flexibility for expanding and shrinking gradient K\"ahler-Ricci solitons is essentially ruled out due to the requirement of having a conical holomorphic soliton vector field. Finally, it has recently been shown by Kotschwar-Wang \cite[Corollary 1.3]{wangl2} that the isometry group of the link of the asymptotic cone of a complete shrinking gradient K\"ahler-Ricci soliton with quadratic curvature decay embeds into the isometry group of the soliton itself; compare
with statement (b)(ii) of the theorem above.

\subsubsection{Application to expanding gradient K\"ahler-Ricci solitons}
As an application of Theorem \ref{main}, we exploit the uniqueness \cite[Proposition 8.2.5]{ishii} of canonical models of normal varieties to obtain a classification theorem for complete expanding gradient K\"ahler-Ricci solitons whose curvature decays quadratically with derivatives. This provides a partial
answer to question 7 of \cite[Section 10]{FIK}.

\begin{mcor}[Strong uniqueness for expanders]\label{unique}
Let $(C_{0},\,g_{0})$ be a K\"ahler cone with radial function $r$.
Then there exists a unique (up to pullback by biholomorphisms)
complete expanding gradient K\"ahler-Ricci soliton $(M,\,g,\,X)$
whose curvature $\operatorname{Rm}(g)$ satisfies
\begin{equation}\label{hubba}
A_{k}(g):=\sup_{x\in M}|(\nabla^{g})^{k}\operatorname{Rm}(g)|_{g}(x)d_{g}(p,\,x)^{2+k}<\infty\quad\textrm{for all $k\in\mathbb{N}_{0}$},
\end{equation}
where $d_{g}(p,\,\cdot)$ denotes the distance to a fixed point $p\in M$ with respect to $g$, with tangent cone
$(C_{0},\,g_{0})$ if and only if $C_{0}$ has a smooth canonical model. When this is the case,
\begin{enumerate}[label=\textnormal{(\alph{*})}, ref=(\alph{*})]
  \item $M$ is the smooth canonical model of $C_{0}$, and
  \item there exists a resolution map $\pi:M\to C_{0}$ such that $d\pi(X)=r\partial_{r}$ and
\begin{equation*}
|(\nabla^{g_{0}})^k(\pi_{*}g-g_{0}-\operatorname{Ric}(g_{0}))|_{g_0} \leq C_{k}r^{-4-k}\quad\textrm{for all $k\in\mathbb{N}_{0}$}.
\end{equation*}
\end{enumerate}
\end{mcor}

This corollary yields an algebraic description of those K\"ahler cones appearing as the tangent cone of a complete expanding gradient K\"ahler-Ricci soliton
with quadratic curvature decay with derivatives -- such cones are precisely those admitting a smooth canonical model. In general, the canonical model will be singular and in particular, for a two-dimensional cone, it is obtained by contracting all exceptional curves with self-intersection $(-2)$ in the minimal resolution. Applying Corollary \ref{unique} to this case yields a classification of two-dimensional complete expanding gradient K\"ahler-Ricci solitons with quadratic curvature decay with derivatives.
\newpage
\begin{mcor}[Classification of two-dimensional expanders]\label{unique2d}
Let $(C_{0},\,g_{0})$ be a two-dimensional K\"ahler cone with radial function $r$.
Then there exists a unique (up to pullback by biholomorphisms) two-dimensional complete expanding gradient K\"ahler-Ricci soliton
$(M,\,g,\,X)$ whose curvature $\operatorname{Rm}(g)$ satisfies
\begin{equation*}
A_{k}(g):=\sup_{x\in M}|(\nabla^{g})^{k}\operatorname{Rm}(g)|_{g}(x)d_{g}(p,\,x)^{2+k}<\infty\quad\textrm{for all $k\in\mathbb{N}_{0}$},
\end{equation*}
where $d_{g}(p,\,\cdot)$ denotes the distance to a fixed point $p\in M$ with respect to $g$, with tangent cone
$(C_{0},\,g_{0})$ if and only if $C_{0}$ is biholomorphic to either:
\begin{enumerate}[label=\textnormal{(\Roman{*})}, ref=(\Roman{*})]
    \item $\mathbb{C}^{2}/\Gamma$, where $\Gamma$ is a finite subgroup of $U(2)$ acting freely on $\mathbb{C}^{2}\setminus\{0\}$
    that is generated by the matrix $\left(\begin{array}{cc}
                                        e^{\frac{2\pi i}{p}} & 0 \\
                                        0 & e^{\frac{2\pi i q}{p} } \\
                                      \end{array}
                                    \right)$ where $p$ and $q$ coprime integers with $p>q>0$, and after writing
$$\frac{q}{p}=r_{1}-\frac{1}{r_{2}-\frac{1}{\cdots-\frac{1}{r_{k}}}},$$ we have that $r_{j}>2$ for $j=1,\ldots,k$;
    \item $L^{\times}$, the blowdown of the zero section of a negative line bundle $L\to C$ over a proper curve $C$ of genus $g>0$;
      \item $L^{\times}/G$, where $G$ is a finite group of automorphisms of a proper curve $C$ of genus $g>0$ and
      $L^{\times}$ is the blowdown of the zero section of a $G$-invariant negative line bundle $L\to C$ over $C$
       with $G$ acting freely on $L^{\times}$ except at the apex, such that the (unique) minimal good resolution $\pi:M\to L^{\times}/G$
       contains no $(-1)$ or $(-2)$-curves.
  \end{enumerate}
When this is the case,
\begin{enumerate}[label=\textnormal{(\alph{*})}, ref=(\alph{*})]
\item there exists a resolution map $\pi:M\to C_{0}$ such that $d\pi(X)=r\partial_{r}$ and
\begin{equation*}
|(\nabla^{g_{0}})^k(\pi_{*}g-g_{0}-\operatorname{Ric}(g_{0}))|_{g_0} \leq C_{k}r^{-4-k}\quad\textrm{for all $k\in\mathbb{N}_{0}$},
\end{equation*}
and
\item $\pi:M\to C_{0}$ is:
  \begin{enumerate}[label=\textnormal{(\roman{*})}, ref=(\roman{*})]
  \item the minimal resolution $\pi:M\to\mathbb{C}^{2}/\Gamma$ when $C_{0}$ is as in (I);
  \item the blowdown map $\pi:L\to L^{\times}$ when $C_{0}$ is as in (II);
  \item the minimal good resolution $\pi:M\to L^{\times}/G$ when $C_{0}$ is as in (III).
  \end{enumerate}
\end{enumerate}
\end{mcor}

Note that the K\"ahler cones of item (III) here are (up to analytic isomorphism) in one-to-one correspondence with the following data which encodes the exceptional set of the minimal good resolution $\pi:M\to L^{\times}/G$.
        \begin{itemize}
          \item A weighted dual graph which is a star comprising a central vertex with $n$ branches, each of finite length, with the $j$-th vertex of the $i$-th branch labelled with an integer $b_{ij}\geq3$ and the central curve labelled with an integer $b\geq1$.
              The central vertex represents the curve $C$ of genus $g>0$ with self-intersection $-b$ and the $j$-th vertex of the $i$-th branch represents a $\mathbb{P}^{1}$ with self-intersection $-b_{ij}$. The intersection matrix given by this graph must be negative-definite. There is a numerical criterion to determine when this is the case.
                                      \item The analytic type of $C$.
           \item A negative line bundle on $C$ of degree $-b$ (which is the normal bundle of $C$ in $M$).
          \item $n$ marked points on $C$.
        \end{itemize}
        The algebraic equations of $C_{0}$ can be reconstructed from this data by the ansatz in \cite[Section 5]{pinkhamm}.

\newpage
\subsubsection{Application to shrinking gradient K\"ahler-Ricci solitons}

Tian-Zhu \cite{tianzhu2} showed that the soliton vector field of a compact shrinking gradient
K\"ahler-Ricci soliton is unique up to holomorphic automorphisms of the underlying complex manifold.
The method of proof there involved defining a weighted volume functional $F$ which was strictly convex and was shown
to be in fact independent of the metric structure of the soliton. The soliton vector field was then characterised
as the unique critical point of this functional.

Tian-Zhu's proof breaks down in the non-compact case due to the fact that in
general, one cannot a priori guarantee that the weighted volume functional (defined analytically in terms of a certain integral) is well-defined
on a non-compact K\"ahler manifold, let alone investigate its convexity properties. Our key observation to circumvent this difficulty is to apply the \emph{Duistermaat-Heckman theorem} from symplectic geometry which provides a localisation formula
to express the weighted volume functional in terms of an algebraic formula involving only the fixed point set of a torus action. This latter algebraic formula is much more amenable. Combined with a version of Matsushima's theorem \cite{matsushima} for complete non-compact shrinking gradient K\"ahler-Ricci solitons (cf.~Theorem \ref{matty}) and Iwasawa's theorem \cite{iwasawa},
we are able to implement Tian-Zhu's strategy of proof for the compact case to obtain the following non-compact analogue of their uniqueness result.

\begin{mtheorem}[Uniqueness of the soliton vector field for shrinkers]\label{uunique}
Let $M$ be a non-compact complex manifold with complex structure $J$ endowed with the effective holomorphic action of a real torus $T$. Denote by $\mathfrak{t}$ the Lie algebra of $T$. Then there exists at most one element $\xi\in\mathfrak{t}$ that admits a complete shrinking gradient K\"ahler-Ricci soliton $(M,\,g,\,X)$ with bounded Ricci curvature with $X=\nabla^g f=-J\xi$ for a smooth real-valued function $f$ on $M$.
\end{mtheorem}

We expect that the assumption of bounded Ricci curvature is superfluous in the statement of this theorem
and that, given the uniqueness of the soliton vector field in the Lie algebra of the torus here, the corresponding shrinking gradient K\"ahler-Ricci soliton is also unique up to automorphisms of the complex structure commuting with the flow of the soliton vector field.

Not only does the Duistermaat-Heckman theorem imply the uniqueness of the soliton vector field in our case,
it also provides a formula to compute the unique critical point of the weighted volume functional, the point at which
the soliton vector field is achieved. Using this formula, we compute explicitly the soliton vector field of a
shrinking gradient K\"ahler-Ricci soliton with bounded Ricci curvature on $\mathbb{C}^{n}$ and
on the total space of $\mathcal{O}(-k)$ over $\mathbb{P}^{n-1}$ for $0<k<n$ (see Examples \ref{one} and \ref{three} respectively). This recovers the polynomials of Feldman-Ilmanen-Knopf
\cite[equation (36)]{FIK}. Having identified the soliton vector field on these manifolds,
we then use our non-compact version of Matsushima's theorem (Theorem \ref{matty}) together with an application of Iwasawa's theorem \cite{iwasawa} to deduce that, up to pullback by an element of $GL(n,\,\mathbb{C})$, the corresponding soliton metrics have to be invariant under the action of $U(n)$. Consequently, thanks to a uniqueness theorem of Feldman-Ilmanen-Knopf \cite[Proposition 9.3]{FIK}, up to pullback by an
element of $GL(n,\,\mathbb{C})$, the shrinking gradient K\"ahler-Ricci soliton must be the flat Gaussian shrinking soliton if on $\mathbb{C}^{n}$ or the unique $U(n)$-invariant shrinking gradient K\"ahler-Ricci soliton constructed by Feldman-Ilmanen-Knopf \cite{FIK} if on the total space of $\mathcal{O}(-k)$ over $\mathbb{P}^{n-1}$ for $0<k<n$.

In the complex two-dimensional case, we are actually able to identify the underlying complex manifold of a shrinking gradient K\"ahler-Ricci soliton
whose scalar curvature tends to zero at infinity as either $\mathbb{C}^{2}$ or $\mathbb{C}^{2}$ blown up at the origin using the fact that the soliton, if non-trivial, has an asymptotic cone with strictly positive scalar curvature. Using a classification
theorem for Sasaki manifolds in real dimension three \cite{belgun} and the fact that a shrinking K\"ahler-Ricci soliton can only contain $(-1)$-curves
in complex dimension two, this is enough to identify the asymptotic cone at infinity as $\mathbb{C}^{2}$ from which the identification of the underlying complex manifold easily follows. Combined with the above discussion, this yields a complete classification of such solitons in two complex dimensions.

These conclusions are summarised in the following theorem.
\begin{mtheorem}[Classification of shrinkers]\label{classify}
Let $(M,\,g,\,X)$ be a complete shrinking gradient K\"ahler-Ricci soliton.

{\rm (1)} If $M=\mathbb{C}^{n}$ and $g$ has bounded Ricci curvature, then up to pullback by an element of $GL(n,\,\mathbb{C})$,
$(M,\,g,\,X)$ is the flat Gaussian shrinking soliton.

{\rm (2)} If $M$ is the total space of $\mathcal{O}(-k)$ over $\mathbb{P}^{n-1}$ for $0<k<n$ and $g$ has bounded Ricci curvature, then up to pullback by an element of $GL(n,\,\mathbb{C})$, $(M,\,g,\,X)$ is the unique $U(n)$-invariant shrinking gradient K\"ahler-Ricci soliton constructed by Feldman-Ilmanen-Knopf \cite{FIK}.

{\rm (3)} If $\dim_{\mathbb{C}}M=2$ and the scalar curvature of $g$ tends to zero at infinity, then up to pullback by an element of $GL(2,\,\mathbb{C})$,
$(M,\,g,\,X)$ is the flat Gaussian shrinking soliton on $\mathbb{C}^{2}$ or the unique $U(2)$-invariant shrinking gradient K\"ahler-Ricci soliton
constructed by Feldman-Ilmanen-Knopf \cite{FIK} on the total space of $\mathcal{O}(-1)$ over $\mathbb{P}^{1}$.
\end{mtheorem}

As exemplified in complex dimension two, in contrast to the expanding case, not many K\"ahler cones appear as tangent cones of complete shrinking gradient K\"ahler-Ricci solitons.

\subsection{Outline of Paper} We begin in Section \ref{prelim} by recalling the basics of K\"ahler cones, K\"ahler-Ricci solitons, metric measure spaces, and the relevant algebraic geometry that we require. We also mention some important properties of the soliton vector field and of real holomorphic vector fields that commute with the soliton vector field. In Section \ref{proofofA}, we prove Theorem \ref{main} for expanding gradient K\"ahler-Ricci solitons. The proof for shrinking K\"ahler-Ricci solitons is verbatim. By a theorem of Siepmann \cite[Theorem 4.3.1]{Siepmann}, under our curvature assumption, a complete expanding gradient Ricci soliton flows out of a Riemannian cone. Our starting point is to prove some preliminary lemmas before
providing a refinement of Siepmann's theorem, namely
Theorem \ref{ac}, where, in the course of its proof, we construct a diffeomorphism between the cone and the end of the Ricci soliton
using the flow of the soliton vector field that encapsulates the asymptotics of the soliton along the end. We then show in Proposition \ref{biholo} that
if the soliton is K\"ahler, then the cone is K\"ahler with respect to a complex structure that makes
the aforementioned diffeomorphism a biholomorphism. In Theorem \ref{hello}, this biholomorphism is then shown to extend to an equivariant resolution with the properties as stated in Theorem \ref{main}.

In the first part of Section \ref{classie}, we use Theorem \ref{main} to prove Corollary \ref{unique}.
This also requires an application of previous work from \cite{con-der}. In the latter part of Section \ref{classie}, we apply Corollary \ref{unique} to two-dimensional expanding gradient K\"ahler-Ricci solitons to conclude the statement of Corollary \ref{unique2d}, making use of the classification of two-dimensional K\"ahler cones, namely Theorem \ref{classs}.

From Section \ref{futake} onwards, we turn our attention exclusively to complete shrinking gradient K\"ahler-Ricci solitons.
We begin in Section \ref{matsushimaa} by proving a Matsushima-type theorem stating that the Lie algebra of real holomorphic vector fields commuting with
the soliton vector field may be written as a direct sum. This is the statement of Theorem \ref{matty}.
Our proof of this theorem follows a manner similar to the proof of Matsushima's theorem on K\"ahler-Einstein Fano manifolds
as presented in \cite[Proof of Theorem 5.1]{kobayashi1}. After deriving some properties of the automorphism groups of a complete shrinking gradient K\"ahler-Ricci soliton $(M,\,g,\,X)$, we then apply Theorem \ref{matty} to prove the maximality of a certain compact Lie group acting on $M$. This is the content of Corollary \ref{maxx}.

We continue in Section \ref{weighted2} by showing in Proposition \ref{tangent} that every real holomorphic Killing vector field on $M$ admits
a Hamiltonian potential satisfying a certain linear equation. This allows us to define a moment map in Definition \ref{moments}
which is used in the definition of the weighted volume functional in Definition \ref{modfut}. The weighted volume functional is vital in proving
the uniqueness statement of Theorem \ref{uunique}, namely that the soliton vector field
of a complete shrinking gradient K\"ahler-Ricci soliton with bounded Ricci curvature in the Lie algebra of a torus is unique. The weighted volume functional is the same as that defined by Tian-Zhu \cite{tianzhu2}, although in our situation it is defined as an integral over the non-compact manifold $M$. This is compensated for by the fact that the domain of definition of the weighted volume functional is restricted to an open cone in the Lie algebra of the torus. Several important properties of the weighted volume functional are then derived in Lemma \ref{critical}, including the crucial fact that it has a unique critical point in its open cone of definition given by the complex structure applied to the soliton vector field $X$.
We conclude this subsection by taking note of the fact that the Duistermaat-Heckman formula may be used to compute the weighted volume functional. In particular, it follows that the weighted volume functional is independent of the complete shrinking gradient K\"ahler-Ricci soliton.

In Section \ref{generalunique}, we prove the uniqueness statement of Theorem \ref{uunique} which has been recalled in the statement of Theorem \ref{uniquee}. The proof of this theorem follows as in \cite[p.322]{tianzhu2} using Iwasawa's theorem \cite{iwasawa} and the corollary of Matsushima's theorem, namely Corollary \ref{maxx} discussed above. Section \ref{pclass} then comprises an application of Theorems \ref{main} and \ref{uunique}
to classify complete shrinking gradient K\"ahler-Ricci solitons with bounded Ricci curvature on $\mathbb{C}^{n}$ and on the total space of $\mathcal{O}(-k)\to\mathbb{P}^{n-1}$ for $0<k<n$. This completes the proof of items (1) and (2) of Theorem \ref{classify}.

In Section \ref{underlying}, we show that the underlying complex manifold $M$ of a two-dimensional shrinking gradient K\"ahler-Ricci soliton
with scalar curvature decaying to zero at infinity is either $\mathbb{C}^{2}$ or $\mathbb{C}^{2}$ blown up at the origin. This is the statement of
Theorem \ref{blue}. Combined with items (1) and (2) of Theorem \ref{classify}, Theorem \ref{blue} suffices to prove item (3) of Theorem \ref{classify}.
The proof of Theorem \ref{blue} relies on first identifying the underlying complex space of the tangent cone. The fact that any non-flat shrinking gradient Ricci soliton has positive scalar curvature implies that the same property holds true on the tangent cone. From this we can identify the cone as a quotient of $\mathbb{C}^{2}$. Theorem \ref{main} then tells us that $M$ is a resolution of this cone which, by virtue of the shrinking K\"ahler-Ricci soliton equation,
can only contain $(-1)$-curves. It turns out that the only possibility is that the cone is biholomorphic to $\mathbb{C}^{2}$ and $M$ is as stated
in Theorem \ref{blue}.

We conclude the paper in Section \ref{conclusion} with some closing remarks and open problems. In Appendix \ref{heckman}, Section \ref{a}, we recall the statement of the Duistermaat-Heckmann theorem on a non-compact symplectic manifold in Theorem \ref{dhthm} as presented in \cite{wu}.
We also provide an outline of its proof in Section \ref{c} after introducing some preliminaries in Section \ref{b}.
We then use it to compute the weighted volume functional and its unique critical point on
$\mathbb{C}^{n}$, on the total space of $\mathcal{O}(-k)$ over $\mathbb{P}^{n-1}$ for $0<k<n$, and on certain holomorphic line bundles over Fano manifolds
in Section \ref{d}. In Section \ref{e}, we characterise algebraically, in the setting of asymptotically conical K\"ahler manifolds, those elements in the Lie algebra of a torus that admit a Hamiltonian potential that is proper and bounded below. A precise statement is given in Theorem \ref{fme}.
For such elements of the Lie algebra of the torus, the weighted volume functional is defined. Finally, in Section \ref{f}, we show directly
that the weighted volume functional is defined on a complete shrinking gradient K\"ahler-Ricci soliton with bounded Ricci curvature
without appealing to the Duistermaat-Heckman theorem. This conclusion follows from the
estimates we derive in Proposition \ref{prop-growth-pot} on the growth of those Hamiltonian potentials
that are proper and bounded below.

\subsection{Acknowledgements} The authors wish to thank Richard Bamler, John Lott,  Bing Wang, and Yuanqi Wang for useful discussions and for their interest in this work, and Brett Kotschwar for his interest, for useful discussions involving \cite{wangl, wangl2}, and for comments on a preliminary version
of this article. Part of this work was carried out while the first two authors where visiting the Institut Henri Poincar\'e
as part of the Research in Paris program in July 2019. They wish to thank the institute for their hospitality
and for the excellent working conditions provided.

The first author is supported by Simons Collaboration Grant $\#581984$. The second author is supported
by grant ANR-17-CE40-0034 of the French National Research Agency ANR (Project CCEM) and Fondation Louis D., Project ``Jeunes G\'eom\`etres''. The third author is supported by the Simons Collaboration Grant on Special
Holonomy in Geometry, Analysis, and Physics ($\#488633$), an Alfred P.~Sloan fellowship, and NSF grant DMS-$1708420$.

\newpage
\section{Preliminaries}\label{prelim}

\subsection{Riemannian cones} For us, the definition of a Riemannian cone will take the following form.
\begin{definition}\label{cone}
Let $(S, g_{S})$ be a closed Riemannian manifold. The \emph{Riemannian cone} $C_{0}$ with \emph{link} $S$ is defined to be $\R_{>0}\times S$ with metric $g_0 = dr^2 \oplus r^2g_{S}$ up to isometry. The radius function $r$ is then characterised
intrinsically as the distance from the apex in the metric completion.
\end{definition}
The following is a simple computation.

\begin{lemma}\label{scalar}
Let $(S,\,g_{S})$ be a closed Riemannian manifold of real dimension $m$ and let $(C_{0},\,g_{0})$ be the Riemannian cone with link $S$ and radial function $r$. Then the Ricci curvature $\Ric(g_{S})$ of $g_{S}$ and the Ricci curvature $\Ric(g_{0})$ of the cone metric $g_{0}$ over $(S,\,g_{S})$ are related by
$$\Ric(g_{0})=\Ric(g_{S})-(m-1)g_{S}.$$
In particular, the scalar curvatures $R_{g_{0}}$ and $R_{g_{S}}$ of $g_{0}$ and $g_{S}$ respectively are related by
\begin{equation*}
R_{g_{0}}=\frac{1}{r^{2}}\left(R_{g_{S}}-m(m-1)\right).
\end{equation*}
\end{lemma}

\subsection{K\"ahler cones}
We may further impose that a Riemannian cone is K\"ahler, as the next definition demonstrates.
\begin{definition}A \emph{K{\"a}hler cone} is a Riemannian cone $(C_{0},g_0)$ such that $g_0$ is K{\"a}hler, together with a choice of $g_0$-parallel complex structure $J_0$. This will in fact often be unique up to sign. We then have a K{\"a}hler form $\omega_0(X,Y) = g_0(J_0X,Y)$, and $\omega_0 = \frac{i}{2}\p\bar{\p} r^2$ with respect to $J_0$.
\end{definition}

The vector field $r\partial_{r}$ is real holomorphic and $\xi:=J_{0}r\partial_r$ is real holomorphic and Killing \cite[Appendix A]{Yau}. This latter vector field is known as the \emph{Reeb
vector field}. The closure of its flow in the isometry
group of the link of the cone generates the holomorphic isometric action of a real torus on $C_{0}$ that
fixes the apex of the cone. We call a K\"ahler cone ``quasiregular'' if this action is an $S^{1}$-action (and,
in particular, ``regular'' if this $S^{1}$-action is free), and ``irregular'' if the action generated is that of a
real torus of rank $>1$.

Every K\"ahler cone is affine algebraic.
\begin{theorem}\label{t:affine}
For every K{\"a}hler cone $(C_{0},g_0,J_0)$, the complex manifold $(C_{0},J_0)$ is isomorphic to the smooth part of a normal algebraic variety $V \subset \C^N$ with one singular point. In addition, $V$ can be taken to be invariant under a $\C^*$-action $(t, z_1,\ldots,z_N) \mapsto (t^{w_1}z_1,\ldots,t^{w_N}z_N)$ such that all of the $w_i $ are positive
integers.
\end{theorem}
\noindent This can be deduced from arguments written down by van Coevering in \cite[\S 3.1]{vanC4}.

K\"ahler cones of complex dimension two have been classified.
\begin{theorem}[{\cite[Theorem 8]{belgun}, \cite[Theorem 1.1]{pinkhamm}}]\label{classs}
Let $C_{0}$ be a K\"ahler cone of complex dimension two. Then $C_{0}$ is biholomorphic to either:
\begin{enumerate}
  \item $\mathbb{C}^{2}/\Gamma$, where $\Gamma$ is a finite subset of $U(2)$ acting freely on $\mathbb{C}^{2}\setminus\{0\}$; or
  \item  the blowdown $L^{\times}$ of the zero section of a negative line bundle $L\to C$ over a smooth proper curve $C$ of
  genus $g$ with $g>0$; or
    \item $L^{\times}/G$, where $G$ is a finite group of automorphisms of a proper curve $C$ of genus $g>0$ and
      $L^{\times}$ is the blowdown of the zero section of a $G$-invariant negative line bundle $L\to C$ over $C$
       with $G$ acting freely on $L^{\times}$ except at the apex.
            \end{enumerate}
In cases (ii) and (iii), the corresponding Reeb vector field is quasi-regular and is generated by a scaling of the standard $S^{1}$-action
on the fibres of $L$.
\end{theorem}

Any automorphism of a resolution of a K\"ahler cone preserves the exceptional set of the resolution.
\begin{lemma}\label{auto}
Let $\pi:M\to C_{0}$ be a resolution of a K\"ahler cone $C_{0}$ with exceptional set $E$.
Denote by $J$ the complex structure on $M$. Then for any automorphism $\sigma$ of $(M,\,J)$, $\sigma(E)\subseteq E$. In particular, real holomorphic vector fields on $M$ are tangent to $E$.
\end{lemma}
\noindent Such an automorphism of $(M,\,J)$ therefore descends to an automorphism of the cone fixing the apex of the cone.
\begin{proof}[Proof of Lemma \ref{auto}]
If $\sigma$ is an automorphism of $(M,\,J)$, then $\pi\circ\sigma:M\to C_{0}$ is also a resolution of $C_{0}$. The exceptional set of this resolution is then a compact analytic subset of $M$. Since $E$ is the maximal compact analytic subset
of $M$, we must have that $(\pi\circ\sigma)^{-1}(o)\subseteq E$, where $o$ denotes the apex of $C_{0}$, i.e., $\sigma^{-1}(E)\subseteq E$.
\end{proof}

The holomorphic torus action on a K\"ahler cone leads to the notion of an \emph{equivariant resolution}.
\begin{definition}\label{equivariantt}
Let $C_0$ be a K\"ahler cone with complex strucutre $J_{0}$, let $\pi:M\to C_0$ be a resolution of $C_0$, and let
$G$ be a Lie subgroup of the automorphism group of $(C_{0},\,J_{0})$ fixing the apex of $C_{0}$.
We say that $\pi:M\to C_0$ is an \emph{equivariant resolution with respect to $G$} if the action of $G$ on $C_{0}$ extends to a holomorphic action
on $M$ in such a way that $\pi(g\cdot x)=g\cdot\pi(x)$ for all $x\in M$ and $g\in G$.
\end{definition}
\noindent Such a resolution of a K\"ahler cone always exists; see \cite[Proposition 3.9.1]{kollar}.

A closed Riemannian manifold $(S,\,g_{S})$ is \emph{Sasaki} if and only if its Riemannian cone is a K\"ahler cone \cite{book:Boyer}, in which case we identify $(S,\,g_{S})$ with the level set $\{r=1\}$ of its corresponding K\"ahler cone, $r$ here denoting the radial function of the cone. The restriction of the Reeb vector field to this level set induces a non-zero vector field $\xi\equiv J_{0}r\partial_r|_{\{r=1\}}$ on $S$. Let $\eta$ denote the $g_{S}$-dual one-form of $\xi$. Then we get a $g_{S}$-orthogonal decomposition $TS=\mathcal{D}\oplus\langle\xi\rangle$, where $\mathcal{D}$ is the kernel of $\eta$ and $\langle\xi\rangle$ is the $\mathbb{R}$-span of $\xi$ in $TS$, and correspondingly a decomposition of the metric $g_{S}$ as $g_{S}=\eta\otimes\eta + g^{T}$, where $g^{T}=g_{S}|_{\mathcal{D}}$. The metric $g^{T}$ is invariant under the flow of $\xi$ and induces a Riemannian metric on the local leaf space of the foliation of $S$ induced by the flow of $\xi$. We call $g^{T}$ the \emph{transverse metric}. We can then define the \emph{transverse scalar curvature} $R^{T}$ and the \emph{transverse Ricci curvature $\Ric^{T}$} as the corresponding curvatures of $g^{T}$. We also get an induced \emph{transverse complex structure $J^{T}$} on the local leaf space of the foliation with respect to which $g^{T}$ is K\"ahler
given by the restriction of the complex structure of the cone to $\mathcal{D}$. In particular, $\Ric^{T}$ will be $J^{T}$-invariant.
We have the following relationships between the various curvatures.

\begin{lemma}[{\cite[Theorem 7.3.12]{book:Boyer}}]\label{ricci}
Let $(S,\,g_{S})$ be a $(2n+1)$-real dimensional Sasaki manifold. Then the following identities hold.
\begin{enumerate}[label=\textnormal{(\roman{*})}, ref=(\roman{*})]
\item $\Ric(g_{S})(X,\,\xi)=2n\eta(X)$ for any vector field $X$.
\item $\Ric(g_{S})(X,\,Y)=\Ric^{T}(X,\,Y)-2g_{S}(X,\,Y)$ for any vector fields $X,\,Y\in\mathcal{D}$.
\end{enumerate}
\end{lemma}
In particular, we deduce that:
\begin{corollary}\label{morescalar}
Let $(S,\,g_{S})$ be a $(2n+1)$-real dimensional Sasaki manifold with scalar curvature $R_{g_{S}}$. Then $$R^{T}=R_{g_{S}}+2n.$$
\end{corollary}

\subsection{Canonical models}

Resolutions of K\"ahler cones that are consistent with admitting an expanding K\"ahler-Ricci soliton are of the following type.
\begin{definition}[{\cite[Definition 8.2.4]{ishii}}]\label{cannon}
A partial resolution $\pi:M\to C_{0}$ of a normal isolated singularity $x\in C_{0}$ is called a \emph{canonical model} if
\begin{enumerate}[label=\textnormal{(\roman{*})}, ref=(\roman{*})]
\item $M$ has at worst canonical singularities;
\item $K_{M}$ is $\pi$-ample.
\end{enumerate}
\end{definition}
Note that the choice of partial resolution $\pi$ is part of the data here. The existence of a canonical model $\pi:M\to C_{0}$ is guaranteed by \cite{BCHM}
and it is unique up to isomorphisms over $C_{0}$ \cite[Proposition 8.2.5]{ishii}. We have the following criterion to determine when $K_{M}$ is $\pi$-ample.
\begin{theorem}[Nakai's criterion for a mapping {\cite[Corollary 1.7.9]{lazarfeld}}]\label{nakai}
Let $\pi:M\to C_{0}$ be a proper morphism of schemes. A $\mathbb{Q}$-divisor $D$ on $M$ is $\pi$-ample
if and only if $(D^{\operatorname{dim}_{\mathbb{C}}V}\cdot V)>0$ for every irreducible subvariety
$V\subset M$ of positive dimension that maps to a point in $C_{0}$.
\end{theorem}
\noindent In particular, in complex dimension two, item (ii) of Definition \ref{cannon} implies that the exceptional set of the canonical model cannot contain any $(-1)$ or $(-2)$-curves.

In our case, $C_{0}$ will be a K\"ahler cone, hence is affine algebraic, and $x$ will be the apex of the cone. As the next lemma shows, the canonical model of $C_{0}$ is quasi-projective.
\begin{lemma}\label{quasi}
Let $\pi:M\to C_{0}$ be the canonical model of a K\"ahler cone $C_{0}$. Then $M$ is quasi-projective.
\end{lemma}

\begin{proof}
From Theorem \ref{t:affine}, we see that $C_{0}$ admits an affine embedding that is invariant under a $\mathbb{C}^{*}$-action with positive integer weights. Taking the weighted projective closure of $C_{0}$ with respect to this action, we obtain a projective compactification $\overline{C_{0}}$ of $C_{0}$ by adding an ample divisor $D$ at infinity. In particular, $\overline{C_{0}}$ will have at worst orbifold singularities along $D$. Let $\sigma:\overline{N}\to\overline{C_{0}}$ denote the canonical model of $\overline{C_{0}}$. By construction, $\overline{N}$ is projective and the restricted map $\sigma|_{N}:N\to C_{0}$, where $N:=\overline{N}\setminus\sigma^{-1}(D)$, is a canonical model of $C_{0}$. By uniqueness of canonical models, $M$ must be biholomorphic to $N$, hence $\overline{N}$ provides a projective compactification of $M$ obtained by adjoining the set
$\sigma^{-1}(D)$ to $M$ at infinity. $M$ is therefore quasi-projective, as claimed.
\end{proof}

In addition, uniqueness of the canonical model implies that
the canonical model of a K\"ahler cone is equivariant with respect to the torus action on $C_{0}$ generated
by the flow of $J_{0}r\partial_{r}$ when the canonical model is smooth.

\begin{lemma}\label{liftt}
Let $C_{0}$ be a K\"ahler cone with complex structure $J_{0}$ and radial function $r$
and let $\pi:M\to C_{0}$ denote the canonical model of $C_{0}$. If $M$ is smooth, then the resolution $\pi:M\to C_{0}$ is equivariant
with respect to the holomorphic isometric torus action on $C_{0}$ generated by $J_{0}r\partial_{r}$. In particular, there
exists a holomorphic vector field $X$ on $M$ such that $d\pi(X)=r\partial_{r}$.
\end{lemma}

\begin{proof}
Let $T$ denote the torus generated by the flow of $J_{0}r\partial_{r}$ and let $T_{\mathbb{C}}$
denote its complexification. We will show that the holomorphic action of $T_{\C}$ on $C_{0}$ lifts to a holomorphic action on $M$.

For $h\in T_{\C}$, let $\psi_{h}:C_{0}\to C_{0}$ denote the corresponding automorphism.
Since $\psi_{h}\circ\pi:M\to C_{0}$ is again a canonical model and since
$\pi:M\to C_{0}$ is unique up to isomorphisms over $C_{0}$ \cite[Proposition 8.2.5]{ishii},
there exists a unique biholomorphism $\widetilde{\psi}_{h}:M\to M$ such that $\pi\circ\widetilde{\psi}_{h}=\psi_{h}\circ\pi$.
This is the desired lift of $\psi_{h}$. Thus, we have a well-defined map $\phi:T_{\C}\times M\to M$
defined by $\phi(h,\,x)=\widetilde{\psi}_{h}(x)$.
Since $\widetilde{\psi}_{h}$ coincides with $\psi_{h}$ off of
the exceptional set $E$ of the resolution $\pi:M\to C_{0}$, $\phi|_{T_{\C}\times(M\setminus E)}$ is
holomorphic. We wish to show that $\phi$ is holomorphic globally.

To this end, let $h\in T_{\C}$, let $x\in E$, let $y=\widetilde{\psi}_{h}(x)\in E$, let $B_{x}$ be an open ball in
a chart containing $x$ with $x$ in its interior and let $B_{y}$ be an open ball in a chart containing
$y$ with $y$ in its interior. Since $\widetilde{\psi}_{h}$ is continuous, by shrinking $B_{x}$ if necessary, we may assume that
$\widetilde{\psi}_{h}(B_{x})\subset B_{y}$. Let $U$ be a neighbourhood of $h$ in $T_{\C}$ such that $\widetilde{\psi}_{h'}(B_{x}\setminus E)\subset B_{y}$
for all $h'\in U$. Again, this is possible because $\phi|_{T_{\C}\times(M\setminus E)}$ is continuous.
Then since $\widetilde{\psi}_{h'}:M\to M$ is itself continuous and preserves $E$ for each fixed $h'\in U$ and
since $B_{x}\cap E$ lies in the closure of $B_{x}\setminus E$,
we have, after shrinking $B_{x}$ further if necessary, that $\widetilde{\psi}_{h'}(B_{x})\subset B_{y}$ for all $h'\in U$.

Next, let $N_{\epsilon}:=\{x\in C_{0}:r(x)<\epsilon\}$ for $\epsilon>0$. Then $N_{\epsilon}\setminus\{o\}$ is
foliated by disjoint punctured discs, obtained as the orbits in $N_{\epsilon}\setminus\{o\}$ of a $\mathbb{C}^{*}$-action from within $T_{\C}$.
The open set $\widehat{N}_{\epsilon}:=\pi^{-1}(N_{\epsilon})$ will then be a neighbourhood of $E$ in $M$ with $\widehat{N}_{\epsilon}\setminus E$ foliated by
disjoint punctured discs. Let $B_{x}'\subset B_{x}$ be an open ball containing $x$ strictly contained in $B_{x}$
such that $\partial B_{x}\cap\partial B_{x}'=\emptyset$ and let $\epsilon>0$
be sufficiently small so that each point of $(\widehat{N}_{\epsilon}\cap B_{x}')\setminus E$ is contained in a punctured holomorphic disk of radius $\epsilon$ which itself is contained in $(\widehat{N}_{\epsilon}\cap B_{x})\setminus E$.
Let $V:=\widehat{N}_{\epsilon}\cap B_{x}'$. Then this is an open neighbourhood of $x$ in $M$ and each point $z\in V\setminus E$ will
lie on a unique punctured holomorphic disk which we shall denote by $D_{z}$. We have that $D_{z}\subseteq(\widehat{N}_{\epsilon}\cap B_{x})\setminus E$
and that $\partial D_{z}\subseteq B_{x}\cap\partial\widehat{N}_{\epsilon}\subset M\setminus E$. Let $\overline{D}_{z}$ denote the closure of $D_{z}$ in $M$.
Since $\widetilde{\psi}_{h'}:M\to M$ is holomorphic for each fixed $h'\in U$, we see
from the maximum principle that for all $z\in V\setminus E$ and all $h'\in U$,
$$|\widetilde{\psi}_{h'}(z)|\leq\sup_{w\,\in\,\overline{D}_{z}}|\widetilde{\psi}_{h'}(w)|\leq\sup_{\{w\,\in\,\partial D_{z}\}}|\widetilde{\psi}_{h'}(w)|
\leq\sup_{\{w\,\in\,B_{x}\cap\partial\widehat{N}_{\epsilon}\}}|\phi(h',\,w)|\leq C$$
for some constant $C>0$, where the last inequality follows from the fact that
$\phi|_{T_{\C}\times(M\setminus E)}$ is holomorphic, hence continuous.
Thus, $\phi|_{U\times(V\setminus E)}$
is a bounded holomorphic function. Since $E\cap V$ is an analytic subset of $V$, it follows from the Riemann
Extension Theorem that $\phi|_{U\times(V\setminus E)}$ has a unique extension to a holomorphic function
$\widetilde{\phi}:U\times V\to B_{y}$. Due to the fact that $\phi(h',\,\cdot):M\to M$
is holomorphic for each fixed $h'\in T_{\mathbb{C}}$, we have from uniqueness of holomorphic extensions
that $\widetilde{\phi}(h',\,\cdot)=\phi(h',\,\cdot):V\to B_{y}$ for all $h'\in U$ so that
in fact $\widetilde{\phi}=\phi|_{U\times V}$. Thus, we see that $\phi|_{U\times V}$ is holomorphic.
Since being holomorphic is a local property, this suffices to show that $\phi:T_{\mathbb{C}}\times M\to M$ is holomorphic, as desired.
\end{proof}

\subsection{Minimal models}

We consider two types of resolution of a normal isolated singularity of complex dimension two, the first being the minimal resolution.
\begin{definition}[{\cite[Definition 7.1.14]{ishii}}]
A resolution $\pi:M\to C_{0}$ of a normal isolated singularity $x\in C_{0}$ is called a
\emph{minimal resolution} if for every resolution $\pi':M'\to C_{0}$ of $C_{0}$ there exists a unique
morphism $\varphi:M'\to M$ such that $\pi'$ factors as $\pi'=\pi\circ\varphi$.
\end{definition}

By definition, if there exists a minimal resolution, then it is unique up to isomorphisms over $C_{0}$. The following shows that there exists a minimal resolution for
two-dimensional isolated normal singularities.
\begin{theorem}[{\cite[Theorem 7.1.15]{ishii}}]\label{minn}
Assume that $\operatorname{dim}_{\mathbb{C}}C_{0}= 2$. A resolution $\pi:M\to C_{0}$
of an isolated normal singularity $x\in C_{0}$ is the minimal resolution if and only if $\pi^{-1}(\{x\})$ does not contain a
$(-1)$-curve. In particular, there exists a minimal resolution.
\end{theorem}

We also consider ``good'' resolutions. We henceforth follow \cite{orlik, pinkhamm}.
\begin{definition}
A resolution $\pi:M\to C_{0}$ of a normal isolated surface singularity $x\in C_{0}$ is called \emph{good} if
\begin{enumerate}
  \item all of the components of the exceptional divisor of $\pi:M\to C_{0}$ are smooth and intersect transversally;
  \item not more than two components pass through any given point;
  \item two different components intersect at most once.
\end{enumerate}
\end{definition}

It is known that there is a unique resolution which is minimal among all good resolutions for two-dimensional isolated normal singularities \cite[Theorem 5.12]{laufer} which we henceforth refer to as the ``minimal good resolution'' (not to be confused with the ``minimal resolution''). In general, the minimal resolution of a two-dimensional isolated normal singularity will not be the minimal good resolution of the singularity; by Theorem \ref{minn}, the two coincide precisely when the minimal good resolution does not contain any $(-1)$-curves.

For the K\"ahler cones of Theorem \ref{classs}(ii), the minimal good resolution is given by $\pi:L\to L^{\times}$, that is, by contracting
the zero section of $L$. By adjunction, this resolution will be the canonical model of the singularity.

As for the K\"ahler cones of Theorem \ref{classs}(iii), the situation is slightly more complicated. The minimal good resolution of $L^{\times}/G$ is obtained as follows. A partial resolution is given by the induced map $\pi:L/G\to L^{\times}/G$ between the quotient spaces. The variety $L/G$ will only have isolated cyclic quotient singularities along the exceptional set of $\pi$ \cite[Lemma 3.5]{pinkhamm}, which comprises a single curve $C$ of genus $g>0$.
Each of these cyclic quotient singularities has a minimal resolution with exceptional set a string of $\mathbb{P}^{1}$'s.
Resolving them with this resolution then yields the minimal good resolution of $L^{\times}/G$, the exceptional set of which will then comprise the curve $C$ (of genus $g>0$) with branches of $\mathbb{P}^{1}$'s stemming from finitely many points of $C$. The original singularity is determined up to analytic isomorphism by this data which can be succinctly stored in a ``weighted dual graph''. This we now explain.

The \emph{weighted dual graph} of a good resolution is a graph each vertex of which represents a component of the exceptional divisor, weighted by
self-intersection. Two vertices are connected if the corresponding components intersect.

In our case, the weighted dual graph of the minimal good resolution of $L^{\times}/G$ is represented by a \emph{star}, that is, a connected tree where at most
one vertex is connected to no more than two other vertices. $C$ itself is contained in the exceptional set of the minimal good resolution,
hence one of the vertices of the star will represent $C$. We call $C$ the \emph{central curve}. The
connected components of the graph minus the central curve are called the \emph{branches} of the graph and are indexed by
$i$, $1\leq i \leq n$. The curves of the $i$-th branch are denoted by $C_{ij}$, $1\leq j\leq r_{i}$, where $C_{i1}$ intersects $C$ and $C_{ij}$
intersects $C_{i,\,j+1}$. Let $b=-C.C$ and $b_{ij}=-C_{ij}.C_{ij}$. Then $b_{ij}\geq 2$ and $b\geq 1$. Finally, set
$$\frac{d_{i}}{e_{i}}=b_{i1}-\cfrac{1}{b_{i2}-\cfrac{1}{{\ddots}\cfrac{1}{b_{ir_{i}}}}}$$
with $e_{i}<d_{i}$ and $e_{i}$ and $d_{i}$ relatively prime. Then one has:
\begin{theorem}[{\cite[Theorem 2.1]{pinkhamm}}]\label{keylargo}
The singularity $L^{\times}/G$ is determined up to analytic isomorphism by the following data:
\begin{enumerate}
  \item The weighted dual graph of the minimal good resolution.
  \item The analytic type of the central curve $C$ (of genus $g>0$).
  \item The conormal bundle of $C$ in the resolution.
  \item The $n$ points $P_{i}=C\cap C_{i1}$ on $C$.
\end{enumerate}
Conversely, given any data as above, there exists a unique singularity of the form
$L^{\times}/G$ having this data, provided that the intersection matrix given by the graph in (i) is
negative-definite; this condition can be written as $$b-\sum_{i=1}^{n}\frac{e_{i}}{d_{i}}>0.$$
\end{theorem}
\noindent Indeed, the algorithm that recovers the algebraic equations cutting out $L^{\times}/G$ is laid out in \cite[Section 5]{pinkhamm}.
However, it does not identify the group $G$.

For the K\"ahler cones of Theorem \ref{classs}(iii),
the minimal good resolution does not contain any $(-1)$-curves, hence it coincides with the minimal model of the singularity.
Moreover, since the central curve has trivial or negative anti-canonical bundle, adjunction tells us that the canonical model is obtained
from the minimal good resolution by further contracting all of its $(-2)$-curves. However,
the result of this will be singular unless the minimal good resolution does not contain any $(-2)$-curves. Thus, the canonical model will
be smooth and coincide with the minimal good resolution if the minimal good resolution does not contain any $(-2)$-curves.
Conversely, if the canonical model of the singularity is smooth, then, since it cannot contain any $(-1)$-curves, it coincides with the minimal model which itself coincides with the minimal good resolution for the cones in question so that the minimal good resolution does not contain any $(-2)$-curves
since the canonical model cannot contain any $(-2)$-curves. Combining this observation with Theorem \ref{keylargo},
we are able to characterise those cones of Theorem \ref{classs}(iii) that admit a smooth canonical model.
\begin{prop}\label{sauce}
A K\"ahler cone of Theorem \ref{classs}(iii) admits a smooth canonical model if and only if the minimal good resolution does not contain any $(-2)$-curves.
These cones are in one-to-one correspondence with the data (i)-(iv) listed in Theorem \ref{keylargo}
with the intersection matrix of the graph in (i) being negative-definite and with the labels $b_{ij}$ of this graph being $\geq3$.
Moreover, the canonical model and the minimal resolution of such a cone are given by the minimal good resolution.
\end{prop}

\newpage
\subsection{Ricci solitons}
The metrics we are interested in are the following.
\begin{definition}
A \emph{Ricci soliton} is a triple $(M,\,g,\,X)$, where $M$ is a Riemannian manifold with a complete Riemannian metric $g$
and a vector field $X$ satisfying the equation
\begin{equation}\label{soliton2}
\Ric(g)+\frac{1}{2}\mathcal{L}_{X}g=\frac{\lambda}{2}g
\end{equation}
for some $\lambda\in\{-1,\,0,\,1\}$.  We call $X$ the \emph{soliton vector field} and say that $(M,\,g,\,X)$
 is a \emph{gradient} Ricci soliton if $X=\nabla^{g} f$ for some real-valued smooth function $f$ on $M$.
  In this latter case, equation \eqref{soliton2} reduces to
\begin{equation}\label{soliton1}
\Ric(g)+\operatorname{Hess}_{g}(f)=\frac{\lambda}{2}g,
\end{equation}
where $\operatorname{Hess}_{g}$ denotes the Hessian with respect to $g$.

If $g$ is complete and K\"ahler with K\"ahler form $\omega$, then we say that $(M,\,g,\,X)$ is a \emph{gradient K\"ahler-Ricci soliton} if
$X=\nabla^{g} f$ for some real-valued smooth function $f$ on $M$, $X$ is complete and real holomorphic, and
\begin{equation}\label{krseqn}
\rho_{\omega}+i\partial\bar{\partial}f=\lambda\omega,
\end{equation}
where $\rho_{\omega}$ is the Ricci form of $\omega$ and $\lambda$ is as above. For gradient Ricci solitons and
gradient K\"ahler-Ricci solitons, the function $f$ satisfying $X=\nabla^g f$ is called the \emph{soliton potential}.

Finally, a Ricci soliton and a K\"ahler-Ricci soliton are said to be \emph{expanding} if $\lambda=-1$ and \emph{shrinking} if $\lambda=1$
in equations \eqref{soliton2} and \eqref{krseqn} respectively.

\end{definition}
Note that for a gradient K\"ahler-Ricci soliton $(M,\,g,\,X)$ with complex structure $J$, the vector field $JX$ is Killing
by \cite[Lemma 2.3.8]{fut2}. We also have the following asymptotics on the soliton potential of a complete expanding gradient Ricci soliton with quadratic Ricci curvature decay.
\begin{prop}\label{pot-fct-est}
Let $(M^n,g,\nabla^g f)$ be a complete expanding gradient Ricci soliton of real dimension $n$, i.e., $2\Ric(g)-\mathcal{L}_{\nabla^g f}(g)=-g$.
If $\Ric(g)=O(d_g(p,\,\cdot)^{-2})$, where $d_g(p,\,\cdot)$ denotes the distance to a fixed point $p\in M$,
then the function $(-f)$ is equivalent to $d_g(p,\,\cdot)^2/4$ as $d_g(p,\,\cdot)$ tends to $+\infty$.
\end{prop}

\begin{proof}
See \cite{Che-Der} or {\cite[Lemma $4.2.1$] {Siepmann}}.
\end{proof}

Because of Proposition \ref{pot-fct-est}, we prefer to deal with an asymptotically positive soliton potential. Henceforth, an \textit{expanding} gradient Ricci soliton will be a triple $(M,\,g,\,X)$, where $X=\nabla^gf$ for some real-valued smooth function $f$ on $M$, such that the equation
\begin{equation}
\begin{split}
2\Ric(g)-\mathcal{L}_{X}g=-g\label{exp-sol-equ-riem}
\end{split}
\end{equation}
is satisfied. When the Ricci curvature of $g$ decays quadratically, the bound of Proposition \ref{pot-fct-est} on $f$ may then be given as
\begin{equation}\label{boundd}
\frac{d^{2}_{g}(p,\,x)}{4}-c_{1}d_{g}(p,\,x)-c_{2}\leq f(x)\leq\frac{d^{2}_{g}(p,\,x)}{4}+c_{1}d_{g}(p,\,x)+c_{2},
\end{equation}
where $p\in M$ is fixed and $c_{1}$ and $c_{2}$ are positive constants depending on $p$. In particular, $f$ is proper
under the assumption of quadratic Ricci curvature decay on the expanding soliton metric.

In the case of a shrinking gradient Ricci soliton, the quadratic growth of the soliton potential is always satisfied without further conditions on the decay of the Ricci tensor at infinity. More precisely, one has the following.
\begin{theorem}\label{theo-basic-prop-shrink}
Let $(M,\,g,\,X)$ be a complete non-compact shrinking gradient Ricci soliton satisfying \eqref{soliton2} with
$\lambda=1$ with soliton vector field $X=\nabla^{g}f$ for a smooth real-valued function $f:M\to\mathbb{R}$.
Then the following properties hold true.
\begin{enumerate}
\item \textnormal{(Growth of the soliton potential {\cite[Theorem 1.1]{caoo}}).}~For $x\in M$, $f$ satisfies the estimates
$$\frac{1}{4}(d_{g}(p,\,x)-c_{1})^{2}\leq f(x)\leq\frac{1}{4}(d_{g}(p,\,x)+c_{2})^{2},$$
where $d_{g}(p,\,\cdot)$ denotes the distance to a fixed point $p\in M$ with respect to $g$.
Here, $c_{1}$ and $c_{2}$ are positive constants depending only on the real dimension of $M$ and
the geometry of $g$ on the unit ball $B_{p}(1)$ based at $p$.\\

\item \textnormal{(Polynomial volume growth \cite[Theorem 1.2]{caoo}).}~For each $x\in M$, there exists a positive constant $C>0$
such that $$\operatorname{vol}_{g}(B_{r}(x))\leq Cr^{n}\quad\textrm{for $r > 0$ sufficiently large},$$
where $n=\operatorname{dim}_{\mathbb{R}}M$.\\

\item \textnormal{(Regularity at infinity).}~If the curvature tensor decays quadratically, i.e., if $A_0(g)<+\infty$, then the soliton metric has quadratic curvature decay with derivatives, i.e., $A_k(g)<+\infty$ for all $k\in \N$.
\end{enumerate}
\end{theorem}

\begin{proof}
References for items (i) and (ii) have been provided above. Item (iii)
concerning the covariant derivatives of the curvature tensor follows from Shi's estimates for ancient solutions of the Ricci flow; see {\cite[Section $2.2.3$]{wangl}} for a proof.
\end{proof}

\begin{remark}
The regularity at infinity stated in Theorem \ref{theo-basic-prop-shrink}(iii) does not hold for expanding gradient Ricci solitons; see \cite{Der-Smo-Pos-Cur-Con} for examples of expanding gradient Ricci solitons coming out of metric cones with a finite amount of regularity at infinity.
\end{remark}

The next lemma collects together some well-known Ricci soliton identities concerning shrinking gradient K\"ahler-Ricci solitons that we require.
\begin{lemma}[Ricci soliton identities]\label{solitonid}
Let $(M,\,g,\,X)$ be a shrinking gradient K\"ahler-Ricci soliton of complex dimension $n$ satisfying \eqref{krseqn} with $\lambda=1$
with soliton vector field $X=\nabla^{g}f$ for a smooth real-valued function $f:M\to\mathbb{R}$. Then the trace and first order soliton identities are:
\begin{equation*}
\begin{split}
&\Delta_{\omega} f +\frac{R_{g}}{2}=n,\\
&\nabla^g R_{g}-2 \Ric(g)(X)=0, \\
&|\nabla^g f|^2+R_{g}-2f=\operatorname{const.},\\
\end{split}
\end{equation*}
where $R_{g}$ denotes the scalar curvature of $g$ and $|\nabla^g f|^2:=g^{ij}\partial_if\partial_{j}f$.
\end{lemma}

\begin{remark}\label{rk-sol-id}
We henceforth normalize the soliton potential $f$ of a shrinking gradient K\"ahler-Ricci soliton of complex dimension $n$ satisfying \eqref{krseqn} with $\lambda=1$ so that $|\nabla^g f|^2+R_{g}-2f=2n$. The choice of constant $2n$ is dictated by the following equation satisfied by $f$:
\begin{equation*}
\Delta_{\omega}f-\frac{X}{2}\cdot f=-f.
\end{equation*}
This choice of constant also implies that $f+n$ is non-negative on $M$ since the scalar curvature $R_g$ of $g$ is necessarily non-negative.
\end{remark}

\begin{proof}
The proof of Lemma \ref{solitonid} is classic. Compose \eqref{krseqn} with $\lambda=1$ with the complex structure of $M$
in the first argument and trace the resulting identity to obtain
\begin{eqnarray}
R_g+\Delta_gf=2n,\label{riem-trace-id}
\end{eqnarray}
where $\Delta_g$ denotes the Riemannian Laplacian with respect to $g$ acting on functions. Recalling that $2\Delta_{\omega}=\Delta_g$, the first identity then follows by dividing (\ref{riem-trace-id}) across by $2$.

Next, take the divergence of \eqref{krseqn} with $\lambda=1$ to obtain
\begin{eqnarray*}
0&=&\Div_g\left(\Ric(g)+\frac{1}{2}\mathcal{L}_{\nabla^gf}(g)\right)\\
&=&\frac{\nabla^gR_g}{2}+\frac{1}{2}\left(\frac{1}{2}\nabla^g\left(\tr_g\mathcal{L}_{\nabla^gf}(g)\right)+\Delta_g\nabla^gf+\Ric(g)(\nabla^gf)\right)\\
&=&\frac{\nabla^gR_g}{2}+\frac{1}{2}\left(\nabla^g\Delta_gf+\Delta_g\nabla^gf+\Ric(g)(\nabla^gf)\right)\\
&=&\frac{\nabla^gR_g}{2}+\nabla^g\Delta_gf+\Ric(g)(\nabla^gf)\\
&=&-\frac{\nabla^gR_g}{2}+\Ric(g)(\nabla^gf),
\end{eqnarray*}
where we have used the trace version of the Bianchi identity in the first line together with the Bochner formula on functions in the third line and the trace identity (\ref{riem-trace-id}) in the last line. This proves the second identity.

Finally, combining \eqref{krseqn} with $\lambda=1$ with the previous identity, we obtain
\begin{eqnarray*}
0=\nabla^g R_{g}-2 \Ric(g)(\nabla^gf)&=&\nabla^gR_g-2\nabla^gf+\mathcal{L}_{\nabla^gf}(g)(\nabla^gf)\\
&=&\nabla^g\left(R_g-2f+|\nabla^gf|^2_g\right).
\end{eqnarray*}
Since $M$ is connected, the function $R_g-2f+|\nabla^gf|^2_g$ is constant on $M$. This verifies the third identity.
\end{proof}

K\"ahler cones are quasi-projective. This property is inherited by complete expanding and shrinking gradient K\"ahler-Ricci solitons
 on resolutions of K\"ahler cones.
\begin{prop}\label{quasi2}
Let $(M,\,g,\,X)$ be a complete expanding or shrinking gradient K\"ahler-Ricci soliton on a resolution $\pi:M\to C_{0}$
of a K\"ahler cone $C_{0}$. Then $M$ is quasi-projective.
\end{prop}

\begin{proof}
We prove this proposition in the case that $(M,\,g,\,X)$ is an expanding gradient K\"ahler-Ricci soliton. The proof for the shrinking case
is similar.

As explained in the proof of Lemma \ref{quasi}, by adding an appropriate ample divisor $D$ to $C_{0}$ at infinity, we obtain a projective compactification $\overline{C_{0}}$ of $C_{0}$ so that $\overline{C_{0}}$ will have at worst orbifold singularities along $D$.
Using $D$, we then compactify $M$ at infinity to obtain a compact complex orbifold $\overline{M}$ such that $M=\overline{M}\setminus D$.
We claim that $\overline{M}$ admits an ample line bundle, hence is projective.

Indeed, since the normal orbibundle of $D$ in $\overline{C_{0}}$ is positive, the normal orbibundle of $D$ in $\overline{M}$ will also be positive, hence by the proof of \cite[Lemma 2.3]{Conlon2}, we may endow the line orbibundle $[D]$ on $\overline{M}$ with a non-negatively curved hermitian metric with strictly positive curvature on some tubular neighbourhood $U$ of $D$ in $\overline{M}$. Next note that the curvature $h$ of the hermitian metric induced on $K_{M}$ by the expanding gradient K\"ahler-Ricci soliton metric $g$ is $-\rho_{\omega}$, where $\rho_{\omega}$ is the Ricci form of the K\"ahler form $\omega$ associated to $g$. Let $f$ denote the soliton potential so that $X=\nabla^{g}f$. Then by virtue of the expanding soliton
equation, the curvature of the hermitian metric $e^{f}h$ on $K_{M}$ is precisely the K\"ahler form $\omega$ of $g$.
In particular, the curvature of $e^{f}h$ on $K_{M}$ is a positive form. Extend the hermitian metric $e^{f}h$ on $K_{M}$ to a hermitian metric on $K_{\overline{M}}$ by amalgamating $e^{f}h$ with an arbitrary hermitian metric on $K_{\overline{M}}|_{U}$ using an appropriate bump function
supported on $U$. Then the line orbibundle $K_{\overline{M}}+p[D]$ will be ample for $p$ sufficiently large. A high tensor power
of the resulting line orbibundle will then be an ample line bundle on $\overline{M}$ so that $\overline{M}$ is projective and $M$ is quasi-projective, as claimed.
\end{proof}

Finally, note that to each complete gradient Ricci soliton, one can associate a Ricci flow that evolves via diffeomorphisms and scaling.
We describe this picture for an expanding gradient Ricci soliton next.

For a complete expanding gradient Ricci soliton $(M,\,g,\,X)$ with soliton potential $f$, set $$g(t):=t\varphi_{t}^{*}g,\,t>0,$$
where $\varphi_{t}$ is a family of diffeomorphisms generated by the gradient vector field $-\frac{1}{t}X$ with $\varphi_{1}=\operatorname{id}$, i.e.,
\begin{equation}\label{flowbaby}
\frac{\partial\varphi_{t}}{\partial t}(x)=-\frac{\nabla^g f(\varphi_{t}(x))}{t},\quad\varphi_{1}=\operatorname{id}.
\end{equation}
Then $\partial_t g(t)=-2\operatorname{Ric}(g(t))$ for $t>0$, $g(1)=g$, and defining $f(t)=\varphi_{t}^{*}f$ so that $f(1)=f$, $g(t)$ satisfies
\begin{equation}\label{rfsoliton}
\operatorname{Ric}(g(t))-\operatorname{Hess}_{g(t)}f(t)+\frac{g(t)}{2t}=0\quad\textrm{for all $t>0$.}
\end{equation}
Taking the divergence of this equation and using the Bianchi identity yields
\begin{equation}\label{div}
R_{g(t)}+|\nabla^{g(t)}f(t)|^{2}_{g(t)}-\frac{f(t)}{t}=\frac{C_{1}}{t}
\end{equation}
for some constant $C_{1}$, where $R_{g(t)}$ denotes the scalar curvature of $g(t)$.

Similarly, for a complete expanding gradient K\"ahler-Ricci soliton with K\"ahler form $\omega$, one obtains a solution of the
K\"ahler-Ricci flow $\partial_t \omega(t)=-\rho_{\omega(t)}$, where $\rho_{\omega(t)}$ denotes the Ricci form
of $\omega(t)$. The difference in normalisations between \eqref{soliton1} and \eqref{krseqn} is accounted for by the fact that the constant
preceding the Ricci term in the Ricci flow is $-2$ and that preceding the Ricci term in the
K\"ahler-Ricci flow is $-1$. In the same way that a K\"ahler-Ricci flow yields a solution of the Ricci flow and vice-versa,
a solution of the Ricci flow which is K\"ahler yields a solution of the K\"ahler-Ricci flow, the same holds
true for gradient Ricci solitons and gradient K\"ahler-Ricci soliton. Indeed, a
 solution $(M,\,g,\,X)$ of \eqref{krseqn} yields a solution of \eqref{soliton1} by replacing $g$ with $2g$
 and composing \eqref{krseqn} with the complex structure in the first arguments.
 Conversely, a solution $(M,\,g,\,X)$ of \eqref{soliton1} for which $g$ is K\"ahler and $X$ is
  real holomorphic defines a solution of $\eqref{krseqn}$ after replacing $g$ with $\frac{g}{2}$ and composing \eqref{soliton1} with the complex structure in the first arguments.

\subsection{Properties of the soliton vector field}
In this subsection, we provide sufficient conditions for which the zero set of the soliton vector field
of a complete shrinking gradient K\"ahler-Ricci soliton is compact. We begin with the following simple observation.
\begin{lemma}\label{compactt}
Let $(M,\,g,\,X)$ be a complete shrinking gradient K\"ahler-Ricci soliton with bounded scalar curvature.
Then the zero set of $X$ is compact.
\end{lemma}

\begin{proof}
With $f$ denoting the soliton potential, the boundedness of the scalar curvature $R_{g}$ of $g$ together with the properness of $f$ as a consequence of Theorem \ref{theo-basic-prop-shrink}(i) imply that $2f-R_{g}$ is proper. From the soliton identity $|\nabla^g f|^2+R_g=2f$ (Lemma \ref{solitonid}),
we then see that the function $|\nabla^{g} f|^{2}$ is proper. The compactness of the zero set of $X$ is now immediate.
\end{proof}

In the case that $M$ is in addition ``$1$-convex'', meaning
that $M$ carries a plurisubharmonic exhaustion function which is strictly plurisubharmonic
outside of a compact set, we can be more precise. Since a $1$-convex space is in particular holomorphically convex,
$M$ in this case will admit a ``Remmert reduction'' $p:M\to M'$ \cite{Grau:62}, i.e., a proper holomorphic map $p:M\to M'$
onto a normal Stein space $M'$ with finitely many isolated singularities obtained by contracting the maximal compact analytic subset $E$ of $M$.
As a Stein space with only finitely many isolated singularities, \cite[Theorem 3.1]{SCV6} asserts that $M'$ admits an embedding
$h:M'\to\mathbb{C}^{P}$ into $\mathbb{C}^{P}$ for some $P\in\mathbb{N}$. We have:
\begin{prop}\label{sexxy}
Let $(M,\,g,\,X)$ be a complete shrinking gradient K\"ahler-Ricci soliton of complex dimension $n$
with bounded scalar curvature. Assume that $M$ is $1$-convex with maximal compact analytic subset $E$. Then the zero set of $X$ is compact and:
\begin{enumerate}[label=\textnormal{(\roman{*})}, ref=(\roman{*})]
\item if $E=\emptyset$, then the zero set of $X$ comprises a single point and $M$ is biholomorphic to $\mathbb{C}^{n}$, or
\item if $E\neq\emptyset$, then the zero set of $X$ is contained in $E$.
\end{enumerate}
\end{prop}

Before we prove this proposition, an auxiliary result is required that will be used several times throughout.
\begin{prop}\label{alix}
Let $(N,\,g)$ be a complete Riemannian manifold and let $u:N\rightarrow\mathbb{R}$ be a $C^2$ function that is proper and bounded below. Assume that the flow $\phi_{x}(t)$ of $\nabla^g u$ with $\phi_{x}(0)=x\in N$ exists for all $t\in(-\infty,\,0]$. Then, for any $x\in N$, the orbit $(\phi_{x}(t))_{t\,\leq\,0}$ accumulates in the critical set of $u$, i.e., for all sequences $(t_i)_i$ diverging to $-\infty$, there exists a subsequence $(t'_i)_i$ such that $(\phi_{t'_i}(x))_i$ converges to a point $x_{\infty}\in N$ satisfying $\nabla^gu(x_{\infty})=0$.
\end{prop}

\begin{proof}[Proof of Proposition \ref{alix}]
Let $x\in N$ and let $(\phi_{x}(t))_{t\,\leq\,0}$ denote the flow of $\nabla^g u$ passing through $x$ at $t=0$ and defined for all non-negative times. Since $\partial_t\phi_{x}(t)=\nabla^g u(\phi_{x}(t))$, the function $t\in(-\infty,0]\mapsto u(\phi_{x}(t))\in\mathbb{R}$ is a non-decreasing function and for all non-positive times $t$,
\begin{eqnarray}
u(x)-u(\phi_{x}(t))=\int_t^0|\nabla^gu|^2_g(\phi_{x}(\tau))\,d\tau\geq 0.\label{finite-ene-orbit}
\end{eqnarray}
In particular, the orbit $(\phi_{x}(t))_{t\,\leq\, 0}$ lies in the sub-level set $\{y\in M:u(y)\leq u(x)\}$ of $u$ which is compact since $u$ is proper and bounded below. Moreover, since $u(\phi_{x}(t))$ is bounded from below, the estimate (\ref{finite-ene-orbit}) implies that the function $\tau\in(-\infty,0]\mapsto |\nabla^gu|^2(\phi_{x}(\tau))\in\mathbb{R}$ is integrable on $(-\infty,0]$; that is,
\begin{eqnarray}
\int_{-\infty}^0|\nabla^gu|^2(\phi_{x}(\tau))\,d\tau<+\infty.\label{finiteness-energy}
\end{eqnarray}

Now, since $u$ is $C^2$ and the orbit $(\phi_{x}(t))_{t\,\leq\,0}$ lies in a compact subset of $M$, the function $\tau\in(-\infty,0]\mapsto |\nabla^gu|^2(\phi_{x}(\tau))\in\mathbb{R}$ is Lipschitz, i.e., there is a positive constant $C$ such that $$\left||\nabla^gu|^2(\phi_{x}(t))-|\nabla^gu|^2(\phi_{x}(s))\right|\leq C|t-s|\quad\textrm{for all $s$ and $t$ in $(-\infty,0]$.}$$ This fact, together with (\ref{finiteness-energy}), implies that $\lim_{\tau\to-\infty}|\nabla^gu|(\phi_{x}(\tau))=0.$
This allows us to conclude that any accumulation point of $(\phi_{x}(t))_{t\,\leq\,0}$ lies in the critical set of $u$.
\end{proof}

We now provide the proof of Proposition \ref{sexxy}.
\begin{proof}[Proof of Proposition \ref{sexxy}]
Let $f$ denote the soliton potential and let $M_{0}(X)$ denote the zero set of $X$, a set which is compact by Lemma \ref{compactt}.
Our first claim encapsulates the structure of $M_{0}(X)$.
\begin{claim}\label{properties}
Each connected component of $M_{0}(X)$ is a smooth compact complex submanifold of $M$ contained in a level set of $f$.
\end{claim}

\begin{proof}[Proof of Claim \ref{properties}]
Let $J$ denote the complex structure of $M$ and let $F$ be a connected component of $M_{0}(X)$. Then since $F$ is locally the zero set of the holomorphic
vector field $X^{1,\,0}=\frac{1}{2}(X-iJX)$, it is a complex analytic subvariety of $M$. Furthermore, as a connected component of the zero
set of the Killing vector field $JX$, it is a totally geodesic submanifold by \cite[Theorem 5.3, p.60]{kobayashi1}. Hence
$F$ is a smooth complex submanifold of $M$.

Next observe that along any geodesic $\gamma(t)$ in $F$, we have for the soliton potential $f$,
$$\frac{d}{dt}f(\gamma(t))=df(\gamma'(t))=g(X,\,\gamma'(t))=0$$
so that $f$ is constant on $F$. Consequently, $F$ is contained in a level set of $f$.
From Theorem \ref{theo-basic-prop-shrink}(i), we know that $f$ is proper
so that the level sets of $f$ are compact. Thus, as a closed subset of a compact set, $F$ is compact.
\end{proof}

Now note that, by \cite[Proof of Lemma 1]{frankel}, $f$ is a Morse-Bott function on $M$. The critical submanifolds of $f$
are precisely the connected components of $M_{0}(X)$. Since $M$ is K\"ahler,
the Morse indices (i.e., the number of negative eigenvalues of $\operatorname{Hess}(f)$) of the critical submanifolds are all even \cite{frankel}.
Write $$M_{0}(X)=M^{(0)}\cup\bigcup_{k=1}^{n}M^{(2k)},$$
where $M^{(j)}$ denotes the disjoint union of the critical submanifolds of $M_{0}(X)$ of index $j$.
As a consequence of Claim \ref{properties}, we see that each connected component of $M_{0}(X)$, being a compact complex submanifold of $M$, is either
contained in the maximal compact analytic subset $E$ of $M$ or is an isolated point contained in $M\setminus E$.
Suppose that there exists an isolated point $x\in M^{(j)}\cap(M\setminus E)$ for some $j\geq2$.

Using ideas from \cite[p.3332]{chen-soliton} in this paragraph, we see from \cite{bryant} that the holomorphic vector field $X^{1,\,0}$ is linearisable at each of its critical points, meaning in particular that there exist local holomorphic coordinates $(z_{1},\ldots,z_{n})$ centred at $x$
such that $X^{1,\,0}=\sum_{j\,=\,1}^{n}a_{j}z_{j}\partial_{z_{j}}$ with $a_{j}\in\mathbb{R}$ for all $j=1,\ldots,n$. Since $\operatorname{Hess}(f)$ has at least one negative eigenvalue at $x$, we have that $a_{i}<0$ for some $i$. Without loss of generality, we may assume that
$i=n$ so that $a_{n}<0$. Now, clearly the orbits of $JX$ on the $z_{n}$-axis are all periodic. Fix one
such orbit $\theta:S^{1}\to M$. Then we can construct a map $R:S^{1}(\simeq\mathbb{R}/T\mathbb{Z})\times\mathbb{R}\to M$
by defining $R(s,\,t)$ to be $\phi_{t}(\theta(s))$, where $\phi_{t}$
is the integral curve of the negative gradient flow of $f$ and $T$ is the period of the orbit of $\theta$. Since
$[X,\,JX]=0$, $R$ is holomorphic and by the Riemann removable singularity theorem, $R$ extends to a non-trivial holomorphic map $\bar{R}:\mathbb{C}\to M$
with $\bar{R}(0)=x$.

Since $f$ is decreasing along its negative gradient flow and is bounded from below, we see that
$f(\bar{R}(z))$ is bounded for all $z\in\mathbb{C}$. Hence, by properness of $f$, the set
$\{\bar{R}(z):z\in\mathbb{C}\}$ is contained in a compact subset of $M$.
Letting $p:M\to M'$ denote the Remmert reduction of $M$ and recalling that $M'$ admits an embedding
$h:M'\to\mathbb{C}^{P}$ into $\mathbb{C}^{P}$ for some $P\in\mathbb{N}$,
we therefore obtain a bounded non-trivial holomorphic map $h\circ p\circ\bar{R}:\mathbb{C}\to\mathbb{C}^{P}$. By Liouville's theorem, such a map is constant. This is a contradiction. Thus $M_{0}(X)\cap(M\setminus E)$, if non-empty, is contained in $M^{(0)}$.

The next claim concerns the structure of $M^{(0)}$.
\begin{claim}\label{end}
$M^{(0)}$ is a non-empty, connected, compact complex submanifold of $M$ that comprises the global minima of $f$.
\end{claim}

\begin{proof}
$M^{(0)}$ is clearly non-empty since $f$ attains a global minimum and, as a closed
subset of the compact set $M_{0}(X)$, comprises finitely many connected, compact, complex submanifolds of $M$ by Claim \ref{properties}. To see that $M^{(0)}$
comprises one connected component only, recall that the soliton vector field $X$ is complete. Then by Proposition \ref{alix}, for any
point $x\in M$, the forward orbit of the negative gradient flow of $f$ beginning at $x$ converges to a point of $M_{0}(X)$. This gives rise to a stratification of $M$, namely $M=\bigsqcup_{k\,=\,0}^{n}W^{s}(M^{(2k)}),$ where
$$W^{s}(M^{(2k)})=\{x\in M:\lim_{t\to-\infty}\phi_{x}(t)\in M^{(2k)}\},$$
$\phi_{x}:\mathbb{R}\to M$ here denoting the gradient flow of $f$ beginning at $x$. Note that
$$M^{(0)}=W^{s}(M^{(0)})=M\mathbin{\big\backslash}\bigsqcup_{k\,=\,1}^{n}W^{s}(M^{(2k)}).$$
Now, since $M_{0}(X)$ is compact, for each $k$,
$W^{s}(M^{(2k)})$ comprises finitely many connected components,
each of which is an open submanifold of $M$ of real dimension $2n-2k$
\cite[Proposition 3.2]{austin}. The complement of finitely
many submanifolds of real codimension at least two in a
connected manifold is still connected. Hence $M^{(0)}$ is connected, as desired.
Finally, since $M^{(0)}$ contains all of the local minima of $f$ and, comprising only one connected component,
is contained in a level set of $f$ by Claim \ref{properties},
it must be the set of global minima of $f$.
\end{proof}

Now, we have already established the fact that $M_{0}(X)\cap(M\setminus E)$, if non-empty, is contained in $M^{(0)}$.
Thus, if $E=\emptyset$, then, since $M_{0}(X)$ is non-empty as $f$ attains a global minimum, we
must have that $M_{0}(X)=M^{(0)}$ so that $M_{0}(X)$ is a non-empty, connected, compact complex submanifold of $M$ by Claim \ref{end}.
Since $M$ is affine if $E=\emptyset$, we deduce that $M_{0}(X)$ must comprise a single point if $E=\emptyset$.
It then follows from \cite{bryant} that $M$ is biholomorphic to $\mathbb{C}^{n}$ if $E=\emptyset$. This is case (i) of the proposition.

Next consider the case when $E\neq\emptyset$. If $M_{0}(X)\cap(M\setminus E)=\emptyset$,
then $M_{0}(X)\subseteq E$ and we are in case (ii) of the proposition. So, to derive a contradiction, suppose that $E\neq\emptyset$ and $M_{0}(X)\cap(M\setminus E)\neq\emptyset$. In light of the above, we must have that $M^{(0)}\cap E=\emptyset$ and that $M_{0}(X)\cap(M\setminus E)=M^{(0)}$ which comprises a single point $x$ say. Moreover, $\left(\bigcup_{j\,=\,1}^{n}M^{(2j)}\right)\cap E\neq\emptyset$ since otherwise $M$ would be biholomorphic to $\mathbb{C}^{n}$ by \cite{bryant}, thereby yielding a contradiction. Thus, noting that $f(M^{(0)})$ is the global minimum value of $f$ by Claim \ref{end}, let $A$ be the smallest critical value of $f$ with $A>f(M^{(0)})$ and let $y\in f^{-1}(\{A\})$. Then we must have that $y\in M^{(k)}\subseteq E$ for some $k\geq 2$ by what we have just said. As before, we can construct a holomorphic map $\bar{R}:\mathbb{C}\to M$
with $\bar{R}(0)=y$. Since $f$ is decreasing along its negative gradient flow and since there are no critical values of
$f$ in the open interval $(f(M^{(0)}),\,A)$, we see from Proposition \ref{alix}
that necessarily $\lim_{z\to+\infty}\bar{R}(z)=x$. The Riemann removable singularity theorem
then applies and allows us to extend $\bar{R}$ to a holomorphic map $\bar{R}':\mathbb{P}^{1}\to M$. Since $x\neq y$ and $x\notin E$,
what we have constructed is a non-trivial holomorphic curve in $M$ that is not contained in $E$. This contradicts the maximality of $E$.
Thus, cases (i) and (ii) of the proposition are the only two possibilities that can occur. This completes the proof.
\end{proof}

\subsection{Properties of real holomorphic vector fields commuting with the soliton vector field}

In this subsection, we mention some properties of real holomorphic vector fields that commute with the soliton vector field
on a complete shrinking gradient K\"ahler-Ricci soliton. As the next proposition demonstrates, a bound on the
Ricci curvature yields control on the growth of the norm of these vector fields.

\begin{prop}\label{keywest}
Let $(M,\,g,\,X)$ be a complete shrinking gradient K\"ahler-Ricci soliton with bounded Ricci curvature
and let $d_{g}(p,\,\cdot)$ denote the distance to a fixed point $p\in M$
with respect to $g$. Then there exists $a>0$ such that for every real holomorphic vector field $Y$ on $M$ with $[X,\,Y]=0$, $|Y|^{2}_{g}(x)=O(d_{g}(p,\,x)^{a})$.
\end{prop}

\begin{proof}
Let $|\cdot|$ denote the norm with respect to $g$ and let $\operatorname{Ric}$ denote the Ricci curvature of $g$.
Since $|\operatorname{Ric}|$ is bounded so that the scalar curvature of $g$ is bounded, it follows from Lemma \ref{compactt} that the zero set of $X$ is contained in a compact subset of $M$. For $A>0$, let $K:=f^{-1}([2A,\,4A])$ and $N=f^{-1}((-\infty,\,3A])$. Since $f$ is proper and bounded below as a consequence of Theorem \ref{theo-basic-prop-shrink}(i), $K$ and $N$ are compact subsets of $M$. Choose $A$ sufficiently large so that all of the critical points of $f$ are contained in $f^{-1}((-\infty,\,A])$ and so that $A>\sup_{M}|R_{g}|$.
Let $\gamma_{x}(t)$ denote the integral curve of $X$ with $\gamma_{x}(0)=x\in M$. We begin with the following claim.
\begin{claim}\label{hotmail}
Let $y\in M\setminus N$. Then there exists $x\in K$ and $t_{0}>0$ such that $y=\gamma_{x}(t_{0})$.
\end{claim}

\noindent That is to say, every point of $M\setminus N$ lies on an integral curve of $X$ passing through $K$.

\begin{proof}[Proof of Claim \ref{hotmail}]
For $y\in M\setminus K$, we see from the soliton identity $|\nabla^g f|^2+R_g=2f$ that
\begin{eqnarray*}
\frac{d}{dt}f(\gamma_y(t))=|\nabla^g f|^2_g(\gamma_y(t))=2f(\gamma_y(t))-R_g(\gamma_y(t)).
\end{eqnarray*}
Using the upper bound on $|R_{g}|$, we deduce that
\begin{eqnarray*}
\left|\frac{d}{dt}f(\gamma_y(t))-2f(\gamma_y(t))\right|\leq2A.
\end{eqnarray*}
Integrating this differential inequality for $t<0$ then yields the inequalities
\begin{eqnarray}\label{fmll}
(f(y)+A)e^{2t}-A\leq f(\gamma_y(t))\leq(f(y)-A)e^{2t}+A\quad\textrm{for $t<0$.}
\end{eqnarray}
Set $t_{0}=-\frac{1}{2}\ln\left(\frac{3A}{f(y)+A}\right)>0$.
Then from \eqref{fmll} we see that
$$2A\leq f(\gamma_{y}(-t_{0}))\leq3A\left(\frac{f(y)-A}{f(y)+A}\right)+A=3A\left(1-\frac{2A}{f(y)+A}\right)+A\leq4A.$$
Thus, $y=\gamma_{x}(t_{0})$ where $x=\gamma_{y}(-t_{0})\in K$. This proves the claim.
\end{proof}
\noindent Next observe that
\begin{equation*}
\begin{split}
\mathcal{L}_{X}(|Y|^{2})&=(\mathcal{L}_{X}g)(Y,\,Y)=g(Y,\,Y)-\operatorname{Ric}(Y,\,Y)=|Y|^{2}-\operatorname{Ric}(Y,\,Y).
\end{split}
\end{equation*}
For $x\in M$ a point where $X\neq0$, let
$h(t):=|Y|^{2}(\gamma_x(t))$. Then we can rewrite the previous equation as
$$h'(t)=h(t)-\operatorname{Ric}(Y,\,Y)(\gamma_x(t))$$
so that
\begin{equation}\label{ode}
\frac{h'(t)}{h(t)}=1-\frac{\operatorname{Ric}(Y,\,Y)(\gamma_x(t))}{h(t)}.
\end{equation}
Analysing the error term here, we have that
$$\frac{\operatorname{Ric}(Y,\,Y)(\gamma_x(t))}{h(t)}=\frac{\operatorname{Ric}(Y,\,Y)}{|Y|^{2}}.$$
Since $|\operatorname{Ric}|$ is bounded by assumption, we then have that
$$\left|\frac{\operatorname{Ric}(Y,\,Y)}{|Y|^{2}}\right|\leq C$$
for a constant $C>0$, so that \eqref{ode} gives us the bound
$$\left|\frac{h'(t)}{h(t)}\right|\leq2a$$
for some $a>0$.
Solving this for $t>0$ yields
$$-2at\leq\ln(h(t))-\ln(h(0))\leq2at$$
so that in particular, $$h(t)\leq|Y|^{2}(x)e^{2at}\qquad\textrm{for all $t>0$}.$$
Hence,
\begin{equation}\label{salsa}
|Y|^{2}(\gamma_x(t))\leq |Y|^{2}(x)e^{2at}\quad\textrm{for all $t>0$}.
\end{equation}

Let $y\in M\setminus N$. Then by Claim \ref{hotmail}, there is an $x\in K$ and $t_{0}>0$ such that
$y=\gamma_{x}(t_{0})$. Applying the above inequality to this choice of $x$ and $t_{0}$, we deduce that
$$|Y|^{2}(y)\leq |Y|^{2}(x)e^{2at_{0}}.$$
Now, as in the proof of Claim \ref{hotmail}, we have that
\begin{eqnarray*}
\left|\frac{d}{dt}f(\gamma_x(t))-2f(\gamma_x(t))\right|\leq2A.
\end{eqnarray*}
Integrating this for $t>0$ yields the fact that
$$(f(x)-A)e^{2t}+A\leq f(\gamma_{x}(t))\leq (f(x)+A)e^{2t}-A\quad\textrm{for all $t>0$}.$$
Since $x\in K$ so that $f(x)\geq2A$, we see from the left-hand side of this inequality that
$f(\gamma_{x}(t))\geq A(1+e^{2t})$ so that
$$e^{2t}\leq\frac{f(\gamma_{x}(t))}{A}-1\quad\textrm{for all $t>0$}.$$
Plugging this into \eqref{salsa} and setting $t=t_{0}$ results in the bound
$$|Y|^{2}(y)\leq|Y|^{2}(x)e^{2at_{0}}\leq|Y|^{2}(x)\left(\frac{f(y)}{A}-1\right)^{a}\leq\left(\sup_{K}|Y|^{2}\right)\left(\frac{f(y)}{A}-1\right)^{a}.$$
Since $K$ is compact and $f$ grows quadratically with respect to the distance to a fixed point $p\in M$ by Theorem \ref{theo-basic-prop-shrink}(i),
we arrive at the estimate
\begin{eqnarray*}
|Y|(z)\leq c_1d_g(p,\,z)^{a}+c_2\quad\textrm{for all $z\in M$},
\end{eqnarray*}
for some positive constants $c_1, c_2>0$. This leads to the desired conclusion.
\end{proof}

\begin{remark}\label{jumper}
In the case that the Ricci curvature decays quadratically at infinity, the constant $a$ may be taken to be equal to $2$ in Proposition \ref{keywest}.
\end{remark}

We can also show that such vector fields are complete when the zero set of the soliton vector field is compact.
\begin{lemma}\label{completee}
Let $(M,\,g,\,X)$ be a complete shrinking gradient K\"ahler-Ricci soliton. Assume that
the zero set of $X$ is compact (which, by Lemma \ref{compactt}, is the case when the scalar curvature of $g$ is bounded).
Then every real holomorphic vector field $Y$ on $M$ with $[X,\,Y]=0$ is complete.
\end{lemma}

\begin{proof}
Let $Y$ be as in the statement of the lemma and let $K$ be any compact subset of $M$ containing
the zero set of $X$ in its interior. Then note the following.
\begin{itemize}
  \item Since $K$ is compact, there exists $\epsilon_{0}>0$ such that
the flow of $Y$ beginning at any point of $K$ exists on the open interval $(-\epsilon_{0},\,\epsilon_{0})$.
  \item By Proposition \ref{alix}, for any $p\in M$,
there exists $T(p)>0$ such that the image of $p$ under the forward flow of $-X$ for time $T$ will be contained in $K$.
\end{itemize}
Consequently, for any point $p\in M$, by flowing first along $-X$ into $K$ for time $T(p)$,
then flowing along $Y$, then flowing along $X$ for time $T$,
one sees from the fact that $[X,\,Y]=0$
that the flow of $Y$ beginning at any point of $M$ exists on the interval $(-\epsilon_{0},\,\epsilon_{0})$.
This observation suffices to prove the completeness of $Y$.
\end{proof}

\subsection{Basics of metric measure spaces}\label{metricmeasure}

We take the following from \cite{fut}; the notions introduced in this section will be used in Section \ref{matsushimaa}.

A smooth metric measure space is a Riemannian manifold endowed with a weighted volume.
\begin{definition}
A \emph{smooth metric measure space} is a triple $(M,\,g,\,e^{-f}dV_{g})$, where $(M,\,g)$ is a complete Riemannian manifold with Riemannian metric $g$,
$dV_{g}$ is the volume form associated to $g$, and $f:M\to\mathbb{R}$ is a smooth real-valued function.
\end{definition}
\noindent A shrinking gradient Ricci soliton $(M,\,g,\,X)$ with $X=\nabla^g f$ naturally defines a smooth metric measure space $(M,\,g,\,e^{-f}dV_{g})$.
On such a space, we define the weighted Laplacian $\Delta_{f}$ by
$$\Delta_{f}u:=\Delta u-g(\nabla^g f,\,\nabla u)$$
on smooth real-valued functions $u\in C^{\infty}(M,\,\mathbb{R})$. There is a natural $L^{2}$-inner product $\langle\cdot\,,\,\cdot\rangle_{L^{2}_{f}}$ on the space $L^{2}_{f}$ of square-integrable smooth real-valued functions on $M$
with respect to the measure $e^{-f}dV_g$ defined by $$\langle u,\,v\rangle_{L_{f}^{2}}:=\int_{M}uv\,e^{-f}dV_{g},\qquad u,\,v\in L_{f}^{2}.$$
As one can easily verify, the operator $\Delta_{f}$ is self-adjoint with respect to $\langle\cdot\,,\,\cdot\rangle_{L_{f}^{2}}$.

In the K\"ahler case, we have:
\begin{definition}
If $(M,\,g,\,e^{-f}dV_{g})$ is a smooth metric measure space and $(M,\,g)$ is K\"ahler, then we say that $(M,\,g,\,e^{-f}dV_{g})$ is a \emph{K\"ahler metric measure space}.
\end{definition}
\noindent A shrinking gradient K\"ahler-Ricci soliton naturally defines such a space.

Unlike the real case, on a K\"ahler metric measure space we have the weighted $\bar{\partial}$-Laplacian $\Delta_{f}$ defined on smooth complex-valued functions $u\in C^{\infty}(M,\,\mathbb{C})$ by
$$\Delta_{f}u:=\Delta_{\bar{\partial}}u-(\nabla^{1,\,0}u)f=g^{i\bar{\jmath}}\partial_{i\bar{\jmath}}u-g^{i\bar{\jmath}}(\partial_{i}f)(\partial_{\bar{\jmath}}u).$$
This may be a complex-valued function even if $u$ is real-valued.
We define a hermitian inner product on the space $C_0^{\infty}(M,\,\mathbb{C})$ by
$$\langle u,\,v\rangle_{L^{2}_f}:=\int_{M}u\bar{v}\,e^{-f}dV_{g},\qquad u,\,v\in C_0^{\infty}(M,\,\mathbb{C}).$$
Then $\Delta_{f}$ is symmetric with respect to this inner product. In fact, we have that
$$\int_{M}(\Delta_{f}u)\bar{v}\,e^{-f}dV_{g}=\int_{M}u\overline{\Delta_{f}v}\,e^{-f}dV_{g}=
-\int_{M}g(\overline{\partial} u,\overline{\partial} v)\,e^{-f}dV_{g}=-\langle\overline{\partial}u,\overline{\partial}v\rangle_{L_{f}^2}.$$
See \cite{fut} and the references therein for further details.

\section{Proof of Theorem \ref{main}}\label{proofofA}

We first consider Theorem \ref{main} in the expanding case.

\subsection{Construction of a map to the tangent cone}
By a result of Siepmann \cite[Theorem 4.3.1]{Siepmann}, a complete expanding gradient Ricci soliton $(M,\,g,\,X)$ with quadratic curvature decay with derivatives has a unique tangent cone along each end. We first prove a series of lemmas before providing a refinement of Siepmann's result in Theorem \ref{ac} by using the flow of the soliton vector field $X$ to construct a diffeomorphism between each end of the expanding Ricci soliton and its tangent cone $(C_{0},\,g_{0})$ along that end with respect to which $r\partial_{r}$ pushes forward to $2X$, $r$ here denoting the radial coordinate
of $g_{0}$, and with respect to which $g-g_{0}-\operatorname{Ric}(g_{0})=O(r^{-4})$ with derivatives.

Our set-up in this section is as follows.\\
\noindent\fbox{%
    \parbox{\textwidth}{%
$(M,\,g,\,X)$ is a complete expanding gradient Ricci soliton with soliton vector field $X=\nabla^{g}f$ for a smooth real-valued
function $f:M\to\mathbb{R}$ such that for some point $p\in M$ and all $k\in\mathbb{N}_{0}$,
\begin{equation*}
A_{k}(g):=\sup_{x\in M}|(\nabla^{g})^{k}\operatorname{Rm}(g)|_{g}(x)d_{g}(p,\,x)^{2+k}<\infty,
\end{equation*}
where $\operatorname{Rm}(g)$ denotes the curvature of $g$ and $d_{g}(p,\,x)$ denotes the distance between $p$ and $x$ with respect to $g$.}}
\\
The diffeomorphisms $(\varphi_t)_{t\,\in\,(0,\,1]}$ will be as in \eqref{flowbaby}. We begin with:
\begin{lemma}\label{lemma-def-q}
The one-parameter family of functions $(tf\circ\varphi_{t})_{t>0}$ converges to a non-negative
continuous real-valued function $q(x):=\lim_{t\to0^{+}}tf(\varphi_{t}(x))$ on $M$ as $t\to 0^{+}$.
\end{lemma}

\begin{proof}
Since
$$\frac{\partial}{\partial t}f(\varphi_{t}(x))=-\frac{|\nabla^{g}f|^{2}_{g}(\varphi_{t}(x))}{t}$$
so that
$$\frac{\partial}{\partial t}(tf(\varphi_{t}(x)))=(f-|\nabla^{g}f|^{2}_{g})(\varphi_{t}(x))=R_{g}(\varphi_{t}(x))$$
by the soliton identities for expanding gradient Ricci solitons satisfying \eqref{exp-sol-equ-riem},
where $R_{g}$ denotes the scalar curvature of $g$, we see after integrating that
\begin{equation}
tf(\varphi_t(x))-sf(\varphi_s(x))=\int_s^t R_{g}(\varphi_{\tau}(x))\,d\tau,\quad 0<s\leq t.\label{var-pot-fct-resc}
\end{equation}
Since $R_g$ is bounded on $M$, it follows that for all $x\in M$,
\begin{equation}\label{var-pot-fct2}
|tf(\varphi_t(x))-sf(\varphi_s(x))|\leq C(t-s),\quad 0<s\leq t,
\end{equation}
for some positive constant $C$. Thus, $\{t(f\circ\varphi_{t})\}_{t\in(0,\,1]}$ is a Cauchy sequence in
$C^{0}(M)$ and hence converges uniformly as $t\to0^{+}$ to a continuous real-valued function
$q$ on $M$ as in the statement of the lemma.

To see that $q(x)\geq0$ for all $x\in M$, note from the soliton identities that
$$tf(\varphi_t(x))=t|X|^2(\varphi_t(x))+tR_g(\varphi_t(x))\geq tR_g(\varphi_t(x))\geq t\inf_{M} R_g$$
since $t\in(0,1]$ and $R_{g}$ is bounded from below. Letting $t\to0^{+}$ in this inequality yields the desired conclusion.
\end{proof}

For $a>0$, set
$$M_{a}:=\left\{x\in M\,|\,\liminf_{t\to 0^{+}}tf(\varphi_t(x))>a\right\}$$
and
$$M_{0}:=\left\{x\in M\,|\,\liminf_{t\to 0^{+}}tf(\varphi_t(x))>0\right\}.$$
Note that
\begin{itemize}
        \item $M_{a}$ and $M_{0}$ are preserved by $\varphi_{t}$ for all $t\in(0,\,1]$ and $a>0$,
  since $\varphi_{ts}=\varphi_t\circ\varphi_s$ for all positive times $s$ and $t$,
      \end{itemize}
so that for any $a\geq0$, $t(f\circ\varphi_{t})$ defines a family of smooth functions
$t(f\circ\varphi_{t}):M_{a}\to\mathbb{R}$ for $t\in(0,\,1]$.
\begin{lemma}\label{monty}
$t(f\circ\varphi_{t})$ converges to $q$ in $C^{\infty}_{\operatorname{loc}}(M_{0})$ as $t\to0^{+}$. In particular, $q$ is smooth on $M_{0}$.
\end{lemma}

\begin{proof}
Note that for each $a>0$, \cite[Lemma 4.3.3]{Siepmann} ensures that along the Ricci flow $g(t)$ defined by $g$,
the norm of the curvature tensor $\operatorname{Rm}_{g(t)}$ of $g(t)$
is bounded with respect to $g(t)$ when restricted to $M_{a}$. In particular, there exists a positive constant $C$ (depending on $a$) such that for all $t\in(0,1]$, $$\sup_{x\in M_{a}}|\Ric(g(t))|_{g(t)}(x)\leq C.$$ By definition of a Ricci flow, this implies that the metrics $(g(t))_{t\in(0,1]}$ are uniformly equivalent, i.e., there exists a positive constant $C$ (that may vary from line to line) such that
\begin{eqnarray}
C^{-1}g(x)\leq g(t)(x)\leq Cg(x), \quad t\in(0,1],\quad x\in M_{a}.\label{unif-equiv-metrics-end}
\end{eqnarray}
An induction argument (cf.~\cite[Lemma 4.3.6]{Siepmann}) then shows that for all $x\in M_{a}$, $t\in(0,\,1]$, and $k\geq0$,
$$|(\nabla^{g})^{k}(g(t))|_{g}(x)\leq C(k,\,a).$$
Similarly, one obtains that
\begin{equation}\label{qaz}
|(\nabla^{g})^{k}(\operatorname{Rm}_{g(t)})|_{g}(x)\leq C(k,\,a)
\end{equation}
for all $x\in M_{a}$, $t\in(0,\,1]$, and $k\geq0$. As a consequence,
\begin{eqnarray}
|(\nabla^{g})^{k}(tf(\varphi_{t}(x)))|_{g}\leq C(k,\,a)\label{unif-bd-pot-fct}
\end{eqnarray}
for all $x\in M_{a}$, $t\in(0,\,1]$, and $k\geq0$. Indeed, by (\ref{var-pot-fct-resc}) with $t=1$ and $s=t$,
\begin{eqnarray*}
f(x)-tf(\varphi_t(x))=\int_t^1R_g(\varphi_s(x))ds=\int_t^1sR_{g(s)}(x)ds,\quad t\in(0,1],\quad x\in M.
\end{eqnarray*}
In particular, by deriving $k$ times at a point $x\in M_a$, we see that
\begin{eqnarray*}
\left(\nabla^{g}\right)^k\left(t\varphi_t^*f\right)= (\nabla^{g})^kf-\int_t^1s\left(\nabla^{g}\right)^kR_{g(s)}ds,\quad t\in(0,1],
\end{eqnarray*}
which implies the desired inequality \eqref{unif-bd-pot-fct} after invoking (\ref{qaz}). As a result, $t(f\circ\varphi_{t})$ converges in $C^{\infty}_{\operatorname{loc}}(M_{0})$ as $t\to0^{+}$
so that $q$ is smooth on $M_{0}$, as claimed.
\end{proof}

Since $\varphi_{t}$ preserves $M_{a}$ for every $a>0$, we also have that for any $a\geq0$, the Ricci flow $g(t)$ determined by $g$,
namely $g(t):=t\varphi_{t}^{*}g$, defines a family of smooth metrics on $M_{a}$ for all $t\in(0,\,1]$.
This family converges in $C^{\infty}_{\operatorname{loc}}(M_{0})$ as $t\to0^{+}$ as well.
\begin{lemma}\label{metrics}
$g(t)$ converges to a Riemannian metric $\tilde{g}_{0}$ in $C^{\infty}_{\operatorname{loc}}(M_{0})$ as $t\to0^{+}$.
Moreover, $\tilde{g}_{0}=2\operatorname{Hess}_{\tilde{g}_{0}}q$.
\end{lemma}

\begin{proof}
From the definition of the Ricci flow, one deduces from the curvature bounds \eqref{qaz} that $g(t)$ is a Cauchy sequence in $C^{k}(M_{a})$ for
every $k\geq0$ and $a>0$, hence converges uniformly locally as $t\to0^{+}$ in $C^{k}(M_{0})$ for every $k\geq 0$ to a Riemannian metric $\tilde{g}_{0}$ on $M_{0}$. To see that $\tilde{g}_{0}=2\operatorname{Hess}_{\tilde{g}_{0}}q$, multiply \eqref{rfsoliton} across by $t$ and take the limit as $t\to0^{+}$,
recalling that $\lim_{t\to0^{+}}tf(t)=q$ in $C_{\operatorname{loc}}^{\infty}(M_{0})$ by Lemma \ref{monty}.
 \end{proof}

We have the following properties of $q$.
\begin{lemma}\label{lemma-q-proper}
$q$ is proper and bounded below.
\end{lemma}

\begin{proof}
Lemma \ref{lemma-def-q} already implies that $q$ is non-negative. In particular, it is bounded from below. Now, one sees from the quadratic growth of the soliton potential $f$ given by \eqref{boundd}
that for $p$ in the critical set of the soliton potential $f$,
\begin{equation*}
\frac{d^2_{g(t)}(p,x)}{4}-c_1\sqrt{t}d_{g(t)}(p,x)-c_2t\leq tf(\varphi_t(x))\leq \frac{d^2_{g(t)}(p,x)}{4}+c_1\sqrt{t}d_{g(t)}(p,x)+c_2t,\quad t>0,\quad x\in M,
\end{equation*}
for some constants $c_{1},\,c_{2}>0$, where $d_{g(t)}$ denotes the distance with respect to $g(t)$. Using this inequality and taking the limit as $s\to0^{+}$ in \eqref{var-pot-fct2}, one finds that
\begin{equation}\label{flingg}
\frac{d^2_{g(t)}(p,x)}{4}-c_1\sqrt{t}d_{g(t)}(p,x)-c_2t\leq
q(x)\leq \frac{d^2_{g(t)}(p,x)}{4}+c_1\sqrt{t}d_{g(t)}(p,x)+c_2t,\quad t>0, \quad x\in M,
\end{equation}
for some constants $c_{1},\,c_{2}>0$ that may now vary from line to line. Thus, $q(x)\to+\infty$
as $x\to\infty$ and the result follows.
\end{proof}
Moreover, we have:
\begin{lemma}\label{fav}
On $M_{0}$,
\begin{enumerate}[label=\textnormal{(\roman{*})}, ref=(\roman{*})]
\item $|\nabla^{\tilde{g}_{0}}q|^{2}_{\tilde{g}_{0}}=q$ so that the integral curves of $\nabla^{\tilde{g}_{0}}(2\sqrt{q})$ on $M_{0}$ are geodesics;
\item $\nabla^{\tilde{g}_{0}}q=X$.
\end{enumerate}
In particular, $X$ is nowhere vanishing on $M_{0}$.
\end{lemma}

\begin{proof}
To prove part (i), we multiply equation \eqref{div} across by $t^{2}$ and take the limit as $t\to0^{+}$.

As for part (ii), let $\varphi_{t}$ be as in \eqref{flowbaby}. Then on compact subsets of $M_{0}$, we have that
\begin{equation*}
\begin{split}
\nabla^{\tilde{g}_{0}}q&=\lim_{t\to0^{+}}\nabla^{g(t)}(tf(t))=\lim_{t\to0^{+}}\nabla^{t\varphi^{*}_{t}g}(t\varphi^{*}_{t}f)
=\lim_{t\to0^{+}}\frac{1}{t}\nabla^{\varphi^{*}_{t}g}(t\varphi^{*}_{t}f)
=\lim_{t\to0^{+}}\nabla^{\varphi^{*}_{t}g}(\varphi^{*}_{t}f)\\
&=\lim_{t\to0^{+}}\varphi^{*}_{t}(\nabla^{g}f)=\lim_{t\to0^{+}}\varphi_{t}^*X=\lim_{t\to0^{+}}X=X,
\end{split}
\end{equation*}
where the penultimate inequality follows from the fact that $\varphi_{t}$ is generated by the flow of $\frac{1}{t}X$ and $\mathcal{L}_{X}X=0$.
\end{proof}

The above observations then imply:
\begin{lemma}\label{ffinite}
$M$ has only finitely many ends.
\end{lemma}

\begin{proof}
For any $a>0$, $q$ is a smooth function on $M_{a}$ by Lemma \ref{monty}. Furthermore, by Lemma \ref{fav}, $q$ has no critical points in $M_{a}$.
Consequently, using the Morse flow $(\psi^{q}_{t})_{t\geq 0}$ associated to $q$, one sees that all the level sets of $q$ of the form $q^{-1}(\{b\})$ with $b\geq b_0$ for some $b_0\in\R$ large enough are diffeomorphic. Since $q$ is proper by Lemma \ref{lemma-def-q}, such level sets are compact and the map $\psi^q:(t,x)\in (0,+\infty)\times q^{-1}(\{b_0\})\mapsto \psi^q_{t}(x)\in q^{-1}((b_0,+\infty))$ is a diffeomorphism of a neighborhood of $M$ at infinity. Again, since $q$ is proper, the level set $q^{-1}(\{b_0\})$ is compact hence has a finite number of connected components. Thus, $M$ has a finite number of ends.
\end{proof}

Our final lemma is then:
\begin{lemma}\label{fckme}
There exists $A>0$ such that for all $c>A$,
the intersection of each end of $M$ with $q^{-1}(\{c\})$ is compact, connected, and non-empty.
\end{lemma}

\begin{proof}
By Lemma \ref{ffinite}, $M$ has only a finite number of ends. Thus, since $q$ is proper and bounded below
by Lemma \ref{lemma-q-proper}, there exists $A>0$ such that
all of the ends of $M$ are contained in $M\setminus q^{-1}((-\infty,\,A])=M_{A}$.
For any $c>A$, the intersection of $q^{-1}(\{c\})$ with each end of $M$ is then compact and non-empty and comprises one connected component only since, as a consequence of Lemma \ref{fav}, $q$ is strictly increasing along the flow lines of the (nowhere vanishing) vector field $X$ on $M_{0}$.
\end{proof}

Using the above lemmas, we can now construct our map to the tangent cone at infinity.
\begin{theorem}[Map to the tangent cone for expanding Ricci solitons]\label{ac}
Let $A>0$ be as in Lemma \ref{fckme}, let $\rho:M_{0}\to\mathbb{R}_{+}$ be defined by $\rho:=2\sqrt{q}$, and let $S:=\rho^{-1}(\{c\})$ for any $c>0$
with $\frac{c^{2}}{4}>A$ (so that the intersection of each end of $M$ with $S$ is compact, connected, and non-empty).
Then there exists a diffeomorphism $\iota:(c,\,\infty)\times S\to M_{\frac{c^{2}}{4}}$ such that $g_{0}:=\iota^{*}\tilde{g}_{0}=dr^{2}+r^{2}\frac{g_{S}}{c^{2}}$ and $d\iota(r\partial_{r})=2X$, where
$r$ is the coordinate on the $(c,\,\infty)$-factor and $g_{S}$ is the restriction of $\tilde{g}_{0}$ to $S$. Moreover, along $M_{\frac{c^{2}}{4}}$, we have that
\begin{equation}\label{northbeach}
|(\nabla^{g_{0}})^{k}(\iota^{*}g-g_{0}-2\operatorname{Ric}(g_{0}))|_{g_{0}}=O(r^{-4-k})\quad\textrm{for all $k\in\mathbb{N}_{0}$}.
\end{equation}
In particular, $(M,\,g)$ has a unique tangent cone along each end.
\end{theorem}

\begin{proof}
To prove the first part of this statement, we follow the proof of \cite[Theorem 1.7.2]{daii}.

We have that $\operatorname{Hess}_{\tilde{g}_{0}}(\rho^{2})=2\tilde{g}_{0}$ from Lemma \ref{metrics} and we know from Lemma \ref{fav}(i) that $|\nabla^{\tilde{g}_{0}}\rho^{2}|^2_{\tilde{g}_{0}}=4\rho^{2}$ is constant along the level sets of $\rho$ and that the integral curves of $\nabla^{\tilde{g}_{0}}\rho$ are geodesics.
Then we have that
$$\nabla^{\tilde{g}_{0}}\rho^{2}=2\rho\nabla^{\tilde{g}_{0}}\rho\quad\textrm{and}\quad\operatorname{Hess}_{\tilde{g}_{0}}
(\rho^{2})=2d\rho^{2}+2\rho\operatorname{Hess}_{\tilde{g}_{0}}(\rho)$$
so that $$2\tilde{g}_{0}=\operatorname{Hess}_{\tilde{g}_{0}}(\rho^{2})=2d\rho^{2}+2\rho\operatorname{Hess}_{\tilde{g}_{0}}(\rho).$$
Hence, $$\operatorname{Hess}_{\tilde{g}_{0}}(\rho)=\frac{\tilde{g}_{0}}{\rho}\quad\textrm{on the $\tilde{g}_{0}$-orthogonal complement of $\nabla^{\tilde{g}_{0}}\rho$.}$$
On the other hand, $\tilde{g}_{0}=d\rho^{2}+\tilde{g}_{\rho}$ with $\tilde{g}_{\rho}$ the restriction of $\tilde{g}_{0}$ to the level set of $\rho$, and
$$\mathcal{L}_{\nabla^{\tilde{g}_{0}}\rho}\tilde{g}_{\rho}=\mathcal{L}_{\nabla^{\tilde{g}_{0}}\rho}(\tilde{g}_{0}-d\rho^{2})=
2\operatorname{Hess}_{\tilde{g}_{0}}(\rho)-16\rho d\rho^{2}=\frac{2\tilde{g}_{\rho}}{\rho}+\left(\frac{2}{\rho}-16\rho\right)d\rho^{2},$$
so that $$\mathcal{L}_{\nabla^{\tilde{g}_{0}}\rho}\tilde{g}_{\rho}=\frac{2\tilde{g}_{\rho}}{\rho}\quad\textrm{on the $\tilde{g}_{0}$-orthogonal complement of $\nabla^{\tilde{g}_{0}}\rho$}.$$
Thus,
\begin{equation}\label{gogo}
\mathcal{L}_{\rho\nabla^{\tilde{g}_{0}}\rho}\tilde{g}_{\rho}=2\tilde{g}_{\rho}\quad\textrm{on the $\tilde{g}_{0}$-orthogonal complement of $\nabla^{\tilde{g}_{0}}\rho$}.
\end{equation}

Next define a map $\iota:(c,\,\infty)\times S\to M_{\frac{c^{2}}{4}}$ by
$$(r,\,x)\mapsto\Phi_{x}(r-c),$$
where $\Phi_{x}(\cdot)$ denotes the flow of $\nabla^{\tilde{g}_{0}}\rho$ with $\Phi_{x}(0)=x$.
By choice of $c$, this map is well-defined. Moreover, $d\iota(\partial_{r})=\nabla^{\tilde{g}_{0}}\rho$ by construction, and since
$\rho(\Phi_{x}(t))=t+c,$ we have that $\iota^{*}\rho=r$ so that $d\iota(r\partial_{r})=\rho\nabla^{\tilde{g}_{0}}\rho=2X$. In this new frame, we thus have that
$$\iota^{*}\tilde{g}_{0}=dr^{2}+\iota^{*}\tilde{g}_{\rho},$$ where we find from \eqref{gogo} that
\begin{equation*}
\mathcal{L}_{r\partial_{r}}\iota^{*}\tilde{g}_{\rho}=2\iota^{*}\tilde{g}_{\rho}\quad\textrm{on the $\iota^{*}\tilde{g}_{0}$-orthogonal complement of $\partial_{r}$}.
\end{equation*}
Hence, $\iota^{*}\tilde{g}_{\rho}=r^{2}\frac{g_{S}}{c^{2}}$ so that $\iota^{*}\tilde{g}_{0}=dr^{2}+r^{2}\frac{g_{S}}{c^{2}}$, as claimed.

As for the fact that \eqref{northbeach} holds true along $M_{\frac{c^{2}}{4}}$, we have from Young's inequality applied to \eqref{flingg} that
\begin{eqnarray}
C^{-1}d_{g(t)}(p,\,x)-C\sqrt{t}\leq \rho(x)\leq Cd_{g(t)}(p,\,x)+C\sqrt{t},\quad t\in(0,1],\quad x\in M_{0}.\label{eq-rk-rho-dist}
\end{eqnarray}
Using this, we can now prove an estimate less sharp than \eqref{northbeach}.
\begin{claim}\label{toughie}
For all $x\in M_{\frac{c^{2}}{4}}$ and $k\in\mathbb{N}_{0}$,
\begin{equation}
\begin{split}
|(\nabla^{g})^{k}(g-\tilde{g}_0)|_{g}(x)&\leq C_k\rho(x)^{-2-k}.\label{sharp-est-nabla-g}
\end{split}
\end{equation}
\end{claim}
\begin{proof}[Proof of Claim \ref{toughie}]
Let us prove the claim first for $k=0$. Since the curvature tensor of $g$ (and hence that of $g(s)$)
decays quadratically with derivatives, we have that for any $p$ in the critical set of $f$ and
for any $x\in M_{\frac{c^{2}}{4}}$,
\begin{equation*}
\begin{split}
|g-\tilde{g}_0|_{g}(x)&\leq\int_{0}^{1}|\partial_{s}g(s)|_{g}(x)\,ds\leq C\int_{0}^{1}|\Ric(g(s))|_{g}(x)\,ds\\
&\leq C\int_{0}^{1}d_{g(s)}(p,x)^{-2}\,ds\leq C\rho(x)^{-2},\\
\end{split}
\end{equation*}
where we have used \eqref{unif-equiv-metrics-end} and \eqref{eq-rk-rho-dist} after increasing $c$ if necessary.

As for the case $k=1$, we must work slightly harder. Recall that if $T$ is a tensor on $M$, then $\nabla^{g(t)}T=\nabla^gT+g(t)^{-1}\ast\nabla^g(g(t)-g)\ast T$ since at the level of Christoffel symbols, one has that
$$\Gamma(g(t))_{ij}^k=\Gamma(g)_{ij}^k+\frac{1}{2}g(t)^{km}\left(\nabla_i^g(g(t)-g)_{jm}+\nabla_j^g(g(t)-g)_{im}-\nabla_m^g(g(t)-g)_{ij}\right).$$
Thus, for all $x\in M_{\frac{c^{2}}{4}}$ and $t\in(0,1]$, we have that
\begin{equation*}
\begin{split}
\partial_t|\nabla^{g}(g(t)-g)|^2_{g}(x)&\geq -4|\nabla^g\Ric(g(t))|_g(x)|\nabla^g(g(t)-g)|_g(x)\\
&\geq -4\left(|\nabla^{g(t)}\Ric(g(t))|_g(x)+\left(\left|\left(\nabla^g-\nabla^{g(t)}\right)\Ric(g(t))\right|_g(x)\right)\right)|\nabla^g(g(t)-g)|_g(x)\\
&\geq -C\left(d_{g(t)}(p,x)^{-3}+\left|\left(\nabla^g-\nabla^{g(t)}\right)\Ric(g(t))\right|_g(x)\right)|\nabla^g(g(t)-g)|_g(x)\\
&\geq -C\left(\rho(x)^{-3}+|\nabla^g(g(t)-g)|_g(x)|\Ric(g(t))|_g(x)\right)|\nabla^g(g(t)-g)|_g(x)\\
&\geq -C\left(\rho(x)^{-3}+|\nabla^g(g(t)-g)|_g(x)\right)|\nabla^g(g(t)-g)|_g(x)\\
&\geq -C|\nabla^g(g(t)-g)|^2(x)-C\rho(x)^{-6},
\end{split}
\end{equation*}
where $C$ denotes a positive constant that may vary from line to line and where we used Young's inequality in the last line. Recalling that $|\nabla^{g}(g(t)-g)|^2_{g}=0$ when $t=1$, one can integrate the previous differential inequality between a time $t\in(0,\,1)$ and $t=1$ to obtain
\begin{eqnarray*}
|\nabla^{g}(g(t)-g)|^2_{g}(x)\leq C\rho(x)^{-6},\quad x\in M_{\frac{c^{2}}{4}},\quad t\in(0,1],
\end{eqnarray*}
for some positive constant $C$ uniform in time. This fact implies the desired estimate (\ref{sharp-est-nabla-g}) for $k=1$ by letting $t\to0^{+}$.

The cases $k\geq 2$ are proved by induction on $k$.
\end{proof}

It follows from Claim \ref{toughie} that
\begin{equation*}
|(\nabla^{\tilde{g}_{0}})^{k}(g-\tilde{g}_{0})|_{\tilde{g}_{0}}\leq C_{k}\rho^{-2-k}\quad\textrm{for all $k\in\mathbb{N}_{0}$},
\end{equation*}
so that after pulling back by $\iota$, we have that
\begin{equation*}
|(\nabla^{g_{0}})^{k}(\iota^{*}g-g_{0})|_{g_{0}}\leq C_{k}r^{-2-k}\quad\textrm{for all $k\in\mathbb{N}_{0}$}.
\end{equation*}
We now prove \eqref{northbeach}. To this end,
recall that $\varphi_{t}(x)$ satisfies
\begin{equation*}
\frac{\partial\varphi_{t}}{\partial t}(x)=-\frac{\nabla^g f(\varphi_{t}(x))}{t},\quad\varphi_{1}=\operatorname{id}.
\end{equation*}
Since $4\nabla^{g}f=\nabla^{g_{0}}\rho^{2}=2\rho\nabla^{g_{0}}\rho$ by Lemma \ref{fav}(ii), we have that
\begin{equation*}
\begin{split}
\frac{d}{dt}\rho(\varphi_{t}(x))&=d\rho|_{\varphi_{t}(x)}\left(\frac{\partial\varphi_{t}}{\partial t}(x)\right)
=d\rho|_{\varphi_{t}(x)}\left(-\frac{\nabla^g f(\varphi_{t}(x))}{t}\right)\\
&=-\frac{\rho(\varphi_{t}(x))}{2t}d\rho|_{\varphi_{t}(x)}(\nabla^{g_{0}}\rho|_{\varphi_{t}(x)})=-\frac{\rho(\varphi_{t}(x))}{2t}
\end{split}
\end{equation*}
so that
\begin{equation}\label{popp-0}
\rho(\varphi_{t}(x))=\frac{\rho(x)}{t^{\frac{1}{2}}}.
\end{equation}
Let $\hat{\varphi}_{t}(x)$ satisfy
\begin{equation*}
\frac{\partial\hat{\varphi}_{t}}{\partial t}(x)=-\frac{r}{2t}\frac{\partial}{\partial r}(\hat{\varphi}_{t}(x)),
\quad\hat{\varphi}_{1}=\operatorname{id}.
\end{equation*}
Then since $d\iota(r\partial_{r})=2X$, we
have that $\iota\circ\hat{\varphi}_{t}=\varphi_{t}\circ\iota$, and in light of \eqref{popp-0}, we see that $\hat{\varphi}_{t}^{*}r=\frac{r}{t^{\frac{1}{2}}}$ so that $t\hat{\varphi}_{t}^{*}g_{0}=g_{0}$. Recall that the Ricci flow $g(t)$ defined by $g$ is given by $g(t)=t\varphi^{*}_{t}g,\,t\in(0,\,1]$.
Together with the scaling properties of the norm induced on tensors by $g_0$ and the invariance of the Levi-Civita connection under rescalings, these observations imply that
\begin{eqnarray*}
|(\nabla^{\hat{\varphi}_{t}^*g_{0}})^{k}(\iota^{*}g(t)-g_{0})|_{\hat{\varphi}_{t}^*g_{0}}(x)=|(\nabla^{t^{-1}g_{0}})^{k}(\iota^{*}g(t)-g_{0})|_{t^{-1}g_{0}}(x)=t^{1+\frac{k}{2}}|(\nabla^{g_{0}})^{k}(\iota^{*}g(t)-g_{0})|_{g_{0}}(x),
\end{eqnarray*}
so that
\begin{equation*}
\begin{split}
|(\nabla^{g_{0}})^{k}(\iota^{*}g(t)-g_{0})|_{g_{0}}(x)
&=t^{-1-\frac{k}{2}}|(\nabla^{\hat{\varphi}_{t}^*g_{0}})^{k}(\iota^{*}g(t)-g_{0})|_{\hat{\varphi}_{t}^*g_{0}}(x)\\
&=t^{-1-\frac{k}{2}}\cdot t|(\nabla^{\hat{\varphi}_{t}^*g_{0}})^{k}(\hat{\varphi}^{*}_{t}\iota^{*}g-\hat{\varphi}_{t}^{*}g_{0})|_{\hat{\varphi}_{t}^*g_{0}}(x)\\
&=t^{-1-\frac{k}{2}}\cdot t|(\nabla^{g_{0}})^{k}(\iota^{*}g-g_{0})|_{g_{0}}(\hat{\varphi}_{t}(x))\\
&\leq C_{k}t\cdot t^{-1-\frac{k}{2}}\cdot(r(\hat{\varphi}_{t}(x)))^{-2-k}\\
&=C_{k}t\cdot t^{-1-\frac{k}{2}}\cdot t^{1+\frac{k}{2}}(r(x))^{-2-k}\quad\textrm{for all $k\in\mathbb{N}_{0}$},
\end{split}
\end{equation*}
i.e.,
\begin{equation*}
|(\nabla^{g_{0}})^{k}(\iota^{*}g(t)-g_{0})|_{g_{0}}(x)\leq C_{k}tr^{-2-k}\quad\textrm{for all $k\in\mathbb{N}_{0}$}.
\end{equation*}
In particular,
$$|(\nabla^{g_{0}})^{k}(\operatorname{Ric}(\iota^{*}g(t))-\operatorname{Ric}(g_{0}))|_{g_{0}}\leq C_{k}tr^{-4-k}\quad\textrm{for all $k\in\mathbb{N}_{0}$},$$
which is clear from the expression of the components of the Ricci curvature in local coordinates. Consequently, we have the improved estimate
\begin{equation*}
|(\nabla^{g_{0}})^{k}(\iota^{*}g-g_0-2\operatorname{Ric}(g_{0}))|_{g_{0}}(x)\leq C_{k}\int_{0}^{1}|(\nabla^{g_{0}})^{k}(\operatorname{Ric}(\iota^{*}g(s))-\operatorname{Ric}(g_{0}))|_{g_{0}}(x)\,ds
\leq C_{k}r(x)^{-4-k}.
\end{equation*}
This is precisely \eqref{northbeach}.
\end{proof}

Thus, an expanding gradient Ricci soliton $M$ with quadratic curvature decay with
derivatives has a unique tangent cone $C_{0}$ along each of its ends $V$. Moreover, there
is a diffeomorphism $$\iota:C_{0}\setminus K\to V,$$ where $K\subset C_{0}$ is a compact
subset containing the apex of $C_{0}$, induced by the flow of the vector field $\frac{2X}{\rho}$.
The statement of Theorem \ref{ac} is verbatim the same for expanding gradient K\"ahler-Ricci solitons except
that $d\iota(r\partial_{r})=X$ rather than $2X$
and $2\operatorname{Ric}(g_{0})$ is replaced by $\operatorname{Ric}(g_{0})$ in \eqref{northbeach}, accounting for the difference in normalisation between Ricci solitons and K\"ahler-Ricci solitons.

\subsection{Existence of a resolution map to the tangent cone}

In the case that $(M,\,g,\,X)$ is a complete expanding gradient K\"ahler-Ricci soliton
with quadratic curvature decay with derivatives, the soliton potential is proper \cite{Che-Der},
hence $M$ has only one end $V$ \cite{munteanu} with tangent cone $C_{0}$ along the end. (Also
note from \cite{munteanu} that any complete shrinking gradient K\"ahler-Ricci soliton has only one end without any curvature assumption on the metric.)
Along $V$, we have from Theorem \ref{ac} the diffeomorphism $\iota:C_{0}\setminus K\to V$ for $K\subset C_{0}$ a compact
subset containing the apex of $C_{0}$. As we will now see, the inverse of this map
actually extends to define a resolution $\pi:M\to C_{0}$ with respect to which $d\pi(X)=r\partial_{r}$.
We first show that $\iota$ is a biholomorphism with respect to
a complex structure on $C_{0}$ that makes the cone metric $g_{0}$ K\"ahler. As the next proposition demonstrates, the K\"ahlerity of
the soliton implies this fact.
\begin{prop}\label{biholo}
Let $(M,\,g,\,X)$ be a complete expanding gradient K\"ahler-Ricci soliton with complex structure $J$ such that for some point $p\in M$
 and all $k\in\mathbb{N}_{0}$,
\begin{equation*}
A_{k}(g):=\sup_{x\in M}|(\nabla^{g})^{k}\operatorname{Rm}(g)|_{g}(x)d_{g}(p,\,x)^{2+k}<\infty,
\end{equation*}
where $\operatorname{Rm}(g)$ denotes the curvature of $g$ and $d_{g}(p,\,x)$ denotes the distance between $p$ and $x$ with respect to $g$.
For the unique end $V$ of $M$, let $\tilde{g}_{0}=\lim_{t\to0^{+}}g(t)$ be the limit of the K\"ahler-Ricci flow
 $g(t)$ defined by $(M,\,g,\,X)$ and let $(C_{0},\,g_{0})$ be the unique tangent cone along $V$ with radial
  function $r$ and with $\iota:(C_{0}\setminus K,\,g_{0})\to(V,\,\tilde{g}_{0})$ for $K\subset C_{0}$
   compact containing the apex of $C_{0}$ the isometry of Theorem \ref{ac}. Then $(C_{0},\,g_{0})$ is a K\"ahler cone with respect to $\iota^{*}J$.
   In particular, $\iota:(C_{0}\setminus K,\,\iota^{*}J)\to(V,\,J)$ is a biholomorphism.
\end{prop}

\begin{proof}
Since $\lim_{t\to0^{+}}g(t)=\tilde{g}_{0}$ smoothly on compact subsets of $V$ and $g(t)$ is K\"ahler with respect to $J$, we have that on $V$,
$$\nabla^{\tilde{g}_{0}}J=\lim_{t\to0^{+}}\nabla^{g(t)}J=0,$$ so that $\tilde{g}_{0}$ is K\"ahler with respect to $J$.
 The metric $g_{0}$ is therefore K\"ahler with respect to $\iota^{*}J$ away from a compact subset of $C_{0}$.
  Recall that the radial vector field on a K\"ahler cone is holomorphic with respect to its complex structure.
   Thus, $r\partial_{r}$ is holomorphic on the subset of $C_{0}$ for which $\iota^{*}J$ is defined.
    Flowing along $-r\partial_{r}$ then extends $\iota^{*}J$ to a global complex structure on $C_{0}$ with respect to which $g_{0}$ is K\"ahler.
\end{proof}
\noindent In fact the converse of Proposition \ref{biholo} holds true for shrinking gradient K\"ahler-Ricci solitons; see \cite{kotshy} for details.

The previous proposition implies that $M$ is 1-convex. This property is what allows us to
extend the biholomorphism $\iota^{-1}$ to a resolution $\pi:M\to C_{0}$ that is equivariant with respect to the torus action on $C_{0}$
generated by the flow of $J_{0}r\partial_{r}$. The details are contained in the next theorem.
\begin{theorem}\label{hello}
Let $(M,\,g,\,X)$ be a complete expanding gradient K\"ahler-Ricci soliton with complex structure $J$ such that for some point $p\in M$ and all $k\in\mathbb{N}_{0}$,
\begin{equation*}
A_{k}(g):=\sup_{x\in M}|(\nabla^{g})^{k}\operatorname{Rm}(g)|_{g}(x)d_{g}(p,\,x)^{2+k}<\infty,
\end{equation*}
and let $(C_{0},\,g_{0})$ be its unique tangent cone with radial function $r$ and complex structure $J_{0}$. Then
there exists a holomorphic map $\pi:M\to C_{0}$ that is a resolution of $C_{0}$ with the property that
$d\pi(X)=r\partial_{r}$. Furthermore, the holomorphic isometric real torus action on $(C_{0},\,g_{0},\,J_{0})$ generated by $J_{0}r\partial_{r}$
extends to a holomorphic isometric torus action of $(M,\,g,\,J)$.
\end{theorem}

\begin{proof}
The proof of Theorem \ref{hello} comprises several steps. From Proposition \ref{biholo}, we know that along the unique end $V$ of $M$, there is a biholomorphism $\iota:C_{0}\setminus K\to V$ for $K\subset C_{0}$ compact containing the apex of $C_{0}$. Thus, by \cite[Lemma 2.15]{Conlon}, this in particular implies that $M$ is $1$-convex, hence holomorphically convex, so that there is a Remmert reduction $p:M\to M'$ of $M$. Recall that this is a proper holomorphic map $p:M\to M'$ from $M$ onto a normal Stein space $M'$ with finitely many isolated singularities obtained by contracting the maximal compact analytic subset of $M$. By construction, $M'$ is biholomorphic to $M$ outside compact sets, therefore we have a biholomorphism given by $F:=p\circ\iota:\{x\in C_{0}:r(x)>R\}\to M'\setminus K'$ for some compact subset $K'\subset M'$ and for some $R>0$. We claim that this biholomorphism extends globally.
\begin{claim}\label{beautiful}
The biholomorphism $F:\{x\in C_{0}:r(x)>R\}\to M'\setminus K'$ extends to a biholomorphism $f:C_{0}\to M'$.
\end{claim}

\begin{proof}[Proof of Claim \ref{beautiful}]
Since $M'$, as a Stein space with finitely many isolated singularities, admits an embedding $h:M'\to\mathbb{C}^{P}$ for some $P$ by \cite[Theorem 3.1]{SCV6},
we have a holomorphic function $$h\circ F:\{x\in C_{0}:r(x)>R\}\to\mathbb{C}^{P}.$$ Since $C_{0}$ is in particular an example of a Stein space, this holomorphic function extends to a unique holomorphic function $f:C_{0}\to\mathbb{C}^{P}$ by Hartog's theorem for Stein spaces \cite[Theorem 6.6]{Rossi2}.
The fact that $f(C_{0})\subseteq M'$ follows from Hartog's theorem.
To show that $f$ is in fact a biholomorphism, we construct an inverse map $f^{-1}:M'\to C_{0}$ as an extension of the map $$F^{-1}:M'\setminus K'\to\{x\in C_{0}:r(x)>R\}$$ by applying the previous argument beginning with the fact that $C_{0}$ is affine algebraic.
 \end{proof}

Thus, the Remmert reduction of $M$ is actually $C_{0}$, i.e., the composition $\pi:=f^{-1}\circ p:M\to C_{0}$ is a
proper holomorphic map contracting the maximal compact analytic subset $E$ of $M$ to obtain the cone $C_{0}$. Denote
the connected components of $E$ by $E_{1},\ldots,E_{k}$. Then $\pi$ contracts each $E_{i}$ to a point $p_{i}\in C_{0}$ and restricts to a
biholomorphism $\pi:M\setminus E\to C_{0}\setminus\{p_{1},\ldots,p_{k}\}$. We next show that $\pi$ defines a resolution of $C_{0}$ for which $d\pi(X)=r\partial_{r}$.
Note that at infinity, $\pi=(p\circ\iota)^{-1}\circ p=\iota^{-1}$.

\begin{claim}\label{nice}
The map $\pi:=f^{-1}\circ p:M\to C_{0}$ is a resolution of $C_{0}$ with respect to which
$d\pi(X)=r\partial_{r}$.
\end{claim}

\begin{proof}[Proof of Claim \ref{nice}]
Consider the biholomorphism $\pi:M\setminus E\to C_{0}\setminus\{p_{1},\ldots,p_{k}\}$.
This map allows us to lift the holomorphic vector field $r\partial_{r}$ to a holomorphic vector field $Y:=(d\pi)^{-1}(r\partial_{r})$ on $M\setminus E$. Since at infinity $\pi=\iota^{-1}$ and so identifies the vector field
$r\partial_{r}$ on $C_{0}$ with the vector field $X$ on $M$ outside compact subsets of each, $Y$ will agree with $X$ outside of a compact
subset of $M$. Thus, analyticity implies that $X=Y$ on $M\setminus E$. The next observation is that since the flow lines of $Y$ (and hence $X$)
foliate $M\setminus E$, the flow of $X$ must preserve $E$. Via $\pi$ therefore, the flow of $X$ induces
a flow on $C_{0}$ that fixes the points $p_{1},\ldots,p_{k}$, as above each $p_{i}$ the image of a connected component $E_{i}$ of $E\subset M$ under $\pi$.
The result of this induced flow on $C_{0}$ is a holomorphic vector field $\hat{X}$ that coincides with
 $r\partial_{r}$ on $C_{0}\setminus\{p_{1},\ldots,p_{k}\}$ and which is equal to zero at each $p_{i}$.
By analyticity again, $\hat{X}=r\partial_{r}$ so that $E$ comprises one connected component only which is
mapped to the apex of the cone by $\pi$. Thus, $\pi:M\to C_{0}$ is a resolution of the singularity of the cone
and the vector field $X$ on $M$ is an extension of $(d\pi)^{-1}(r\partial_{r})$ from $M\setminus E$ to $M$
so that $d\pi(X)=r\partial_{r}$, as claimed.
\end{proof}

The resolution $\pi:M\to C_{0}$ is clearly equivariant with respect to the flow of $J_{0}r\partial_{r}$ on $C_{0}$ and the flow of
$JX$ on $M$. We wish to show next that $\pi:M\to C_{0}$ is in fact equivariant with respect to the holomorphic isometric torus action on $C_{0}$ induced by
the flow of $J_{0}r\partial_{r}$ and that the lift of this torus action to $M$ acts isometrically on $g$.
This will conclude the proof of Theorem \ref{hello}.

\begin{claim}\label{nicer}
The holomorphic isometric torus action on $(C_{0},\,g_{0})$ generated by $J_{0}r\partial_{r}$ extends to a holomorphic isometric
action of $(M,\,g,\,J)$ so that in particular, $\pi:M\to C_{0}$ is equivariant with respect to this torus action.
\end{claim}

\begin{proof}[Proof of Claim \ref{nicer}]
Consider the isometry group of $(M,\,g)$ that fixes $E$ endowed with the topology induced by uniform convergence on compact
subsets of $M$. By the Arzel\`a-Ascoli theorem, this is a compact Lie group. Taking the closure of the flow of $JX$ in this group therefore
yields the holomorphic isometric action of a torus $T$ on $(M,\,J,\,g)$. Since the action of $T$
preserves $E$, this action pushes down via $\pi$ to a holomorphic action of $T$ on $C_{0}$ fixing the apex $o$ of $C_{0}$.
Now, by Theorem \ref{ac}, after noting again that $\pi=\iota^{-1}$ at infinity, we see that the soliton metric $g$ and the cone metric $g_{0}$ are asymptotic at infinity. Therefore these metrics are quasi-isometric on $C_{0}\setminus K$, where $K\subset C_{0}$ is any compact subset of $C_{0}$ containing the apex $o$ of $C_{0}$, so that uniform convergence on compact subsets of $C_{0}\setminus\{o\}$
measured with respect to $g$ and $g_{0}$ are equivalent. Recall that $d\pi(X)=r\partial_{r}$ so that the flow of $J_{0}r\partial_{r}$ is dense in $T$
and that the flow of $J_{0}r\partial_{r}$ is isometric with respect to $g_{0}$.
Consequently, every automorphism of $(C_{0},\,J_{0})$ induced by $T$ is obtained as a limit of automorphisms of $(C_{0}\setminus\{o\},\,g_{0},\,J_{0})$ with respect to uniform convergence on compact subsets measured using $g_{0}$. Since a uniform limit of isometries is itself an isometry, it follows that
$T$ acts isometrically with respect to $g_{0}$ on $C_{0}\setminus\{o\}$
so that the action of $T$ on $C_{0}$ preserves the slices of $C_{0}$ and defines a torus in the isometry group of the link of $C_{0}$
in which the flow of $J_{0}r\partial_{r}$ is dense. This final observation concludes the proof of the claim.
\end{proof}

\end{proof}

\subsection{Conclusion of the proof of Theorem \ref{main}}
We now conclude the proof of Theorem \ref{main} for complete expanding gradient K\"ahler-Ricci solitons. Conclusion (a) follows from \cite[Theorem 4.3.1]{Siepmann}, whereas the K\"ahlerity of $(C_{0},\,g_{0})$ as stated in conclusion (b) follows from Proposition \ref{biholo}.
The remainder of conclusion (b), apart from (b)(i), then follows from Theorem \ref{hello}. Conclusion (c) follows from Theorem \ref{ac} after noting that $\pi=\iota^{-1}$ at infinity as above.

As for conclusion (b)(i), the K\"ahler form $\omega$ of the expanding gradient K\"ahler-Ricci soliton satisfies the expanding soliton equation $\rho_{\omega}+i\partial\bar{\partial}f=-\omega$ on $M$, where $\rho_{\omega}$ is the Ricci form of $\omega$ and $f$ is the soliton potential. In $H^{2}(M)$, this equation yields $[-\rho_{\omega}]=[\omega]$. Since $i\rho_{\omega}$ is the curvature form $\Theta$ resulting from the hermitian metric on $K_{M}$ induced by $\omega$, we have that $[i\Theta]=[-\rho_{\omega}]=[\omega]$ so that \eqref{condition-main} is seen to hold true for the expanding soliton K\"ahler form $\omega$ and the curvature form $i\Theta$ it induces on $K_{M}$.

For a complete shrinking gradient Ricci soliton $(M,\,g,\,X)$ with soliton potential $f$, we define a K\"ahler-Ricci flow via $$g(t)=-t\varphi_{t}^{*}g,\quad t<0,$$ where $\varphi_{t}$ is a family of diffeomorphisms generated by the gradient vector field $-\frac{1}{t}X$ with $\varphi_{-1}=\operatorname{id}$, i.e.,
\begin{equation*}
\frac{\partial\varphi_{t}}{\partial t}(x)=-\frac{\nabla^g f(\varphi_{t}(x))}{t},\quad\varphi_{-1}=\operatorname{id}.
\end{equation*}
Then $\frac{\partial g}{\partial t}(t)=-2\operatorname{Ric}(g(t))$ for $t<0$ and $g(-1)=g$. Such a soliton with quadratic curvature decay has quadratic curvature decay with derivatives by Theorem \ref{theo-basic-prop-shrink}(iii), hence, as proved in
\cite[Sections 2.2--2.3]{wangl}, has a unique tangent cone at infinity.  These observations provide the starting point for the proof of Theorem \ref{main} for
complete shrinking gradient K\"ahler-Ricci solitons with quadratic curvature decay. The proof then follows verbatim the proof as in the expanding case.

\section{Classification results for expanding gradient K\"ahler-Ricci solitons with quadratic curvature decay with derivatives}\label{classie}

\subsection{Proof of Corollary \ref{unique}}
Let $(M,\,g,\,X)$ be a complete expanding gradient K\"ahler-Ricci soliton satisfying \eqref{hubba} with tangent cone
$(C_{0},\,g_{0})$ as in Corollary \ref{unique}. Let $\omega$ denote the K\"ahler form of $g$.

To see that $M$ is the canonical model of $C_{0}$,
note first that Theorem \ref{main} asserts that there is a K\"ahler resolution
$\pi:M\to C_{0}$ with exceptional set $E$ such that
\begin{equation}\label{condition1}
\int_{V}(i\Theta)^{k}\wedge\omega^{\dim_{\mathbb{C}}V-k}>0
\end{equation}
for all positive-dimensional irreducible analytic subvarieties $V\subset E$ of $\pi:M\to C_{0}$
and for all integers $k$ such that $1\leq k\leq \dim_{\C}V$, where $\Theta$ denotes the curvature form of the hermitian metric on $K_{M}$ induced by $\omega$.
In particular, \eqref{condition1} implies that
\begin{equation*}
\int_{V}(i\Theta)^{k}\wedge\omega^{\dim_{\mathbb{C}}V-k}>0
\end{equation*}
for all positive-dimensional irreducible \emph{algebraic} subvarieties $V\subset E$ and for all integers $k$ such that $1\leq k\leq \dim_{\C}V$. Setting $k=\dim_{\mathbb{C}}V$, we then see that
\begin{equation*}
\int_{V}(i\Theta)^{\dim_{\mathbb{C}}V}>0\quad\textrm{for every irreducible algebraic subvariety $V\subset E$ of positive dimension.}
\end{equation*}
But since $M$ is quasi-projective by Proposition \ref{quasi2}, this is the same as saying that
\begin{equation*}
(D^{\dim_{\mathbb{C}}V}\cdot V)>0\quad\textrm{for every irreducible algebraic subvariety $V\subset E$ of positive dimension,}
\end{equation*}
where $D$ is now a canonical divisor of $M$. Nakai's criterion for a mapping (Theorem \ref{nakai}) now tells us that $\pi$ is $K_{M}$-ample, so that by definition, $\pi:M\to C_{0}$ is the canonical model of $C_{0}$. Hence $C_{0}$ has a smooth canonical model, namely $M$.

Conversely, suppose that $(C_{0},\,g_{0})$ is a K\"ahler cone with radial function $r$ and with
a smooth canonical model $\pi:M\to C_{0}$. We begin by explaining that \cite[Theorem A]{con-der} holds true without hypothesis (b) of that theorem.
This hypothesis was required in the proof of \cite[Proposition 3.2]{con-der} to show that $\mathcal{L}_{X}\omega=i\partial\bar{\partial}\theta_{X}$,
where $\omega$ is the K\"ahler form of \cite[Proposition 3.1]{con-der}, $X$ is the lift of the radial vector field on the cone,
and $\theta_{X}$ is a smooth real-valued function. The following claim asserts that
this in fact always holds true.

\begin{claim}\label{brill}
Let $(C_{0},\,g_{0})$ be a K\"ahler cone with complex structure $J_{0}$ and radial function $r$ and let
$\pi:M\to C_{0}$ be an equivariant resolution with respect to the real torus action on $C_{0}$ generated by $J_{0}r\partial_{r}$.
Let $X$ be the unique holomorphic vector field on $M$ with $d\pi(X)=r\partial_{r}$
and let $\omega$ be the K\"ahler form of \cite[Proposition 3.1]{con-der}. Then $\mathcal{L}_{X}\omega=i\partial\bar{\partial}\theta_{X}$
for a smooth real-valued function $\theta_{X}:M\to\mathbb{R}$.
\end{claim}

\begin{proof}
Denote the complex structure of $M$ by $J$ and let $X^{1,\,0}=\frac{1}{2}(X-iJX)$.
Then since $\mathcal{L}_{JX}\omega=0$ by construction, we have that
\begin{equation}\label{why}
\frac{1}{2}\mathcal{L}_{X}\omega=\frac{1}{2}d(\omega\lrcorner X)=d(\omega\lrcorner X^{1,\,0}).
\end{equation}
Now by construction, $\omega$ takes the form
$\omega=i\widetilde{\Theta_{h}}+i\partial\bar{\partial}u,$
where $u:M\to\mathbb{R}$ is a smooth real-valued function
and $\widetilde{\Theta_{h}}$ is the average over the action of the torus on $M$ of
the curvature form $\Theta_{h}$ of a hermitian metric $h$ on $K_{M}$. Thus,
\begin{equation}\label{whynot}
\omega\lrcorner X^{1,\,0}=i\widetilde{\Theta_{h}}\lrcorner X^{1,\,0}+i\bar{\partial}(X^{1,\,0}\cdot u).
\end{equation}

Studying the term $i\widetilde{\Theta_{h}}\lrcorner X^{1,\,0}$,
let $n=\dim_{\mathbb{C}}M$, let $\Omega$ be a local holomorphic volume form on $M$, i.e.,
a nowhere vanishing locally defined holomorphic $(n,\,0)$-form (defined in some local holomorphic coordinate chart for example),
and set $$v:=X^{1,\,0}\cdot\log\left(\|\Omega\|_{h}^{2}\right)-\frac{\mathcal{L}_{X^{1,\,0}}\Omega}{\Omega},$$
where $\|\cdot\|_{h}$ denotes the norm with respect to $h$. We claim that $v$ is independent of the choice of $\Omega$ and hence is globally defined. Indeed, any other local holomorphic volume form takes the form $f\Omega$ for some holomorphic function $f$. Then
\begin{equation*}
\begin{split}
X^{1,\,0}\cdot\log\left(\|f\Omega\|_{h}^{2}\right)-\frac{\mathcal{L}_{X^{1,\,0}}(f\Omega)}{f\Omega}
&=X^{1,\,0}\cdot\log|f|^{2}+X^{1,\,0}\cdot\log\left(\|\Omega\|_{h}^{2}\right)-\left(\frac{(X^{1,\,0}\cdot f)\Omega+f\mathcal{L}_{X^{1,\,0}}\Omega}{f\Omega}\right)\\
&=X^{1,\,0}\cdot\log\left(\|\Omega\|_{h}^{2}\right)-\frac{\mathcal{L}_{X^{1,\,0}}\Omega}{\Omega}+\underbrace{
X^{1,\,0}\cdot\log|f|^{2}-\frac{X^{1,\,0}\cdot f}{f}}_{=\,0}\\
&=X^{1,\,0}\cdot\log\left(\|\Omega\|_{h}^{2}\right)-\frac{\mathcal{L}_{X^{1,\,0}}\Omega}{\Omega},
\end{split}
\end{equation*}
as required. Next observe that
\begin{equation*}
i\Theta_{h}\lrcorner X^{1,\,0}=-i\bar{\partial}\left(X^{1,\,0}\cdot\log\left(\|\Omega\|^{2}_{h}\right)\right)
=-i\bar{\partial}\left(X^{1,\,0}\cdot\log\left(\|\Omega\|_{h}^{2}\right)-\frac{\mathcal{L}_{X^{1,\,0}}\Omega}{\Omega}\right)=-i\bar{\partial}v,
\end{equation*}
since $\frac{\mathcal{L}_{X^{1,\,0}}\Omega}{\Omega}$ is a holomorphic function.
Averaging this equation over the action of $T$ then yields
the fact that $i\widetilde{\Theta_{h}}\lrcorner X^{1,\,0}=i\bar{\partial}\tilde{v}$
for a smooth function $\tilde{v}$ on $M$. Plugging this into
\eqref{why}, we thus see from \eqref{whynot} that
$$\mathcal{L}_{X}\omega=2d(i\bar{\partial}(\tilde{v}+X^{1,\,0}\cdot u))
=i\partial\bar{\partial}(2(\tilde{v}+X^{1,\,0}\cdot u)).$$
Hence $\mathcal{L}_{X}\omega=i\partial\bar{\partial}\theta_{X}$,
where $\theta_{X}:=2\operatorname{Re}(\tilde{v}+X^{1,\,0}\cdot u)$,
because $\mathcal{L}_{X}\omega$ is a real $(1,\,1)$-form and
$i\partial\bar{\partial}$ is a real operator.
\end{proof}

\begin{remark}
The existence of the function $v$ satisfying $i\Theta_{h}\lrcorner X^{1,\,0}=-i\bar{\partial}v$
is essentially due to the fact that $X$ has a canonical lift to the total space of $K_{M}$ and
$\Theta_{h}$ is the curvature form of a hermitian metric on $K_{M}$.
\end{remark}

Returning now to our smooth canonical model $\pi:M\to C_{0}$ of $C_{0}$, we
will verify the hypotheses of \cite[Theorem A]{con-der} (apart from the redundant hypothesis (b) of this theorem)
for this resolution to show that $M$ admits a complete expanding gradient K\"ahler-Ricci soliton $g$
with the desired asymptotics. By Lemma \ref{liftt}, the radial vector field $r\partial_{r}$ on $C_{0}$
lifts to a holomorphic vector field $X$ on $M$ with $d\pi(X)=r\partial_{r}$, and by Lemma \ref{quasi}, $M$ is quasi-projective, hence K\"ahler. Moreover,
there exists a K\"ahler form $\sigma$ on $M$ and a hermitian metric on $K_{M}$ with curvature form $\Theta$ such that
\begin{equation}\label{condition}
\int_{V}(i\Theta)^{k}\wedge\sigma^{\dim_{\mathbb{C}}V-k}>0
\end{equation}
for all positive-dimensional irreducible analytic subvarieties $V$ contained in the exceptional set $E$ of $\pi:M\to C_{0}$ and
 for all integers $k$ such that $1\leq k\leq \dim_{\C}V$. Indeed, proceeding as in \cite{DP}, let $\sigma$ be the curvature form of a
  very ample line bundle $L$ on the projective variety which contains $M$ as an open subset and let $\Theta$ be the curvature
   form of the hermitian metric induced on $K_{M}$ by $\sigma$. Then observe that for any analytic subvariety $V\subset E$ of dimension $k$,
\begin{equation}\label{lover}
\int_V (i\Theta)^{k}\wedge\sigma^{\dim_{\mathbb{C}}V-k}=\int_{V\cap H_1\cap \dots\cap H_{\dim_{\mathbb{C}}V-k}}(i\Theta)^{k}
\end{equation}
for generic members $H_1,\dots,H_{\dim_{\mathbb{C}}V-k}$ of the linear system $|L|$, so that $V\cap H_1\cap \dots\cap H_{\dim V-k}\subset E$ is an irreducible subvariety of dimension $k$. Since $E$ is projective (as $M$ is quasi-projective), this intersection is a projective algebraic variety by Chow's theorem. The
 right-hand side of \eqref{lover} may therefore be written as $D^{k}\cdot(V\cap H_1\cap \dots\cap H_{\dim V-k})$,
  where $D$ is a canonical divisor of $M$. By definition of the canonical model, the map $\pi$ is $K_{M}$-ample,
  which by Nakai's criterion for a mapping (cf.~Theorem \ref{nakai}) implies that this intersection is strictly positive.
Thus, we have that \eqref{condition} holds true for the
K\"ahler form $\sigma$ and the curvature form $\Theta$ that it induces on $K_{M}$.
The hypotheses required for the application of \cite[Theorem A]{con-der} are therefore satisfied and
so $M$ admits a complete expanding gradient K\"ahler-Ricci soliton $(M,\,g,\,X)$ with
\begin{equation*}
|(\nabla^{g_{0}})^k(\pi_{*}g-g_{0}-\Ric(g_{0}))|_{g_0 } \leq C_{k}r^{-4-k}\quad\textrm{for all $k\in\mathbb{N}_{0}$}
\end{equation*}
as required.

As for the uniqueness of $(M,\,g,\,X)$,
let $(M_{i},\,g_{i},\,X_{i}),\,i=1,\,2,$ be two complete expanding gradient K\"ahler-Ricci solitons satisfying
\eqref{hubba} with tangent cone $(C_{0},\,g_{0})$. As initially proved, both $M_{1}$ and $M_{2}$ are equal to the unique (smooth) canonical model $M$ of $C_{0}$. Moreover, Theorem \ref{main} asserts that there exists a resolution map
$\pi_{i}:M\to C_{0}$ for $i=1,\,2,$ with $d\pi_{i}(X_{i})=r\partial_{r}$ such that
\begin{equation}\label{honolulu}
|(\nabla^{g_{0}})^k((\pi_{i})_{*}g_{i}-g_{0}-\operatorname{Ric}(g_{0}))|_{g_0} \leq C_{k}r^{-4-k}\quad\textrm{for all $k\in\mathbb{N}_{0}$.}
\end{equation}
The composition $H:=\pi_{2}\circ\pi_{1}^{-1}:C_{0}\to C_{0}$ induces an automorphism of $C_{0}$ fixing the
vertex. As in the proof of Lemma \ref{liftt}, uniqueness of the canonical model
implies that there exists a unique biholomorphism $F:M\to M$ such that
$\pi_{1}\circ F=H\circ \pi_{1}$. Unravelling the definition of
$H$, this yields the fact that $\pi_{1}\circ F=\pi_{2}$.
Consequently, $d\pi_{2}((dF)^{-1}(X_{1}))=d\pi_{1}(X_{1})=r\partial_{r}$
so that $(dF)^{-1}(X_{1})=X_{2}$. Furthermore, in light of \eqref{honolulu}, we have that
\begin{equation}\label{honolulu2}
|(\nabla^{g_{0}})^k((\pi_{2})_{*}(F^{*}g_{1})-g_{0}-\operatorname{Ric}(g_{0}))|_{g_0} \leq C_{k}r^{-4-k}\quad\textrm{for all $k\in\mathbb{N}_{0}$.}
\end{equation}
Thus, $(M,\,F^{*}g_{1},\,X_{2})$ and $(M,\,g_{2},\,X_{2})$
are two expanding gradient K\"ahler-Ricci solitons with the same soliton vector field
which from [\eqref{honolulu},\,$i=2$] and \eqref{honolulu2} in addition satisfy
$|F^{*}g_{1}-g_{2}|=O(r^{-4})$. The uniqueness theorem
\cite[Theorem C(ii)]{con-der} therefore applies
(where, in studying the proof of \cite[Theorem C(ii)]{con-der},
one sees that finite fundamental group is not actually required)
and asserts that $F^{*}g_{1}=g_{2}$. Thus, $(M,\,g,\,X)$ is
unique up to pullback by biholomorphisms of $M$, as claimed.

As for the remainder of Corollary \ref{unique}, item (a) is now clear
and item (b) follows from Theorem \ref{main}.

\subsection{Proof of Corollary \ref{unique2d}}
Corollary \ref{unique2d} follows from Corollary \ref{unique} once we identify
the two-dimensional K\"ahler cones that admit smooth canonical models
as those stated in Corollary \ref{unique2d}(I)-(III)
and realise their respective smooth canonical models as those stated in Corollary \ref{unique2d}(b)(i)-(iii).

To this end, let $C_{0}$ be a two-dimensional K\"ahler cone with a
smooth canonical model $M$. By adjunction, $M$ cannot contain any $(-1)$ or $(-2)$-curves.
In particular, by Theorem \ref{minn}, $M$ coincides with the minimal model of $C_{0}$. Using this information,
we can identify $C_{0}$ and $M$ as follows.

Since $C_{0}$ is a two-dimensional K\"ahler cone, it must be prescribed as in Theorem \ref{classs}.
We henceforth work on a case-by-case basis. If $C_{0}$ is as in Theorem \ref{classs}(i), then $\Gamma$ must be as prescribed in Corollary \ref{unique2d}(I) since $M$ cannot contain any $(-2)$-curves; indeed, see \cite[Figure 2.1 and Theorem 4.1]{Jeff} for details. In this case, $M$ will be the minimal model of $C_{0}$ as stated in Corollary \ref{unique2d}(b)(i). Otherwise, $C_{0}$ may be as in Theorem \ref{classs}(ii) which is precisely the statement of
Corollary \ref{unique2d}(II). In this case, the minimal model $M$ is given as in the statement of Corollary \ref{unique2d}(b)(ii).
Finally, $C_{0}$ may be as in Theorem \ref{classs}(iii). Those cones of Theorem \ref{classs}(iii) that admit a smooth canonical model have been identified in Proposition \ref{sauce} which yields the statement of Corollary \ref{unique2d}(III). For these cones, the minimal resolution is the minimal good resolution which identifies $M$ as in the statement of Corollary \ref{unique2d}(b)(iii).

\section{A volume minimising principle for complete shrinking gradient K\"ahler-Ricci solitons}\label{futake}

We now focus our attention solely on shrinking gradient K\"ahler-Ricci solitons for the remainder of the article. The set-up of this section is as follows. Let $(M,\,g,\,X)$ be a complete shrinking gradient K\"ahler-Ricci soliton of complex dimension $n$ with complex structure $J$, K\"ahler form $\omega$, and with soliton vector field $X=\nabla^{g}f$ for a smooth real-valued function $f:M\to\mathbb{R}$. We assume that a real torus $T$ with Lie algebra $\mathfrak{t}$ acts holomorphically, effectively, and isometrically on $(M,\,g,\,J)$. Then $\mathfrak{t}$ can be identified with real holomorphic Killing vector fields on $M$. We furthermore assume that $JX\in\mathfrak{t}$.

The goal of this section is to prove the uniqueness of the soliton vector field $JX$ in $\mathfrak{t}$ by characterising $JX$ as the unique critical point of a soon-to-be-defined weighted volume functional.

\subsection{A Matsushima-type theorem}\label{matsushimaa}
Let $\mathfrak{aut}^{X}(M)$ denote the Lie algebra of real holomorphic vector fields on $M$ that commute with $X$ and hence $JX$, and let $\mathfrak{g}^{X}$ denote the Lie algebra of real holomorphic $g$-Killing vector fields on $M$ that commute with $X$ and hence $JX$.  Clearly $\mathfrak{g}^{X}$ is a Lie subalgebra of $\mathfrak{aut}^{X}(M)$. In order to prove the uniqueness of the soliton vector field $X$, we need to show that the connected component of the identity of the Lie group of holomorphic isometries of $(M,\,g,\,J)$ commuting with the flow of $X$ is maximal compact in the connected component of the identity of
the Lie group of automorphisms of $(M,\,J)$ commuting with the flow of $X$. This fact will follow from the next theorem, an analogue of Matsushima's theorem \cite{matsushima} for shrinking gradient K\"ahler-Ricci solitons stating that the Lie algebra $\mathfrak{aut}^{X}(M)$ is reducible, after we prove that
the aforementioned groups are indeed Lie groups.

\begin{theorem}[A Matsushima theorem for shrinking K\"ahler-Ricci solitons]\label{matty}
Let $(M,\,g,\,X)$ be a complete shrinking gradient K\"ahler-Ricci soliton with complex structure $J$
endowed with the holomorphic, effective, isometric action of a real torus $T$ with Lie algebra $\mathfrak{t}$ with $JX\in\mathfrak{t}$. If $|\operatorname{Ric}(g)|_{g}$ is bounded, then we have that $$\mathfrak{aut}^{X}(M)=\mathfrak{g}^{X}\oplus J\mathfrak{g}^{X}.$$

\end{theorem}

\noindent We expect this theorem to hold true without the assumption of bounded Ricci curvature.

The proof of Theorem \ref{matty} consists of several steps. Beginning with any real holomorphic vector field $Y\in\mathfrak{aut}^{X}(M)$,
H\"ormander's $L^{2}$-estimates allow for a complex-valued potential, that is, a smooth complex-valued function $u_{Y}$ such that $Y^{1,\,0}=\nabla^{1,\,0}u_{Y}$,
where $Y^{1,\,0}$ is the $(1,\,0)$-part of $Y$. Thanks to the defining equation of a shrinking gradient K\"ahler-Ricci soliton, we can
then modify $u_{Y}$ by a holomorphic function if necessary so that $\Delta_{\omega}u_{Y}+u_{Y}-Y^{1,\,0}\cdot f=0$, $f$ here the soliton potential
with $X=\nabla^{g}f$. Using the fact that
$[X^{1,\,0},\,Y^{1,\,0}]=0$, where $X^{1,\,0}$ is the $(1,\,0)$-part of the soliton vector field $X$,
we then show that $\bar{\partial}(\Delta_{\omega}\bar{u}_{Y}+\bar{u}_{Y}-Y^{1,\,0}\cdot f)=0$. We next apply a Bochner formula followed by an
integration by parts argument to deduce that $\nabla^{0,\,2}\bar{u}_{Y}=0$ so that $\nabla^{1,\,0}\bar{u}_{Y}$ is a holomorphic vector field. The bound on the norm of the Ricci curvature is required to control the boundary term in the integration by parts argument.
The gradient of the real and imaginary parts of $u_{Y}$ will therefore be real holomorphic vector fields so that, once one applies the complex structure
to these vector fields, they become real holomorphic and Killing. From this, the stated decomposition follows. To conclude that the
sum is direct, we make use of a splitting theorem for shrinking gradient Ricci solitons.

\begin{proof}[Proof of Theorem \ref{matty}]
Let $Y\in\mathfrak{aut}^{X}(M)$. Then $Y$ defines a real holomorphic vector field on $M$ with $[X,\,Y]=0$.
Take the $(1,\,0)$-part $Y^{1,\,0}$ of $Y$, i.e., let $Y^{1,\,0}=\frac{1}{2}(Y-iJY)$.
 Then $\bar{\partial}Y^{1,\,0}=0$ so that $\omega\lrcorner Y^{1,\,0}$ is a $\bar{\partial}$-closed $(0,\,1)$-form, where $\omega$ denotes the K\"ahler form of $g$. We first claim that $\omega\lrcorner Y^{1,\,0}$ admits a smooth complex potential.
\begin{claim}\label{baa}
There exists a smooth complex-valued function $u_{Y}$ on $M$ such that
$-i\omega\lrcorner Y^{1,\,0}=\bar{\partial}u_{Y}$, or equivalently, such that $Y^{1,\,0}=\nabla^{1,\,0}u_{Y}$.
\end{claim}

\noindent Note that $u_{Y}$ is unique up to the addition of a holomorphic function.

\begin{proof}[Proof of Claim \ref{baa}]
Let $h$ denote the metric on $-K_{M}$ induced by $\omega$. Then the curvature of the metric $e^{-f}h$ on $-K_{M}$
is precisely $\omega$ by virtue of the defining equation of a shrinking gradient K\"ahler-Ricci soliton.
Treat $\omega\lrcorner Y^{1,\,0}$ as a $-K_{M}$-valued $(n,\,1)$-form. Then since the norm of $\operatorname{Ric}(g)$ is
bounded so that $|Y^{1,\,0}|_{g}$ grows at most polynomially by Proposition \ref{keywest}, we see from the growth on $f$ dictated by
Theorem \ref{theo-basic-prop-shrink}(i) that the $L^{2}$-norm of $\omega\lrcorner Y^{1,\,0}$ measured with respect to $e^{-f}h$ is finite.
An application of H\"ormander's $L^{2}$-estimates \cite[Theorem 6.1, p.376]{demmy} now yields the desired conclusion.
\end{proof}

Next, contracting [\eqref{krseqn}, $\lambda=1$] with $Y^{1,\,0}$ and using the Bochner formula, we see that
$$-i\bar{\partial}\Delta_{\omega}u_{Y}+i\bar{\partial}(Y^{1,\,0}\cdot f)=i\bar{\partial}u_{Y}$$
so that $$\bar{\partial}(\Delta_{\omega}u_{Y}+u_{Y}-Y^{1,\,0}\cdot f)=0.$$
By adding a holomorphic function to $u_{Y}$ if necessary, we may assume that
$$\Delta_{\omega}u_{Y}+u_{Y}-Y^{1,\,0}\cdot f=0.$$
Furthermore, by averaging $u_{Y}$ over the action of $T$, we may assume that $\mathcal{L}_{JX}u_{Y}=0$. These two operations normalise $u_{Y}$. Taking the complex conjugate of this equation gives us
\begin{equation}\label{start3}
\Delta_{\omega}\bar{u}_{Y}+\bar{u}_{Y}-Y^{0,\,1}\cdot f=0,
\end{equation}
since $f$ is real-valued. At this point, we need to make use of another claim, namely the following.

\begin{claim}\label{claim-hol-diff-part}
Let $u$ and $v$ be smooth complex-valued functions and suppose that $U=\nabla^{1,\,0}u$ and $V=\nabla^{1,\,0}v$ are complex-valued holomorphic vector fields on a K\"ahler manifold such that $[U,\,V]=0$, the Lie bracket here extended by complex-linearity. Then $(\nabla^{1,\,0}v-\nabla^{0,\,1}v)(u)$ is holomorphic.
\end{claim}

\begin{proof}[Proof of Claim \ref{claim-hol-diff-part}]
Let $g$ denote the background K\"ahler metric. Then in local coordinates we have that
\begin{equation}\label{start}
g^{i\bar{\jmath}}u_{\bar{\jmath}}\partial_{i}(g^{k\bar{l}}v_{\bar{l}})\partial_{k}=g^{k\bar{l}}v_{\bar{l}}\partial_{k}(g^{i\bar{\jmath}}u_{\bar{\jmath}})\partial_{i}.
\end{equation}
Choose holomorphic normal coordinates at a point such that $g_{i\bar{\jmath}}=\delta_{ij}$ and $\partial g_{i\bar{\jmath}}=\bar{\partial}g_{i\bar{\jmath}}=0$. Henceforth working at this point, we have from \eqref{start} that for all $k$,
\begin{equation}\label{start2}
u_{\bar{\imath}}v_{\bar{k}i}=v_{\bar{\imath}}u_{\bar{k}i}.
\end{equation}
The fact that $\bar{\partial}\nabla^{1,\,0}u=\bar{\partial}\nabla^{1,\,0}v=0$ yields that for all $i,\,j$, $$u_{\bar{\imath}\bar{\jmath}}=v_{\bar{\imath}\bar{\jmath}}=0.$$
Hence we find from \eqref{start2} that $$\bar{\partial}(u_{\bar{\imath}}v_{i}-u_{i}v_{\bar{\imath}})=0,$$
that is, $g^{i\bar{\jmath}}(u_{i}v_{\bar{\jmath}}-u_{\bar{\jmath}}v_{i})$ is a holomorphic function. Next only observe that $g^{i\bar{\jmath}}u_{i}v_{\bar{\jmath}}=(\nabla^{1,\,0}v)(u)$ and $g^{i\bar{\jmath}}v_{i}u_{\bar{\jmath}}=(\nabla^{0,\,1}v)(u)$ to arrive at the fact that $(\nabla^{1,\,0}v-\nabla^{0,\,1}v)(u)$ is holomorphic.
\end{proof}

\begin{remark}
The above claim can be rephrased by saying that $[U,\,V]=-\{u,\,v\}_{\bar{k}}\partial_{k}$ where $\{\cdot\,,\,\cdot\}$ is the Poisson bracket,
so that $\{u,\,v\}$ is a holomorphic function if $[U,\,V]=0$.
\end{remark}

By assumption, $X=\nabla^g f$ is real holomorphic so that $X^{1,\,0}=\frac{1}{2}(\nabla^g f-iJ\nabla^g f)$ is holomorphic. Using the fact that $[X,\,Y]=0$ and that
$X$ and $Y$ are real holomorphic, one may verify that $[X^{1,\,0},\,Y^{1,\,0}]=0$. Therefore applying the above claim with $u=f$ and $v=u_{Y}$ yields the fact that $$(\nabla^{0,\,1}u_{Y})f=(\nabla^{1,\,0}u_{Y})f+\phi$$
for some holomorphic function $\phi$, i.e., $$Y^{0,\,1}\cdot f=Y^{1,\,0}\cdot f+\phi.$$
Plugging this into \eqref{start3}, we arrive at the fact that $$\bar{\partial}(\Delta_{\omega}\bar{u}_{Y}+\bar{u}_{Y}-Y^{1,\,0}\cdot f)=0.$$

For brevity, we henceforth denote $u_{Y}$ by $u$. Before continuing, we need to establish some estimates on the potential $u$ together with its derivatives. We will divide these estimates up into three claims.

\begin{claim}\label{claim-growth-u}
There exists a positive constant $A$ such that $u_Y(x)=\textit{O}\left(d_g(p,\,x)^{A}\right)$ as $d_g(p,\,x)$ tends to $+\infty$.
\end{claim}

\begin{proof}[Proof of Claim \ref{claim-growth-u}]
By Proposition \ref{keywest}, we know that $Y^{1,\,0}$ grows polynomially, i.e., $|Y^{1,\,0}|(x)=O(d_{g}(p,\,x)^{a})$ for some $a>0$, where $d_{g}(p,\,\cdot)$ denotes the distance with respect to $g$ to a fixed point $p\in M$, so that
$|\bar{\partial}u_Y|(x)=O(d_{g}(p,\,x)^{a})$. Then $$\bar{\partial}u_Y(X)=\frac{1}{2}(du_Y(X)+i\underbrace{du_Y(JX)}_{=\,0})=\frac{1}{2}X\cdot u_Y.$$
Thus,
\begin{equation}
\begin{split}
|X\cdot u_Y|=2|\bar{\partial}u_Y(X)|=O(d_{g}(p,\,x)^{a+1}).\label{growth-est-rad-der-u}
\end{split}
\end{equation}
Let $\gamma_{x}(t)$ be an integral curve of $X$ with $\gamma_{x}(0)=x\in M$. Then
\begin{equation*}
\begin{split}
u_Y(\gamma_{x}(t))&=u_Y(\gamma_{x}(0))+\int_{0}^{t}\frac{d}{ds}u_Y(\gamma_{x}(s))\,ds\\
&=u_Y(\gamma_{x}(0))+\int_{0}^{t}(X\cdot u_Y)(\gamma_{x}(s))\,ds\\
&=A+O(e^{(a+1)t})
\end{split}
\end{equation*}
so that $$|u_Y(x)|=O(d_{g}(p,\,x)^{a+1}).$$

\end{proof}

The next claim concerns the weighted $L^2$-integrability of the total gradient of $u$.

\begin{claim}\label{claim-grad-int}
The gradient $\nabla u_Y$ of $u_Y$ belongs to $L^2(e^{-f}\omega^n)$.
\end{claim}
\begin{proof}[Proof of Claim \ref{claim-grad-int}]
Since $\Delta_{\omega}u_Y=Y^{1,\,0}\cdot f-u_Y$, the estimate established in Claim \ref{claim-growth-u} together with the polynomial growth of $X$ and $Y$ at infinity show that $\Delta_{\omega}u_Y$ is growing at most polynomially at infinity as well. The same is true for the drift term $X\cdot u_Y$ by (\ref{growth-est-rad-der-u}). Therefore the drift Laplacian $\Delta_{\omega}u_Y-X\cdot u_Y$ is growing at most polynomially at infinity ensuring its weighted $L^2$-integrability, i.e., $\Delta_{\omega,\,X}u_Y:=\Delta_{\omega}u_Y-X\cdot u_Y\in L^2(e^{-f}\omega^n)$. This implies in turn that $\nabla \Re(u_Y)$ and $\nabla\operatorname{Im}(u_Y)$ belong to $L^2(e^{-f}\omega^n)$. Indeed, by the previous arguments, it suffices to show that if a smooth real-valued function $v:M\to\mathbb{R}$ satisfies $v\in L^2(e^{-f}\omega^n)$ and $\Delta_{\omega,\,X}v\in L^2(e^{-f}\omega^n)$, then $\nabla v\in L^2(e^{-f}\omega^n)$.

To this end, let $R$ be a positive real number and let $\phi_R:M\to[0,1]$ be a cut-off function with compact support in the geodesic ball
$B_g(p,\,2R)$ such that $\phi_R=1$ on $B_g(p,\,R)$ and $|\nabla\phi_R|_g\leq\frac{c}{R}$. Then since $\left(\Delta_{\omega,\,X}\right)v^2=2|\nabla v|^2+2\left<\Delta_{\omega,\,X}v,v\right>$, integrating by parts leads to the inequality
\begin{eqnarray*}
2\int_M|\nabla v|^2\phi^2_Re^{-f}\omega^n&=&\int_M\Delta_{\omega,\,X}v^2\phi_R^2 e^{-f}\omega^n-2\int_M\left<\Delta_{\omega,\,X}v,v\right>\phi^2_Re^{-f}\omega^n\\
&=&-\int_Mg\left(\nabla v^2,\nabla\phi_R^2\right) e^{-f}\omega^n-2\int_M\left<\Delta_{\omega,\,X}v,v\right>\phi^2_Re^{-f}\omega^n\\
&\leq& \int_M|\nabla v|^2\phi^2_Re^{-f}\omega^n+\frac{c}{R^2}\int_Mv^2e^{-f}\omega^n+\int_M\left(\left|\Delta_{\omega,\,X}v\right|^2+|v|^2\right)e^{-f}\omega^n,
\end{eqnarray*}
which yields
\begin{eqnarray*}
\int_M|\nabla v|^2\phi^2_Re^{-f}\omega^n&\leq&\frac{c}{R^2}\int_Mv^2e^{-f}\omega^n+\int_M\left(\left|\Delta_{\omega,\,X}v\right|^2+|v|^2\right)e^{-f}\omega^n.
\end{eqnarray*}
One then obtains the expected result by letting $R$ tend to $+\infty$.
\end{proof}

Finally, it suffices to show that some components of the Hessian of $\bar{u}_Y$ are in $L^2(e^{-f}\omega^n)$.

\begin{claim}\label{growth-hessian-bar-u}
The $(0,\,2)$-part $\nabla^{0,2}\bar{u}_{Y}$ of the Hessian of $\bar{u}_{Y}$ vanishes identically on $M$.
\end{claim}

\begin{proof}[Proof of Claim \ref{growth-hessian-bar-u}]
For clarity, we suppress the dependence of the potential $u_Y$ on the vector field $Y$ in what follows.

Let $R>0$ and let $\phi_R$ be a cut-off function as in the proof of Claim \ref{claim-grad-int}.
Reminiscent of \cite[equation (2.7)]{tianzhu1}, we then find, in normal holomorphic coordinates at a point where the Ricci form $\rho_{\omega}$ of $\omega$ has components $\rho_{i\bar{\jmath}}$, that
\begin{equation*}
\begin{split}
0&=\langle \bar{\partial}\bar{u}\,,\bar{\partial}(\Delta_{\omega}\bar{u}+\bar{u}-Y^{1,\,0}\cdot f)\,\phi_R^2\rangle_{L_f^2}
=\int_{M}\overline{(\Delta_{\omega}\bar{u}+\bar{u}-Y^{1,\,0}\cdot f)_{\bar{\imath}}}\bar{u}_{\bar{\imath}}\,\phi_R^2\,e^{-f}\omega^{n}\\
&=\int_{M}(\Delta_{\omega} u+u-f_{\bar{k}}u_{k})_{i}\bar{u}_{\bar{\imath}}\,\phi_R^2\,e^{-f}\omega^{n}\\
&=\int_{M}(u_{k\bar{k}i}+u_{i}-f_{i\bar{k}}u_{k}-f_{\bar{k}}u_{ik})\bar{u}_{\bar{\imath}}\,\phi_R^2\,e^{-f}\omega^{n}\\
&=\int_{M}(u_{ki\bar{k}}-\rho_{i\bar{s}}u_{s}+u_{i}-f_{i\bar{k}}u_{k}-f_{\bar{k}}u_{ik})\bar{u}_{\bar{\imath}}\,\phi_R^2\,e^{-f}\omega^{n}\quad\textrm{since $u_{ji\bar{\jmath}}=u_{j\bar{\jmath}i}+\rho_{i\bar{s}}u_{s}$,}\\
&=\int_{M}(u_{ik\bar{k}}+\underbrace{(-\rho_{i\bar{s}}u_{s}+u_{i}-f_{i\bar{k}}u_{k})}_{=\,0}-f_{\bar{k}}u_{ik})\bar{u}_{\bar{\imath}}\,\phi_R^2\,e^{-f}\omega^{n}\\
&=\int_{M}(u_{ik\bar{k}}-f_{\bar{k}}u_{ik})\bar{u}_{\bar{\imath}}\,\phi_R^2\,e^{-f}\omega^{n}=\int_{M}\overline{(\Delta_{\omega,\,X}\bar{u}_{\bar{\imath}})}\bar{u}_{\bar{\imath}}\,\phi_R^2\,e^{-f}\omega^{n}\\
&=-\int_{M}\bar{u}_{\bar{\imath}\bar{\jmath}}u_{ij}\,\phi_R^2\,e^{-f}\omega^{n}-
\frac{1}{2}\int_{M}\langle\bar{\partial}\phi^2_R,\bar{\partial}|\bar{\partial}\bar{u}|^2\rangle \,e^{-f}\omega^n.
\end{split}
\end{equation*}
Therefore, by Young's inequality,
\begin{equation*}
\begin{split}
\int_{M}|\nabla^{0,\,2}\bar{u}|_{\omega}^{2}\,\phi^2_R\,e^{-f}\omega^{n}&\leq2\int_M\left(|\bar{\partial}\phi_R||\bar{\partial}\bar{u}|\right)\cdot\left(\phi_R|\nabla^{0,2}\bar{u}|_{\omega}\right)\,e^{-f}\omega^n\\
&\leq 2\int_M|\bar{\partial}\phi_R|^2|\bar{\partial}\bar{u}|^2+\frac{1}{2}\int_M|\nabla^{0,2}\bar{u}|_{\omega}^2\phi_R^2\,e^{-f}\omega^n.
\end{split}
\end{equation*}
By Claim \ref{claim-grad-int}, the previous inequality leads to
\begin{equation*}
\begin{split}
\int_{M}|\nabla^{0,\,2}\bar{u}|_{\omega}^{2}\,\phi^2_R\,e^{-f}\omega^{n}&\leq \frac{c}{R^2}
\int_{M}|\bar{\partial}\bar{u}|^2e^{-f}\omega^n\leq \frac{c'}{R^2}
\end{split}
\end{equation*}
for some positive constants $c$, $c'$ independent of $R$. As $R$ tends to $+\infty$, this not only proves the weighted $L^2$-integrability of $\nabla^{0,2}\bar{u}$, but also the fact that
\begin{equation*}
\int_{M}|\nabla^{0,\,2}\bar{u}|_{\omega}^{2}\,e^{-f}\omega^{n}=0,
\end{equation*}
as desired.
\end{proof}

\noindent Consequently, $\nabla^{0,\,2}\bar{u}_{Y}=0$, from which it follows that $\nabla^{1,\,0}\bar{u}_{Y}$ is a holomorphic vector field.

Thus, $\nabla^{1,\,0}u_{Y}$ and $\nabla^{1,\,0}\bar{u}_{Y}$ are holomorphic vector fields. Write $u_{Y}=v_{Y}+iw_{Y}$, where $v_{Y}$ and $w_{Y}$ are smooth real-valued functions on $M$. Then we deduce that $\nabla^{1,\,0}v_{Y}=\frac{1}{2}(\nabla v_{Y}-iJ\nabla v_{Y})$ and $\nabla^{1,\,0}w_{Y}
=\frac{1}{2}(\nabla w_{Y}-iJ\nabla w_{Y})$ are holomorphic. In particular, $\nabla v_{Y}$ and $\nabla w_{Y}$ are real holomorphic vector fields on $M$ so that
 by
\cite[Lemma 2.3.8]{fut2},
$J\nabla v_{Y}$ and $J\nabla w_{Y}$ are real holomorphic $g$-Killing vector fields. Therefore we have the decomposition
$$\frac{1}{2}(Y-iJY)=Y^{1,\,0}=\nabla^{1,\,0}u_{Y}=\nabla^{1,\,0}(v_{Y}+iw_{Y})=\frac{1}{2}\left(\nabla v_{Y}+J\nabla w_{Y}\right)-\frac{i}{2}\left(J\nabla v_{Y}-\nabla w_{Y}\right)$$ so that
\begin{equation}\label{nicee}
Y=\nabla v_{Y}+J\nabla w_{Y}=J\nabla w_{Y}+J(-J\nabla v_{Y}).
\end{equation}
Moreover, since $\mathcal{L}_{JX}u_{Y}=0$, we have that $\mathcal{L}_{JX}v_{Y}=\mathcal{L}_{JX}w_{Y}=0$ so that $[JX,\,\nabla v_{Y}]=[JX,\,\nabla w_{Y}]=0$
and consequently $[X,\,J\nabla v_{Y}]=[X,\,J\nabla w_{Y}]=0$. Hence $J\nabla v_{Y}$ and $J\nabla w_{Y}$ lie in $\mathfrak{g}^{X}$ leaving \eqref{nicee} as the
desired decomposition.

To show that this decomposition is direct, suppose that $Z\in\mathfrak{g}^{X}\cap J\mathfrak{g}^{X}$. Then $Z=JW$, where $W$ and $JW$ are real holomorphic and Killing. Since $W$ is holomorphic and $JW$ is Killing, $\nabla W$ is symmetric. Since $W$ is Killing, $\nabla W$ is skew-symmetric. Thus, $W$ is parallel.
If $W$ is non-trivial, then by \cite[Corollary 3.2]{fangg}, $(M,\,g)$ splits off a line with $W$ the generator of this line.
In particular, we may write $M=N\times\mathbb{R}$ for a manifold $N$ with $g=g_{N}\oplus dt^{2}$ and $W=\partial_{t}$, here $t$ the coordinate on
the $\mathbb{R}$-direction and $g_{N}$ a shrinking Ricci soliton on $N$. Now the soliton vector field $X$ must split as a direct sum with the summand in the
$\mathbb{R}$-direction necessarily $t\partial_{t}$. Since $[W,\,X]=0$ as $Z\in\mathfrak{g}^{X}$, this yields a contradiction so that $W=0$. Hence the stated decomposition of $\mathfrak{aut}^{X}(M)$ is direct.
 \end{proof}

Since $M$ is non-compact, we need to verify that the various automorphism groups in question are indeed Lie groups. This is necessary for the applications
of Theorem \ref{matty} that we have in mind. We begin with:
\begin{prop}\label{southbeach}
Let $(M,\,g,\,X)$ be a complete shrinking gradient K\"ahler-Ricci soliton with bounded Ricci curvature.
Then there exists a unique connected Lie group $\operatorname{Aut}_{0}^{X}(M)$ (endowed with the compact-open topology) of diffeomorphisms
acting effectively on $M$ with Lie algebra $\mathfrak{aut}^{X}(M)$.
\end{prop}
\noindent $\operatorname{Aut}_{0}^{X}(M)$ is of course the connected component of the identity of the holomorphic automorphisms of $M$ that commute with the
flow of $X$.

The fact that there is a unique Lie group $\operatorname{Aut}_{0}^{X}(M)$ with the stated properties
follows from Palais' Integrability Theorem \cite{palais} (see also \cite[Theorem 3.1, p.13]{kobayashi1}),
once we establish the completeness and finite-dimensionality of $\mathfrak{aut}^{X}(M)$. However,
this theorem only asserts that $\operatorname{Aut}_{0}^{X}(M)$ is a Lie group with respect to the ``modified'' compact-open topology. In order to
see that it is a Lie group with respect to the compact-open topology, we must appeal to \cite[Theorem 5.14]{palais2}, using the fact that
$\operatorname{Aut}_{0}^{X}(M)$ is closed with respect to the compact-open topology. Now, the completeness of the vector fields in $\mathfrak{aut}^{X}(M)$ is clear from Lemma \ref{completee}. As for their finite-dimensionality, we have the following.

\begin{prop}
Let $(M,\,g,\,X)$ be a complete shrinking gradient K\"ahler-Ricci soliton with bounded Ricci curvature
and let $\mathfrak{aut}^{X}(M)$ denote the space of all real holomorphic vector fields $Y$ on $M$ with $[X,\,Y]=0$. Then $\mathfrak{aut}^{X}(M)$ is finite-dimensional.
\end{prop}

\begin{proof}
We provide an analytic proof of this fact.
Letting $|\cdot|$ denote the norm with respect to $g$ and $f$ the soliton potential, we have a natural norm $\|\cdot\|_{L^2_f}^2$ on $\mathfrak{aut}^{X}(M)$ defined by
$$\|Y\|_{L^2_f}^2:=\int_M|Y|^2e^{-f}\omega^n.$$
It suffices to show that the unit ball is compact with respect to this norm. To this end, suppose that we have a sequence $(Y_i)_{i\,\geq\,0}$ with $\|Y_i\|_{L^2_f}=1$. Then by elliptic estimates, we get uniform $C^k$ bounds on $|Y_i|$ over a fixed ball $B_g(p,\,R)\subset M$ once a $C^0$-estimate is established. Now, by a Nash-Moser iteration applied to the norm of $Y_i$, one obtains the following estimate:
\begin{equation*}
\sup_{B_g\left(p,\,\frac{R}{2}\right)}|Y_i|\leq C(n,R)\|Y_i\|_{L^2(B_g(p,\,R))}.
\end{equation*}
Since $\|Y_i\|_{L^2_f}\leq1$, then a fortiori:
\begin{equation*}
\sup_{B_g\left(p,\,\frac{R}{2}\right)}|Y_i|\leq C(n,R)e^{cR^2}\|Y_i\|_{L^2_f(B_g(p,\,R))}\leq C'(n,R).
\end{equation*}
Finally, according to (the proof) of Proposition \ref{keywest}, there is some large radius $R_0>0$ such that
\begin{equation}
|Y_i|(x)\leq C\left(n,\,\sup_M|\Ric(g)|,\,\sup_{B_g(p,\,R_0)}|Y_i|\right)\cdot\left(d_g(p,x)+1\right)^a,\quad x\in M,\label{unif-bound-vec-fiel}
\end{equation}
for some uniform positive constant $a$, $d_{g}(p,\,x)$ here denoting the distance between $p$ and $x$ with respect to $g$. Since $\sup_{B_g(p,\,R_0)}|Y_i|\leq C(n,R_0)$, passing to a subsequence if necessary, we may assume that $(Y_i)_{i\,\geq\,0}$ converges to some $Y_\infty$ on the whole of $M$ in the $C^{\infty}_{loc}(M)$ topology. The question is whether this convergence is strong in the above norm. Thanks to (\ref{unif-bound-vec-fiel}), given $\varepsilon>0$, there exists some positive radius $R$ such that for all indices $i\geq 0$,
 \begin{equation*}
 \|Y_i\|_{L^2_f(M\setminus B_g(p,\,R))}\leq \varepsilon,
  \end{equation*}
since the soliton potential grows quadratically by Theorem \ref{theo-basic-prop-shrink}(i)
and the volume growth of geodesic balls is at most polynomial by Theorem \ref{theo-basic-prop-shrink}(ii). This shows that if $R$ is chosen sufficiently large, then the remainder of the norm outside $B_{g}(p,\,R)$ is uniformly small, hence we do indeed have strong convergence.
\end{proof}

\begin{remark}
Munteanu-Wang \cite[Theorem 1.4]{Mun-Wan-Hol-Fcts} proved that the space of polynomial growth holomorphic functions of a fixed degree on
a shrinking gradient K\"ahler-Ricci soliton is finite-\linebreak dimensional without assuming a Ricci curvature bound. We therefore expect that the above proposition holds true in more generality. We also expect that the ring of holomorphic functions of polynomial growth on $M$ is finitely generated and that $M$ is algebraic, at least under a Ricci bound assumption.
\end{remark}

Recall that $\mathfrak{g}^{X}$ denotes the Lie algebra comprising real holomorphic $g$-Killing vector fields that commute with $X$ and hence $JX$. We next consider the existence of a Lie group with Lie algebra $\mathfrak{g}^{X}$.
\begin{prop}\label{southbeach2}
Let $(M,\,g,\,X)$ be a complete shrinking gradient K\"ahler-Ricci soliton with bounded Ricci curvature.
Then there exists a unique connected Lie group $G_{0}^{X}$ (endowed with the compact-open topology) of diffeomorphisms
acting effectively on $M$ with Lie algebra $\mathfrak{g}^{X}$.
\end{prop}

\begin{proof}
Since $\mathfrak{g}^{X}$ is a Lie subalgebra of the Lie algebra of $g$-Killing vector fields on $M$, $\mathfrak{g}^{X}$
is a finite-dimensional Lie algebra. Furthermore, vector fields
induced by $\mathfrak{g}^{X}$ on $M$ are complete by Lemma \ref{completee}.
By Palais' Integrability Theorem \cite{palais} therefore, there exists a unique connected Lie group
$G^{X}_{0}$ of diffeomorphisms acting effectively on $M$ with Lie algebra $\mathfrak{g}^{X}$.
$G^{X}_{0}$ is precisely the connected component of the identity of the Lie group of holomorphic isometries on $M$ that commute with the
flow of $X$. Since $G^{X}_{0}$ is closed with respect to the compact-open topology,
\cite[Theorem 5.14]{palais2} guarantees that $G^{X}_{0}$ is a Lie group with respect to this topology, as stated.
\end{proof}

We next prove that $G^{X}_{0}$ is a compact Lie subgroup of $\operatorname{Aut}_{0}^{X}(M)$.
\begin{lemma}\label{tannn}
Let $(M,\,g,\,X)$ be a complete shrinking gradient K\"ahler-Ricci soliton with complex structure $J$
and with soliton vector field $X=\nabla^{g}f$ for a smooth real-valued function $f:M\to\mathbb{R}$.
Then elements of $\mathfrak{g}^{X}$ are tangent to the level sets of $f$. Moreover, if $g$ has bounded Ricci curvature,
then $G^{X}_{0}$ is a compact Lie subgroup of $\operatorname{Aut}_{0}^{X}(M)$ (with respect to the compact-open topology).
\end{lemma}

\begin{proof}
Let $Y\in\mathfrak{g}^{X}$. For the first part of the lemma, we will show that $\mathcal{L}_{Y}f=0$ so that the flow of $Y$ preserves the level sets of $f$,
thereby forcing $Y$ to be tangent to the level sets of $f$.

Applying $\mathcal{L}_{Y}$ to the shrinking K\"ahler-Ricci soliton equation, we find that $i\partial\bar{\partial}(\mathcal{L}_{Y}f)=0$.
Notice that since $[JX,\,Y]=0$, we have that $\mathcal{L}_{JX}(\mathcal{L}_{Y}f)=\mathcal{L}_{Y}(\mathcal{L}_{JX}f)=0$. The function
$X\cdot(\mathcal{L}_{Y}f)$ is therefore holomorphic. It is also real-valued, hence must be equal to a constant, $X\cdot(\mathcal{L}_{Y}f)=c_{0}$ say.
Since $X=\nabla^g f$ and $f$ has a minimum, we deduce that in fact $c_{0}=0$ so that $X\cdot(\mathcal{L}_{Y}f)=0$.

Next, deriving with respect to the Killing vector field $Y$ the soliton identity from Lemma \ref{solitonid}, namely
$$\mathcal{L}_{X}f+R_g=X\cdot f+R_g=|\nabla^g f|^2+R_g=2f,$$
making use of the fact that $X$ and $Y$ commute, we obtain
\begin{equation*}
\begin{split}
2\mathcal{L}_{Y}f&=\mathcal{L}_{Y}\left(\mathcal{L}_{X} f\right)+\underbrace{\mathcal{L}_{Y}R_g}_{=\,0}=\mathcal{L}_{X}\left(\mathcal{L}_{Y}f\right)=X\cdot(\mathcal{L}_{Y}f),
\end{split}
\end{equation*}
where we have just seen that this last term vanishes. Hence $\mathcal{L}_{Y}f=0$, as desired.

As for the second part of the lemma, note that under the assumption of bounded Ricci curvature of $g$, $G^{X}_{0}$ and $\operatorname{Aut}_{0}^{X}(M)$ are both Lie groups endowed with the compact-open topology by Propositions \ref{southbeach} and \ref{southbeach2} respectively. In addition,
$G^{X}_{0}$ is a subgroup of $\operatorname{Aut}_{0}^{X}(M)$ since $\mathfrak{g}^{X}$ is a Lie subalgebra of $\mathfrak{aut}^{X}(M)$.
Compactness of $G^{X}_{0}$ with respect to the compact-open topology follows from the Arzel\`a-Ascoli theorem because the level sets of $f$ are compact by properness of $f$ and, as we have just seen, are preserved by $G_{0}^{X}$. Being compact, $G^{X}_{0}$ is then a closed subgroup of $\operatorname{Aut}_{0}^{X}(M)$, hence is a compact Lie subgroup of $\operatorname{Aut}_{0}^{X}(M)$, with everything being relative to the compact-open topology.
\end{proof}

Finally, we can now deduce from Theorem \ref{matty} that $G^{X}_{0}$ is a \emph{maximal} compact Lie subgroup of $\operatorname{Aut}_{0}^{X}(M)$.
\begin{corollary}\label{maxx}
Let $(M,\,g,\,X)$ be a complete shrinking gradient K\"ahler-Ricci soliton with bounded Ricci curvature.
Then $G^{X}_{0}$ is a maximal compact Lie subgroup of $\operatorname{Aut}_{0}^{X}(M)$.
\end{corollary}

\begin{proof}
First note that $G^{X}_{0}$ is a compact Lie subgroup of $\operatorname{Aut}_{0}^{X}(M)$ by Lemma \ref{tannn}.
Now suppose that $G^{X}_{0}$ is not maximal in $\operatorname{Aut}_{0}^{X}(M)$. Then there exists a compact Lie subgroup $K$
of $\operatorname{Aut}_{0}^{X}(M)$ strictly containing $G_{0}$. In particular, the real dimension of
$K$ must be strictly greater than $G^{X}_{0}$. On the Lie algebra level, since the decomposition of Theorem \ref{matty} is direct,
there exists a non-zero real holomorphic vector $Z$ in the Lie algebra of $K$
not contained in $\mathfrak{g}^{X}$ yet contained in $J\mathfrak{g}^{X}$.
Since $K$ is compact, the closure of the flow of $Z$ in $K$ will define a real torus $T^{k}$ of real dimension $k$
in $K$ in which the flow of $Z$ is dense. We complete the proof as in the conclusion of the proof of \cite[Theorem 3.5.1]{paul}.

Consider the vector field $Z$. This is a real holomorphic vector field with $JZ$ Killing.
Since a shrinking soliton has finite fundamental group \cite{wyliee}, we have that $H^{1}(M)=0$.
Hence $JZ$ admits a Hamiltonian potential $u:M\to\mathbb{R}$ so that $Z=\nabla^{g}u$.
If the real dimension $k$ of $T$ is equal to one, then the orbits of $Z$ are all closed, but a gradient flow has
no non-trivial closed integral curves since $\frac{d}{dt}u(\gamma_{x}(t))=|\nabla^{g}u|^{2}\geq0$,
where $\gamma_{x}(t)$ denotes the integral curve of $Z$ with $\gamma_{x}(0)=x\in M$.
Hence $k$ is strictly greater than one. But this is impossible as well. Indeed,
let $x$ be any point of $M$ where $Z(x)\neq 0$. Then $u(\gamma_{x}(t))$ is an increasing function
of $t$ so that $u(\gamma_{x}(t))-u(x)>c$ for some constant $c>0$ say, for all $t>1$. On the other
hand, since the flow of $Z$ is dense in $T$, $\gamma_{x}(t)$ intersects any
neighbourhood of $x$ in $M$ for some $t>1$. This yields another contradiction. Thus, $G^{X}_{0}$ is maximal
in $\operatorname{Aut}_{0}^{X}(M)$, as claimed.
\end{proof}

\newpage
\subsection{The weighted volume functional}\label{weighted2}

Recall that $(M,\,g,\,X)$ is a complete shrinking gradient K\"ahler-Ricci soliton of complex dimension $n$ with complex structure $J$, K\"ahler form $\omega$, and soliton vector field $X=\nabla^g f$ for $f:M\to\mathbb{R}$ smooth, and that by assumption, we have a real torus $T$ with Lie algebra $\mathfrak{t}$ acting holomorphically, effectively, and isometrically on $(M,\,g,\,J)$ with $JX\in\mathfrak{t}$.

In order to make sense of the weighted volume functional of a shrinking gradient K\"ahler-Ricci soliton, we need to define a moment map for the action of $T$ on $M$. This comes down to showing that every element of $\mathfrak{t}$ admits a real Hamiltonian potential, as demonstrated in the next proposition.
Such a potential exists essentially because $T$ acts
by isometries and $H^{1}(M)=0$. However, a Hamiltonian potential is only defined up to a constant. Therefore a normalisation
is required to determine the potential uniquely. We normalise so that the potential lies in the kernel of a certain linear operator, precisely
the condition required to show that $JX$ is the unique critical point of the weighted volume functional.
\begin{prop}\label{tangent}
In the above situation, let $Y\in\mathfrak{t}$ so that $Y$ defines a real holomorphic $g$-Killing vector field on $M$ with $[X,\,Y]=0$.
Then there exists a unique smooth real-valued function \linebreak $u_{Y}:M\to\mathbb{R}$ with $\mathcal{L}_{JX}u_{Y}=0$
such that $\Delta_{\omega}u_{Y}+u_{Y}+\frac{1}{2}(JY)\cdot f=0$ and $du_{Y}=-\omega\lrcorner Y$.
\end{prop}

\begin{proof}
Let $Z:=-JY$. Then $Z$ is real holomorphic and $JZ$ is $g$-Killing. Since a shrinking soliton has finite fundamental group \cite{wyliee},
we have that $H^{1}(M)=0$. This implies that there exists a smooth real-valued function $u_{Y}:M\to\mathbb{R}$
 such that $Z=\nabla^{g}u_{Y}$. Then $du_{Y}\circ J=-\omega\lrcorner Z$. Let $Z^{1,\,0}=\frac{1}{2}(Z-iJZ)$. Then we have that
\begin{equation*}
\begin{split}
\omega\lrcorner Z^{1,\,0}&=\frac{1}{2}\omega\lrcorner Z-\frac{i}{2}\omega\lrcorner JZ\\
&=-\frac{1}{2}du_{Y}\circ J+\frac{i}{2}du_{Y}\\
&=\frac{i}{2}\left(du_{Y}+idu_{Y}\circ J\right)\\
&=i\bar{\partial}u_{Y},
\end{split}
\end{equation*}
i.e., $\bar{\partial}(-iu_{Y})=-\omega\lrcorner Z^{1,\,0}$.
By averaging $u_{Y}$ over the action of $T$, we may assume that
$\mathcal{L}_{JX}u_{Y}=0$. Using the Bochner formula, contracting [\eqref{krseqn}, $\lambda=1$] with $Z^{1,\,0}$ then results in
$$-i\bar{\partial}\Delta_{\omega}u_{Y}+i\bar{\partial}(Z^{1,\,0}\cdot f)=i\bar{\partial}u_{Y}.$$
In other words,
$$\bar{\partial}(\Delta_{\omega}u_{Y}+u_{Y}-Z^{1,\,0}\cdot f)=0.$$
Now, the fact that $\mathcal{L}_{JX}u_{Y}=0$ implies that
$$0=du_{Y}(JX)=g(\nabla^g u_{Y},\,JX)=g(Z,\,JX)=-g(JZ,\,X)=-df(JZ)=-(JZ)\cdot f.$$
In particular, $\Delta_{\omega}u_{Y}+u_{Y}-Z^{1,\,0}\cdot f$ is a real-valued holomorphic function, hence is
equal to a constant. By subtracting this constant from $u_{Y}$ and plugging in the definition of $Z$, we arrive at our desired normalisation of $u_{Y}$:
$$\Delta_{\omega}u_{Y}+u_{Y}+\frac{1}{2}(JY)\cdot f=0.$$
Since $u_{Y}$ is defined up to a constant, this condition determines $u_{Y}$ uniquely.
\end{proof}

With this proposition, we can now define our moment map for the action of $T$ on $M$.
\begin{definition}\label{moments}
Let $\langle\cdot\,,\,\cdot\rangle$ denote the natural pairing between $\mathfrak{t}$ and $\mathfrak{t}^{*}$.
Then we define the \emph{moment map $\mu:M\to\mathfrak{t}^{*}$ for the action of $T$ on $M$} by
$$u_{Y}(x)=\langle\mu(x),\,Y\rangle,\quad x\in M,\qquad\textrm{for all $Y\in\mathfrak{t}$},$$
where $u_{Y}$ is such that $\nabla^g u_{Y}=-JY$, $\mathcal{L}_{JX}u_{Y}=0$, and $\Delta_{\omega}u_{Y}+u_{Y}+\frac{1}{2}(JY)\cdot f=0.$
\end{definition}

We next define the weighted volume functional
for complete shrinking gradient K\"ahler-Ricci solitons.
\begin{definition}[{Weighted volume functional, cf.~\cite[equation (2.3)]{tianzhu2}}]\label{modfut}
Let $(M,\,g,\,X)$ be a complete shrinking gradient K\"ahler-Ricci soliton of complex dimension $n$ with complex structure $J$, K\"ahler form $\omega$, and with soliton vector field $X=\nabla^{g}f$ for a smooth real-valued function $f:M\to\mathbb{R}$, endowed with the holomorphic, effective, isometric action of a real torus $T$ with Lie algebra $\mathfrak{t}$ and with a compact fixed point set.
Let $\mu$ denote the moment map of the action as prescribed in Definition \ref{moments}
and assume that $JX\in\mathfrak{t}$. Let $Y\in\mathfrak{t}$ and let $u_{Y}:=\langle\mu,\,Y\rangle$ be the
Hamiltonian potential of $Y$ so that
$\mathcal{L}_{JX}u_{Y}=0$ and $\Delta_{\omega}u_{Y}+u_{Y}+\frac{1}{2}(JY)\cdot f=0$. Finally, let
$$\Lambda:=\{Y\in\mathfrak{t}:\textrm{$u_{Y}$ is proper and bounded below}\}\subseteq\mathfrak{t}.$$
Then the \emph{weighted volume functional} $F$ is defined by
\begin{equation*}\label{overweight}
F:\Lambda\longrightarrow\mathbb{R}_{>0},\qquad F(Y)=\int_{M}e^{-\langle\mu,\,Y\rangle}\omega^{n}=\int_{M}e^{-u_{Y}}\omega^{n}.
\end{equation*}
\end{definition}

The set $\Lambda$ is an open cone in $\mathfrak{t}$ which is determined by the image of $M$ under $\mu$;
see Proposition \ref{properr} for details. Since $f$ grows quadratically at infinity by Theorem \ref{theo-basic-prop-shrink}(i),
we know that it is in addition proper. Hence $JX$, which by assumption lies in $\mathfrak{t}$, lies in $\Lambda$ so that $\Lambda$ is non-empty.
Thus, by the Duistermaat-Heckman theorem (Theorem \ref{dhthm}), $F$ is seen to be well-defined.

We next list some elementary properties of $F$,
in particular the desired property that characterises $JX$ as the unique critical point of $F$. It is here where our normalisation of the Hamiltonian potentials comes into play.
\begin{lemma}[Volume minimising principle]\label{critical}
Let $(M,\,g,\,X)$ be a complete shrinking gradient K\"ahler-Ricci soliton of complex dimension $n$ with K\"ahler form $\omega$ and with soliton vector field $X=\nabla^{g}f$ for a smooth real-valued function $f:M\to\mathbb{R}$, endowed with the holomorphic, effective, isometric action of
a real torus $T$ with Lie algebra $\mathfrak{t}$. Assume that $JX\in\mathfrak{t}$ and that the Ricci curvature of
$g$ is bounded. Then:
\begin{enumerate}[label=\textnormal{(\roman{*})}, ref=(\roman{*})]
\item $F$ is strictly convex on $\Lambda$.
\item $JX$ is the unique critical point of $F$ in $\Lambda$.
\end{enumerate}
\end{lemma}

\begin{remark}\label{king}
Note that the boundedness of the Ricci curvature of $g$ guarantees here that $F$ is well-defined on $\Lambda$.
Indeed, this is clear from the Duistermaat-Heckman theorem (Theorem \ref{dhthm}) after noting that
the zero set of $X$, which contains the fixed point set of $T$ as a closed subset,
is compact by Lemma \ref{compactt}.
\end{remark}

\begin{proof}[Proof of Proposition \ref{critical}]
\begin{enumerate}[label=\textnormal{(\roman{*})}, ref=(\roman{*})]
\item Let $Y_1,\,Y_2\in\Lambda$. Then the line segment
$tY_1+(1-t)Y_2,\,t\in[0,\,1]$, is contained in $\Lambda$ because $\Lambda$ is convex, as one sees from its definition.
Moreover, by the linearity of the moment map, we have that $$u_{tY_1+(1-t)Y_2}=tu_{Y_1}+(1-t)u_{Y_2}\quad\textrm{for all $t\in[0,\,1]$}.$$ Thus, since the function $x\in\R\mapsto e^{-x}\in\R$ is strictly convex, we find that
\begin{equation*}
F(t\cdot Y_1+(1-t)\cdot Y_2)<t\cdot F(Y_1)+(1-t)\cdot F(Y_2)
\end{equation*}
for all $t\in(0,1)$, unless $Y_1=Y_2$.
\item As a strictly convex function on the convex set $\Lambda$, $F$ has at most one critical point. The claim is that this critical point is obtained at $JX$. Indeed, let $Y\in\mathfrak{t}$ and let $u_{Y}$ denote the Hamiltonian potential of $Y$ normalised so that
        $\Delta_{\omega}u_{Y}+u_{Y}+\frac{1}{2}(JY)\cdot f=0$.
Recall that $-J(JX)=\nabla^g f$ so that $$d_{JX}F(Y)=-\int_{M}u_{Y}e^{-f}\omega^{n}.$$
Let $R$ be a positive real number and let $\phi_R:M\rightarrow[0,1]$ be a cut-off function with compact support in the geodesic ball $B_g(p,\,2R)$ such that $\phi_R=1$ on $B_g(p,\,R)$ and $|\nabla\phi_R|_g\leq\frac{c}{R}$ for some $c>0$. Then using integration by parts,
we have that
\begin{equation*}
\begin{split}
\left|\int_{M}u_{Y}\phi_{R}^{2}\,e^{-f}\omega^{n}\right|&=\left|\int_{M}(\Delta_{\omega}u_{Y}+\frac{1}{2}(JY)\cdot f)\phi_{R}^{2}\,e^{-f}\omega^{n}\right|\\
&=\left|\frac{1}{2}\int_{M}(\Delta_{g}-\nabla^g f\cdot)u_{Y}\phi_{R}^{2}\,e^{-f}\omega^{n}\right|\\
&=\frac{1}{2}\left|\int_{M}g(\nabla u_{Y},\,\nabla(\phi_{R}^{2}))e^{-f}\omega^{n}\right|\\
&\leq\frac{1}{2}\left(\int_{M}|\nabla u_{Y}|^{2}e^{-f}\omega^{n}\right)^\frac{1}{2}\left(|\nabla(\phi_{R}^{2})|^{2}e^{-f}\omega^{n}\right)^{\frac{1}{2}}\\
&=\left(\int_{M}|\nabla u_{Y}|^{2}e^{-f}\omega^{n}\right)^\frac{1}{2}\left(|\nabla\phi_{R}|^{2}|\phi_{R}|^{2}e^{-f}\omega^{n}\right)^{\frac{1}{2}}\\
&\leq\frac{c^{2}}{R^{2}}\left(\int_{M}|\nabla u_{Y}|^{2}e^{-f}\omega^{n}\right)^\frac{1}{2}\left(\int_{M}e^{-f}\omega^{n}\right)^{\frac{1}{2}},\\
\end{split}
\end{equation*}
where the fact that $\nabla u_{Y}\in L^{2}(e^{-f}\omega^{n})$ follows as in the proof of Claim \ref{claim-grad-int}.
Letting $R\to+\infty$, we see that $d_{JX}F(Y)=0$, as required.
\end{enumerate}
\end{proof}

The main tool we use to compute the weighted volume functional is the Duistermaat-Heckman theorem.
The statement of this theorem and a discussion have been relegated
to Appendix \ref{heckman}. It expresses the weighted volume functional in terms of data determined by the induced action on $M$ of the element $Y\in\Lambda$. In particular, this data is independent of the metric $\omega$. Consequently, $F$ is independent of the particular shrinking gradient
K\"ahler-Ricci soliton. It is this observation that will allow us to ascertain the uniqueness of
the soliton vector field $X$ under certain assumptions. This is the content of the next subsection.

\subsection{A general uniqueness theorem}\label{generalunique}
As an application of Corollary \ref{maxx}, we prove the uniqueness statement of Theorem \ref{uunique} for the soliton vector field $X$ of a shrinking gradient K\"ahler-Ricci soliton, the precise statement of which we now recall below.
\begin{theorem}[Theorem \ref{uunique}]\label{uniquee}
Let $M$ be a non-compact complex manifold with complex structure $J$ endowed with the effective holomorphic action of a real torus $T$. Denote by $\mathfrak{t}$ the Lie algebra of $T$. Then there exists at most one element $\xi\in\mathfrak{t}$
that admits a complete shrinking gradient K\"ahler-Ricci soliton $(M,\,g,\,X)$ with bounded Ricci curvature
with $X=\nabla^g f=-J\xi$ for a smooth real-valued function $f$ on $M$.
\end{theorem}
The outline of the proof of this theorem is as follows. Suppose that $M$ admitted two soliton vector fields $X_{1}$ and $X_{2}$. Then
the maximal tori in the Lie groups $\operatorname{Aut}_{0}^{X_{i}}(M),\,i=1,\,2,$ will be conjugate to $T$ by Iwasawa's theorem \cite{iwasawa}.
After choosing an appropriate gauge, $JX_{1}$ and $JX_{2}$ will then be contained in the Lie algebra
$\mathfrak{t}$ of $T$ and both vector fields will be critical points of their respective weighted volume functional. But since the weighted volume functional
is independent of the shrinking K\"ahler-Ricci soliton by the Duistermaat-Heckman theorem, both weighted volume functionals must coincide
so that $JX_{1}=JX_{2}$ by uniqueness of the critical point.

\begin{proof}[Proof of Theorem \ref{uniquee}]
Suppose that $M$ admitted two complete shrinking gradient K\"ahler-Ricci solitons $(M,\,g_{i},\,X_{i})$ with bounded Ricci curvature with $X_{i}=\nabla^{g_{i}} f_{i}$ for $f_{i}:M\to\mathbb{R}$ smooth such that $X_{i}=-J\xi_{i}$ for $\xi_{i}\in\mathfrak{t}$ for $i=1,\,2$. Let $G^{X_{i}}_{0}$ denote the connected component of the identity of the group of holomorphic isometries of $(M,\,g_{i},\,J)$ that commute with the flow of $X_{i}$. Corollary \ref{maxx} then asserts that $G^{X_{i}}_{0}$ is a maximal compact Lie subgroup of the Lie group $\operatorname{Aut}^{X_{i}}_{0}(M)$, the connected component of the identity of the group of automorphisms of $(M,\,J)$ that commute with the flow of $X_{i}$. Denote by $T_{i}$ the maximal real torus in $G^{X_{i}}_{0}$. Then $T_{i}$ is maximal in $\operatorname{Aut}^{X_{i}}_{0}(M)$. For each $v\in\mathfrak{t}$, we have that $[v,\,\xi_{i}]=0$ so that $[v,\,X_{i}]=0$. Hence each element of $T$ commutes with the flow of $X_{i}$ and so $T$ itself is a Lie subgroup of $\operatorname{Aut}^{X_{i}}_{0}(M)$. Without loss of generality, we may assume that $T$ is maximal in $\operatorname{Aut}^{X_{i}}_{0}(M)$. Then by Iwasawa's
theorem \cite{iwasawa},
there exists an element $\alpha_{i}\in\operatorname{Aut}^{X_{i}}_{0}(M)$ such that
$\alpha_{i}T_{i}\alpha_{i}^{-1}=T$. Since $\alpha_{i}$ commutes with the flow of $X_{i}$, necessarily $d\alpha^{-1}_{i}(X_{i})=X_{i}$. Moreover, $\alpha_{i}^{*}g_{i}$ is invariant under $T$. Thus, $(M,\,\tilde{g}_{i},\,\tilde{X}_{i})$ with $\tilde{g}_{i}=\alpha_{i}^{*}g_{i}$ and $\tilde{X}_{i}=d\alpha_{i}^{-1}(X_{i})$ is a $T$-invariant shrinking gradient K\"ahler-Ricci soliton with soliton vector field $\tilde{X}_{i}=X_{i}=-J\xi_{i}$ as before. Hence, by considering this pullback, we may assume that each $(M,\,g_{i},\,X_{i})$ is invariant under $T$.

Now, by assumption we have that $\xi_{1},\,\xi_{2}\in\mathfrak{t}$. Since the corresponding Hamiltonian potentials are the
soliton potentials which themselves are proper and bounded below, we have that $\xi_{i}\in\Lambda_{i}\subset\mathfrak{t}$, $\Lambda_{i}$ here
denoting the open cone of elements of $\mathfrak{t}$ admitting Hamiltonian potentials
with respect to the K\"ahler form $\omega_{i}$ of $g_{i}$ that are proper and bounded below.
We wish to show that $\xi_{1},\,\xi_{2}\in\Lambda_{1}\cap\Lambda_{2}\neq\emptyset$. The result will then follow from an application of the Duistermaat-Heckman theorem. So let $u_{1}$ denote the Hamiltonian potential of $\xi_{1}=JX_{1}$ with respect to $g_{2}$, that is, $\nabla^{g_{2}}u_{1}=X_{1}$, and for $x\in M$, let $\gamma_{1}(t)$ denote the integral curve of $X_{1}$ through $x$ at $t=0$. Then we have that
\begin{equation*}
\begin{split}
u_{1}(\gamma_{1}(t))&=u_{1}(\gamma_{1}(0))+\int_{0}^{t}du_{1}(\dot{\gamma}_{1}(s))\,ds\\
&=u_{1}(x)+\int_{0}^{t}g_{2}(X_{1},\,X_{1})(\gamma_{1}(s))\,ds\\
&=u_{1}(x)+\int_{0}^{t}|X_{1}|^{2}_{g_{2}}(\gamma_{1}(s))\,ds.\\
\end{split}
\end{equation*}
Since $g_{1}$ has bounded Ricci curvature so that the zero set of $X_{1}$ is compact by Lemma \ref{compactt}, and since
each forward orbit of the negative gradient flow of $f_{1}$ converges to a point in the zero set of $X_{1}$ by Proposition \ref{alix}, it is clear that every point of $M$ lies on an integral curve of $X_{1}$ passing through a fixed compact set. Thus, we see that $u_{1}$ is bounded from below. For $t\in[0,\,1]$, consider the vector field $Y_{t}:=t\xi_{1}+(1-t)\xi_{2}$. The Hamiltonian
potential of $Y_{0}$ with respect to $g_{2}$ is $f_{2}$ whereas that of $Y_{1}$ is $u_{1}$. By linearity
of the moment map, the Hamiltonian potential of $Y_{t}$ with respect to $g_{2}$ is $h_{t}:=tu_{1}+(1-t)f_{2}$. Since
$u_{1}$ is bounded from below and $f_{2}$ is proper, $h_{t}$ is proper and bounded below for $t\in[0,\,1)$ so that $Y_{t}\in\Lambda_{2}$
for $t\in[0,\,1)$. In a similar manner, one can show that $Y_{t}\in\Lambda_{1}$
for $t\in(0,\,1]$. The upshot is that $Y_{t}\in\Lambda_{1}\cap\Lambda_{2}\neq\emptyset$ for $t\in(0,\,1)$ with $\xi_{1},\,\xi_{2}\in\overline{\Lambda_{1}\cap\Lambda_{2}}$.

Define a real-valued function $F$ on $[0,\,1]$ as follows: $F(t):=F_1(Y_t)$ if $t\in[0,1)$ and $F(t):=F_2(Y_t)$ if $t\in(0,1]$,
where $F_{i}$ is the weighted volume functional with respect to $\omega_{i}$.
Then $F$ is well-defined as both $F_{1}$ and $F_{2}$ are well-defined because of the Ricci curvature bound (see Remark \ref{king}) and
by the Duistermatt-Heckman formula (Theorem \ref{dhthm}) they are equal on $\Lambda_{1}\cap\Lambda_{2}$, both being equal to a quantity that depends only on the zeroes of the vector fields in $\Lambda_{1}\cap \Lambda_{2}$. Moreover, $F$ is convex and continuous on $[0,1]$ and strictly convex on $(0,1)$. Finally, observe that
\begin{eqnarray*}
F(0)=F_1(\xi_1)\leq\min_{[0,1)}F_1\leq F_1(Y_t)=F_2(Y_t)
\end{eqnarray*}
for every $t\in(0,1)$. By letting $t$ tend to $1$, one sees that $F(0)\leq F(1)$. By symmetry, one also sees that $F(1)\leq F(0)$ which implies that $F(1)=F(0)=\min_{[0,\,1]} F.$ Since $F$ is convex, $F$ must be constant on $[0,\,1]$ which contradicts the fact that $F$ is strictly convex on $(0,1)$ unless $(Y_t)_{t\in(0,1)}$ is reduced to a single point, i.e., unless $\xi_1=\xi_2$. This concludes the proof.\\
\end{proof}

\subsection{Shrinking gradient K\"ahler-Ricci solitons on $\mathbb{C}^{n}$ and $\mathcal{O}(-k)\to\mathbb{P}^{n-1}$ for $0<k<n$}\label{pclass}

Using Theorem \ref{uunique}, we are now able to classify shrinking gradient K\"ahler-Ricci solitons with bounded Ricci curvature on $\mathbb{C}^{n}$ and on the total space of the line bundle $\mathcal{O}(-k)\to\mathbb{P}^{n-1}$ for $0<k<n$ and in doing so, prove items (1) and (2) of Theorem \ref{classify}.
\begin{theorem}[Items (1) and (2) of Theorem \ref{classify}]
Let $(M,\,g,\,X)$ be a complete shrinking gradient K\"ahler-Ricci soliton with bounded Ricci curvature.

{\rm (1)} If $M=\mathbb{C}^{n}$, then up to pullback by an element of $GL(n,\,\mathbb{C})$, $(M,\,g,\,X)$ is the flat Gaussian shrinking soliton.

{\rm (2)} If $M$ is the total space of the line bundle $\mathcal{O}(-k)\to\mathbb{P}^{n-1}$ for $0<k<n$, then up to pullback by an element of $GL(n,\,\mathbb{C})$, $(M,\,g,\,X)$ is the unique $U(n)$-invariant shrinking gradient K\"ahler-Ricci soliton constructed by
Feldman-Ilmanen-Knopf \cite{FIK} on this complex manifold.
\end{theorem}

\begin{proof}
Let $f$ denote the soliton potential of $X$ so that $f:M\to\mathbb{R}$ is a smooth real-valued function with $X=\nabla^{g}f$, and let $M$ be as in item (1) or (2) of the theorem. We make no distinction as of yet. Denote the complex structure of $M$ by $J$ and let $G^{X}_{0}$ denote the connected component of the identity of the holomorphic isometries of $(M,\,J,\,g)$ that
commute with the flow of $X$. Since $g$ has bounded Ricci curvature, $G^{X}_{0}$ is a compact Lie group
by Lemma \ref{tannn}, hence the closure of the flow of $JX$ in $G^{X}_{0}$ yields the holomorphic isometric action of a real torus $T$ on $(M,\,J,\,g)$ with Lie algebra $\mathfrak{t}$ containing $JX$. Since $M$ is $1$-convex by \cite[Lemma 2.15]{Conlon} and the Ricci curvature of $g$ is bounded, Proposition \ref{sexxy} tells us that the zero set of $X$, and correspondingly the fixed point set of $T$, comprises a single point in item (1) and is contained in the zero section of the line bundle in item (2). Furthermore, Proposition \ref{alix} implies that each forward orbit of the
negative gradient flow of $f$ converges to a point in this fixed point set.
By contracting the zero section of the line bundle in item (2), we see that the action of $T$ on $M$
induces an action of $T$ on $\mathbb{C}^{n}/\mathbb{Z}_{k}$ for $k=1,\ldots,n-1$ as appropriate with fixed point set the apex,
and that this action further lifts to an action of $T$ on $\mathbb{C}^{n}$ with an isolated fixed point. The lift of $X$
to $\mathbb{C}^{n}$ then defines a holomorphic vector field on $\mathbb{C}^{n}$ with $J_{0}X\in\mathfrak{t}$, $J_{0}$ denoting the
standard complex structure on $\mathbb{C}^{n}$, and with each forward orbit of $-X$ converging to this isolated fixed point.

By \cite[Section 3.1]{vanC4}, we may choose global holomorphic coordinates $(z_{1},\ldots,z_{n})$ on $\mathbb{C}^{n}$ with respect to which the action of $T$ on $\mathbb{C}^{n}$ is linear, that is, lies in $GL(n,\,\mathbb{C})$.
These coordinates then descend to coordinates on $\mathbb{C}^{n}/\mathbb{Z}_{k}$,
then lift to coordinates on $M$ with respect to which the
action of $T$ on $M$ lies in $GL(n,\,\mathbb{C})$.
Without loss of generality, we may assume that
$T$ is maximal in $GL(n,\,\mathbb{C})$. Then we still have that $JX\in\mathfrak{t}$.
Since any two maximal tori in $GL(n,\,\mathbb{C})$ are conjugate by Iwasawa's theorem \cite{iwasawa}, there exists $\alpha\in GL(n,\,\mathbb{C})$ such that $\alpha T\alpha^{-1}$ is equal to $\{\operatorname{diag}(e^{i\eta_{1}},\ldots,e^{i\eta_{n}}):\eta_{i}\in\mathbb{R}\}$. By considering $\alpha^{*}\omega$, we can therefore assume that $JX$ lies in the Lie algebra $\mathfrak{t}$ of a torus of the form
$T=\{\operatorname{diag}(e^{i\eta_{1}},\ldots,e^{i\eta_{n}}):\eta_{i}\in\mathbb{R}\}$ acting on $M$. We will then have induced
coordinates $(\eta_{1},\ldots,\eta_{n})$ on $\mathfrak{t}$, where $(1,\,0,\ldots,0)\in\mathfrak{t}$ will generate the vector field
$\operatorname{Im}(z_{1}\partial_{z_{1}})$ on $M$, etc.

Since the fixed point set of $T$ is compact, we can now apply Theorem \ref{uunique} which tells us that there is at most one element
of $\mathfrak{t}$ that admits a complete shrinking gradient K\"ahler-Ricci soliton with bounded Ricci curvature.
On $\mathbb{C}^{n}$, we have the flat Gaussian shrinking soliton and
Feldman-Ilmanen-Knopf \cite{FIK} have constructed a complete shrinking gradient K\"ahler-Ricci soliton with bounded Ricci curvature on $M$ for $M$ as in item (2) of the theorem. In all cases, the soliton vector field $Y$ of these solitons satisfies $JY\in\mathfrak{t}$ and each is proportional to $(1,\ldots,1)$
in our coordinates on $\mathfrak{t}$. Therefore we deduce that $JX=\lambda(1,\ldots,1)$ for some $\lambda>0$ so that on $M$,
$\frac{1}{2}(X-iJX)=\lambda z_{i}\partial_{z_{i}}$. The automorphism group of $(M,\,J)$ commuting with the flow of
this vector field is precisely the Lie group $GL(n,\,\mathbb{C})$. Thus, Corollary \ref{maxx} asserts that $G^{X}_{0}$ is maximal compact in $GL(n,\,\mathbb{C})$ and so, by Iwasawa's theorem \cite{iwasawa} again, there exists $\beta\in GL(n,\,\mathbb{C})$ such that $\beta G^{X}_{0}\beta^{-1}=U(n)$. The $(1,\,1)$-form $\beta^{*}\omega$ will then be $U(n)$-invariant and by \cite[Proposition 9.3]{FIK}, the only such complete shrinking gradient K\"ahler-Ricci soliton on $M$ is the flat Gaussian shrinking soliton if $M=\mathbb{C}^{n}$ and that constructed by Feldman-Ilmanen-Knopf if $M$ is as in item (2) of the theorem.
\end{proof}

\newpage
\section{The underlying manifold of a two-dimensional shrinking gradient K\"ahler-Ricci soliton}\label{underlying}

Item (3) of Theorem \ref{classify} will result from items (1) and (2) of Theorem \ref{classify} once we establish the following theorem.
\begin{theorem}\label{blue}
Let $(M,\,g,\,X)$ be a two-dimensional complete shrinking gradient K\"ahler-Ricci soliton whose scalar curvature decays to zero at infinity.
Then $C_{0}$ is biholomorphic to $\mathbb{C}^{2}$ and $M$ is biholomorphic to either $\mathbb{C}^{2}$ or $\mathbb{C}^{2}$ blown up at one point.
\end{theorem}

The key observation in proving this theorem is that the scalar curvature of the asymptotic cone is strictly positive if the shrinking
soliton is not flat. Since we are working in complex dimension two, this allows us to identify the tangent cone at infinity
as a quotient singularity using a classification theorem of Belgun \cite[Theorem 8]{belgun} for three real dimensional Sasaki manifolds. The fact that
$M$ is a resolution of $C_{0}$ by Theorem \ref{main}, combined with the fact that the exceptional set of this resolution must
contain only $(-1)$-curves as imposed by the shrinking K\"ahler-Ricci soliton equation, then allows us to identify $M$ and $C_{0}$.

\subsection{Properties of shrinking Ricci solitons}

We begin by noting some important features of shrinking Ricci solitons that we require in this section. We have the following condition on the scalar curvature of a shrinking gradient Ricci soliton.
\begin{theorem}[{\cite{peng}}]\label{pengg}
Let $(M,\,g,\,X)$ be a complete non-compact non-flat shrinking gradient Ricci soliton with scalar curvature $R_{g}$. Then for any given point $o\in M$, there exists a constant $C>0$ such that $R_{g}(x)d_{g}(x,\,o)^{2}>C^{-1}$ wherever $d_{g}(x,\,o)>C$, where $d_{g}$ denotes the distance function with respect to $g$.
\end{theorem}
This yields the following condition on the scalar curvature of an asymptotic cone of a shrinking Ricci soliton.
\begin{corollary}\label{rpositive}
Let $(M,\,g,\,X)$ be a complete non-compact non-flat shrinking gradient Ricci soliton with tangent cone $(C_{0},\,g_{0})$ along an end. Then the scalar curvature $R_{g_{0}}$ of the cone metric $g_{0}$ is strictly positive.
\end{corollary}

\begin{proof}
The tangent cone at infinity is obtained as a Gromov-Hausdorff limit of a pointed sequence $(M,\,g_{k},\,o):=(M,\,\lambda_{k}^{-2}g,\,o)$ for $o\in M$ fixed, where $\lambda_{k}\to\infty$ as $k\to\infty$. By our asymptotic assumption, the tangent cone is unique and this process recovers the asymptotic cone $(C_{0},\,g_{0})$. Indeed, an arbitrary point $p\in C_0$ with $r(p)=r_{0}>0$ is associated with a sequence $p_{k}\to p$, where $d_{g}(o, p_{k}) = \lambda_{k}r_{0}\to\infty$ as $k\to\infty$, here $r$ denoting the radial coordinate of $g_{0}$ and $d_{g}$ denoting the distance measured with respect to $g$. In particular, we see that
$$R_{g_{0}}(p)=\lim_{k\,\to\,\infty}R_{\lambda_{k}^{-2}g}(p_{k})=\lim_{k\,\to\,\infty}\lambda_{k}^{2}R_{g}(p_{k}),$$ where $R_{\lambda_{k}^{-2}g}$ denotes the scalar curvature of the rescaled metric $\lambda_{k}^{-2}g$. Using the lower bound of Theorem \ref{pengg}, we then have that
$$R_{g_{0}}(p)=\lim_{k\,\to\,\infty}\lambda_{k}^{2}R_{g}(p_{k})>
\lim_{k\,\to\,\infty}\frac{\lambda_{k}^{2}C^{-1}}{d_{g}(o,\,p_{k})^{2}}=
\lim_{k\,\to\,\infty}\frac{\lambda_{k}^{2}C^{-1}}{(\lambda_{k}r_{0})^{2}}=\frac{1}{Cr_{0}^{2}}>0$$
for some positive constant $C$. Since $p$ is arbitrary, it follows that $R_{g_{0}}>0$ away from the apex of $C_{0}$, as claimed.
\end{proof}

\subsection{Proof of Theorem \ref{blue}}

Let $(M,\,g,\,X)$ be a complete non-compact shrinking gradient K\"ahler-Ricci soliton of complex dimension $n+1$ with quadratic curvature decay
and with tangent cone along its end the K\"ahler cone $(C_{0},\,g_{0})$ given by Theorem \ref{main}. Let $r$ denote the radial function of the cone. Then the link of the cone $\{r=1\}$, which we denote by $(S,\,g_{S})$, is a Sasaki manifold of real dimension $2n+1$ foliated by the orbits of the flow of $\xi$, the restriction of the Reeb vector field of the cone to its link.

We know from \cite[Theorem 3]{pigola} that if the scalar curvature $R_{g}$ of $g$ is zero at a point, then $(M,\,g)$ is isometric to Euclidean space.
So we henceforth assume that $R_{g}\neq0$ everywhere so that $(M,\,g)$ is non-flat. Then Corollary \ref{rpositive} tells us that the scalar curvature of the cone $R_{g_{0}}$ is positive. Next we see from Lemma \ref{scalar} that $R_{g_{S}}>2n(2n+1)$ and so it follows from Corollary \ref{morescalar} that
\begin{equation*}
R^{T}>2n(2n+1)+2n=4n(n+1).
\end{equation*}

\subsubsection*{Identification of $C_{0}$} In our case, $(M,\,g,\,X)$ is of complex dimension $2$
and the scalar curvature of $g$ decays to zero at infinity. By \cite{wang22}, the scalar curvature decay implies that the norm of
the curvature tensor of $g$ decays quadratically. Thus, the above applies with $n=1$ and we have the lower bound $R^{T}>8$.
From the classification of 3-dimensional Sasaki manifolds by Belgun \cite[Theorem 8]{belgun}, it then follows that $C_{0}$ is biholomorphic to $\C^{2}/\Gamma$ with
$\Gamma$ a finite subgroup of $U(2)$ acting freely on $\mathbb{C}^{2}\setminus\{0\}$. We next wish to show that $\Gamma=\{\operatorname{id}\}$.

Recall from Theorem \ref{main} that there is a resolution $\pi:M\to C_{0}$ of the singularity of $C_{0}$ with $d\pi(X)=r\partial_{r}$.
 Since $C_{0}$ is biholomorphic to $\mathbb{C}^{2}/\Gamma$ for $\Gamma\subset U(2)$ a finite subgroup acting freely on $\mathbb{C}^{2}\setminus\{0\}$,
it is in particular a rational singularity. It is well-known that the exceptional set of a resolution of such a singularity contains a string of $\mathbb{P}^{1}$'s \cite[Lemma 1.3]{Bri}. Since
$g$ is a shrinking K\"ahler-Ricci soliton, each of these $\mathbb{P}^{1}$'s must have self-intersection $(-1)$ by adjunction.
Moreover, since $C_{0}$ is obtained from $M$ by blowing down all of these $(-1)$-curves, $C_{0}$ must in fact be smooth at the apex so that $\Gamma=\{\operatorname{id}\}$ and $C_{0}$ is biholomorphic to $\mathbb{C}^{2}$.

\subsubsection*{Identification of $M$} It follows that $M$ is then an iterated blowup of $\mathbb{C}^{2}$ at the origin containing only $(-1)$-curves. The only iterated blowups of $\mathbb{C}^{2}$ at the origin containing complex curves of this type are $\mathbb{C}^{2}$ and $\mathbb{C}^{2}$ blown up at one point, since any further iterated blowup would introduce at least one $\mathbb{P}^{1}$ with self-intersection $(-k)$ for some $k\geq2$. The conclusion is then that $M$ must be biholomorphic to either $\mathbb{C}^{2}$ or $\mathbb{C}^{2}$ blown up at one point.

\section{Concluding remarks}\label{conclusion}

We conclude with a discussion of future directions of research emanating from the results within this paper.

\subsection{The conjectural picture}

The results on shrinking gradient K\"ahler-Ricci solitons presented here allow us to speculate on possible deeper connections between such metrics and algebraic geometry. In the compact case, Berman-Witt-Nystrom \cite{wit} gave an algebraic formula for the weighted volume functional and its derivative.
We generalize this result to the non-compact case under suitable assumptions, making use of the results of Wu \cite{Wuu}. We begin with the definition of
an anti-canonically polarised K\"ahler manifold, the underlying complex manifold of a shrinking K\"ahler-Ricci soliton.

\begin{definition}\label{jaws}
An \emph{anti-canonically polarised K\"ahler manifold} is a K\"ahler manifold $M$ admitting a K\"ahler form
$\omega$ together with a hermitian metric on $-K_{M}$ with curvature form $\Theta$ such that
\begin{equation*}
\int_{V}(i\Theta)^{k}\wedge\omega^{\dim_{\mathbb{C}}V-k}>0
\end{equation*}
for all positive-dimensional irreducible compact analytic subvarieties $V$ of $M$ and for all integers $k$ such that $1\leq k\leq \dim_{\C}V$.
\end{definition}

By \cite[Theorem 4.2]{DP}, a compact anti-canonically polarised K\"ahler manifold is a Fano manifold. Moreover, any shrinking K\"ahler-Ricci soliton naturally lives on an anti-canonically polarised K\"ahler manifold.

Under certain criteria, we can write an algebraic formula for the weighted volume functional.
\begin{prop}\label{algebraic2}
Let $(M,\,\omega)$ be a (possibly non-compact) K\"ahler manifold of complex dimension $n$ with K\"ahler form $\omega$ on which there is a Hamiltonian action of a real torus $T$ with moment map $\mu:M\to\mathfrak{t}^{*}$, where $\mathfrak{t}$ is the Lie algebra of $T$ and $\mathfrak{t}^{*}$
its dual. Assume that the fixed point set of $T$ is compact and that
\begin{enumerate}
  \item $H^{p}(M,\,\mathcal{O}(-kK_{M}))=0$ for all $p>0$ and for all $k$ sufficiently large; and that
  \item $\omega$ is the curvature form of a hermitian metric on $-K_{M}$.
\end{enumerate}
If there exists an element $\zeta_{0}\in\mathfrak{t}$ such that the component of the moment map $u_{\zeta_{0}}=\langle\mu,\,\zeta_{0}\rangle$ is proper and bounded below, then
\begin{equation}\label{algebraic}
\int_{M}e^{-\langle\mu,\,\zeta\rangle}\frac{\omega^{n}}{n!}=\lim_{k\to\infty}\frac{1}{k^{n}}\operatorname{char}H^{0}(M,\,\mathcal{O}(-kK_{M}))\left(\frac{\zeta}
{k}\right)
\end{equation}
for all $\zeta$ in an open cone $\Lambda\subset\mathfrak{t}$.
\end{prop}
In this situation, the character $\operatorname{char}H^{0}(M,\,\mathcal{O}(-kK_{M}))$ is well-defined by \cite{Wuu}.
Moreover, by \cite[Theorem 4.5]{Ohsawa}, the vanishing condition (i) holds true for any $1$-convex anti-canonically polarised K\"ahler
manifold and condition (ii) holds true for any shrinking gradient K\"ahler-Ricci soliton. In particular, if $(M,\,\omega,\,X)$ is a complete shrinking gradient K\"ahler-Ricci soliton with Ricci curvature decaying to zero at infinity, endowed with the holomorphic, effective, isometric action of a real torus $T$ with Lie algebra $\mathfrak{t}$ containing $JX$, then the above theorem applies. The volume minimising principle (Lemma \ref{critical}) then tells us that for such a soliton, the unique minimum of the weighted volume functional is obtained at $JX$.

Before we present the proof of Proposition \ref{algebraic}, it is necessary to introduce some notation. Our notation will
mostly follow \cite{Wuu}. We denote by $M^{T}$ the fixed-point set of $T$ in $M$. By assumption, this is compact. If non-empty, it is a
complex submanifold of $M$. Let $F$ be the set of connected components of $M^{T}$. Then
$M^{T}=\bigcup_{\alpha\in F}M^{T}_{\alpha}$, where $M^{T}_{\alpha}$ is the component labelled by $\alpha\in F$.
Let $n_{\alpha}=\dim_{\mathbb{C}}M_{\alpha}^{T}$ and let $N_{\alpha}\to M^{T}_{\alpha}$ be the holomorphic normal
bundle of $M_{\alpha}^{T}$ in $M$. $T$ acts on $N_{\alpha}$ preserving the base $M_{\alpha}^{T}$ pointwise. The weights
of the isotropy representation on the normal fiber remain constant within any connected component. Let $\ell$
be the integral lattice in the Lie algebra $\mathfrak{t}$ of $T$, let $\ell^{*}\subset\mathfrak{t}^{*}$ denote the dual
lattice, and let $\lambda_{\alpha,\,i}\in\ell^{*}\setminus\{0\}$ ($1\leq i\leq n-n_{\alpha}$) be the isotropy weights
on $N_{\alpha}$. The hyperplanes $(\lambda_{\alpha,\,i})^{\perp}\subset\mathfrak{t}$ cut $\mathfrak{t}$ into open polyhedral
cones called \emph{action chambers} \cite{wu}. Choose an action chamber $C$. We define $\nu_{\alpha}^{C}$ as the number of weights
$\lambda_{\alpha,\,i}\in C^{*}$, where $C^{*}$ is the dual cone in $\mathfrak{t}^{*}$ defined by $C^{*}=\{\xi\in\mathfrak{t}^{*}:\langle\xi,\,C\rangle>0\}$.
Let $N_{\alpha}^{C}$ be the direct sum of the sub-bundles corresponding to the weights $\lambda_{\alpha,\,i}\in C^{*}$.
Then $N_{\alpha}=N_{\alpha}^{C}\oplus N_{\alpha}^{-C}$. $\nu_{\alpha}^{C}$ is the rank of the holomorphic vector bundle
$N_{\alpha}^{C}$; that of $N_{\alpha}^{-C}$ is $\nu_{\alpha}^{-C}=n-n_{\alpha}-\nu_{\alpha}^{C}$.

\begin{proof}[Proof of Proposition \ref{algebraic2}]
For $k\in\mathbb{N}$ sufficiently large, the vanishing assumption (i) together with \cite[equation (3.41)]{Wuu} implies that
$$\operatorname{char}H^{0}(M,\,\mathcal{O}(-kK_{M}))=\sum_{\alpha\in F}(-1)^{n-n_{\alpha}-\nu_{\alpha}^{C}}
\int_{M^{T}_{\alpha}}\operatorname{ch}^{T}\left(\frac{-kK_{M}|_{M^{T}_{\alpha}}\otimes\det(N_{\alpha}^{-C})}{\det(1-(N_{\alpha}^{C})^{*})\otimes\det(1-N_{\alpha}^{-C})}
\right)\operatorname{td}(M^{T}_{\alpha}),$$
where, if $R$ is a finite-dimensional representation of $T$,
$$\frac{1}{\det(1-R)}:=\oplus_{m=0}^{\infty}\operatorname{Sym}^{m}(R),$$
and where $\operatorname{ch}^{T}$ denotes the equivariant Chern character.
For a fixed $\alpha\in F$, we therefore have that
\begin{equation*}
\begin{split}
&\operatorname{ch}^{T}\left(\frac{-kK_{M}|_{M_{\alpha}^{T}}\otimes\det(N_{\alpha}^{-C})}{\det(1-(N_{\alpha}^{C})^{*})\otimes\det(1-N_{\alpha}^{-C})}\right)
\operatorname{td}(M_{\alpha}^{T})\\
&=\operatorname{ch}^{T}(-kK_{M}|_{M_{\alpha}^{T}})\operatorname{ch}^{T}\left(\frac{1}{\det(1-(N_{\alpha}^{C})^{*})}\right)
\operatorname{ch}^{T}\left(\frac{1}{\det(1-N_{\alpha}^{-C})}\right)\operatorname{ch}^{T}(\det(N_{\alpha}^{-C}))
\operatorname{td}(M_{\alpha}^{T}).\\
\end{split}
\end{equation*}
Now, $$\operatorname{td}(M_{\alpha}^{T})=1+\frac{1}{2}c_{1}(-K_{M_{\alpha}^{T}})+...$$
Analysing the term $\operatorname{ch}^{T}(-kK_{M}|_{M_{\alpha}^{T}})$,
we have by adjunction that
$$K_{M}|_{M_{\alpha}^{T}}=K_{M_{\alpha}^{T}}-\det(N_{\alpha})$$
so that
$$-kK_{M}|_{M_{\alpha}^{T}}=(-kK_{M_{\alpha}^{T}})+k\det(N_{\alpha}).$$
Now $M_{\alpha}^{T}$ is fixed under the action of $T$ and so the action of $T$ on $-kK_{M_{\alpha}^{T}}$ is trivial.
The torus $T$ therefore acts on $-kK_{M}|_{M_{\alpha}^{T}}$ as multiplication by $e^{k\sum_{i=1}^{n-n_{\alpha}}\lambda_{\alpha,\,i}}$, where
recall that $n_{\alpha}$ is the dimension of $M_{\alpha}^{T}$. Thus, we have that
$$c_{1}^{T}(-kK_{M}|_{M_{\alpha}^{T}})=kc_{1}(-K_{M}|_{M_{\alpha}^{T}})+k\sum_{i=1}^{n-n_{\alpha}}\lambda_{\alpha,\,i},$$
where $c_{1}^{T}$ is the equivariant first Chern class, so that
$$\operatorname{ch}^{T}(-kK_{M}|_{M_{\alpha}^{T}})\left(\frac{\zeta}{k}\right)=e^{c_{1}^{T}(-kK_{M}|_{M_{\alpha}^{T}})\left(\frac{\zeta}{k}\right)}
=e^{kc_{1}(-K_{M}|_{M_{\alpha}^{T}})}e^{\sum_{i=1}^{n-n_{\alpha}}\lambda_{\alpha,\,i}(\zeta)}.$$

Next analysing the second term, we may write $N_{\alpha}^{C}=\oplus_{\{i\,:\,\lambda_{\alpha,\,i}\,\in\, C^{*}\}}L_{\alpha,\,i}$,
where $L_{\alpha,\,i}$ is the line subbundle of $N_{\alpha}$ with isotropy weight $\lambda_{\alpha,\,i}$.
Then we have that
$$\frac{1}{\det(1-(N_{\alpha}^{C})^{*})}=\oplus_{m=0}^{\infty}\operatorname{Sym}^{m}((N_{\alpha}^{C})^{*})=\otimes_{\{i\,:\,\lambda_{\alpha,\,i}\,\in\, C^{*}\}}\frac{1}{\det(1-L^{*}_{\alpha,\,i})}$$
so that
\begin{equation*}
\begin{split}
\operatorname{ch}^{T}\left(\frac{1}{\det(1-(N_{\alpha}^{C})^{*})}\right)&=\operatorname{ch}^{T}\left(\otimes_{\{i\,:\,\lambda_{\alpha,\,i}\,\in\, C^{*}\}}\frac{1}{\det(1-L^{*}_{\alpha,\,i})}\right)=\prod_{{\{i\,:\,\lambda_{\alpha,\,i}\,\in\, C^{*}\}}}\operatorname{ch}^{T}
\left(\frac{1}{\det(1-L^{*}_{\alpha,\,i})}\right).\\
\end{split}
\end{equation*}
Now observe that for each $i$,
\begin{equation*}
\begin{split}
\operatorname{ch}^{T}\left(\frac{1}{\det(1-L^{*}_{\alpha,\,i})}\right)&=\operatorname{ch}^{T}(\oplus_{m=0}^{\infty}\operatorname{Sym}^{m}(L^{*}_{\alpha,\,i}))=\operatorname{ch}^{T}(\oplus_{m=0}^{\infty}(L^{*}_{\alpha,\,i})^{m})\\
&=\sum_{m=0}^{\infty}(\operatorname{ch}^{T}(L^{*}_{\alpha,\,i}))^{m}=\sum_{m=0}^{\infty}\left(e^{-\lambda_{\alpha,\,i}+c_{1}(L^{*}_{\alpha,\,i})}\right)^{m}\\
&=\frac{1}{1-e^{-\lambda_{\alpha,\,i}+c_{1}(L^{*}_{\alpha,\,i})}}.
\end{split}
\end{equation*}
Hence
$$\operatorname{ch}^{T}\left(\frac{1}{\det(1-(N_{\alpha}^{C})^{*})}\right)=\prod_{\{i\,:\,\lambda_{\alpha,\,i}\,\in\, C^{*}\}}
\frac{1}{1-e^{-\lambda_{\alpha,\,i}+c_{1}(L^{*}_{\alpha,\,i})}}.$$
Consequently,
\begin{equation*}
\begin{split}
&\operatorname{ch}^{T}\left(\frac{1}{\det(1-(N_{\alpha}^{C})^{*})}\right)\left(\frac{\zeta}{k}\right)=
\prod_{\{i\,:\,\lambda_{\alpha,\,i}\,\in\, C^{*}\}}\frac{1}{1-e^{-\frac{1}{k}\lambda_{\alpha,\,i}(\zeta)-c_{1}(L_{\alpha,\,i})}}\\
&=\prod_{\{i\,:\,\lambda_{\alpha,\,i}\,\in\, C^{*}\}}\frac{1}{1-(1+(-\frac{1}{k}\lambda_{\alpha,\,i}(\zeta)-c_{1}(L_{\alpha,\,i}))+\sum_{l=2}^{\infty}\frac{1}{l!}(-\frac{1}{k}
\lambda_{\alpha,\,i}(\zeta)-c_{1}(L_{\alpha,\,i}))^{l})}\\
&=\prod_{\{i\,:\,\lambda_{\alpha,\,i}\,\in\, C^{*}\}}\frac{1}{\frac{1}{k}\lambda_{\alpha,\,i}(\zeta)+c_{1}(L_{\alpha,\,i})-\sum_{l=2}^{\infty}
\frac{1}{l!}(-\frac{1}{k}\lambda_{\alpha,\,i}(\zeta)-c_{1}(L_{\alpha,\,i}))^{l}}\\
&=\prod_{\{i\,:\,\lambda_{\alpha,\,i}\,\in\, C^{*}\}}\frac{k}{\lambda_{\alpha,\,i}(\zeta)+kc_{1}(L_{\alpha,\,i})-k\sum_{l=2}^{\infty}
\frac{1}{l!}(-\frac{1}{k}\lambda_{\alpha,\,i}(\zeta)-c_{1}(L_{\alpha,\,i}))^{l}}\\
&=\prod_{\{i\,:\,\lambda_{\alpha,\,i}\,\in\,C^{*}\}}\frac{k}
{\lambda_{\alpha,\,i}(\zeta)\left(1+\frac{kc_{1}(L_{\alpha,\,i})}{\lambda_{\alpha,\,i}(\zeta)}+\frac{k}{\lambda_{\alpha,\,i}
(\zeta)}\sum_{l=2}^{\infty}\frac{(-1)^{l+1}}{l!}(\frac{1}{k}\lambda_{\alpha,\,i}(\zeta)+
c_{1}(L_{\alpha,\,i}))^{l}\right)}.\\
\end{split}
\end{equation*}
Now,
$$\left(\frac{1}{k}\lambda_{\alpha,\,i}(\zeta)+c_{1}(L_{\alpha,\,i})\right)^{l}=c_{1}(L_{\alpha,\,i})^{l}+\frac{l}{k}c_{1}(L_{\alpha,\,i})^{l-1}
\lambda_{\alpha,\,i}(\zeta)+O(k^{-2})$$
so that
$$k\left(\frac{1}{k}\lambda_{\alpha,\,i}(\zeta)+c_{1}(L_{\alpha,\,i})\right)^{l}=
kc_{1}(L_{\alpha,\,i})^{l}+lc_{1}(L_{\alpha,\,i})^{l-1}\lambda_{\alpha,\,i}(\zeta)+O(k^{-1}).$$
Since $l\geq2$, we have that
$$\frac{k}{\lambda_{\alpha,\,i}(\zeta)}\sum_{l=2}^{\infty}\frac{(-1)^{l+1}}{l!}\left(\frac{1}{k}\lambda_{\alpha,\,i}(\zeta)+
c_{1}(L_{\alpha,\,i})\right)^{l}=O(k)c_{1}(L_{\alpha,\,i})^{2}P_{1}+c_{1}(L_{\alpha,\,i})P_{2}+O(k^{-1}),$$
where $P_{1}$ and $P_{2}$ are polynomials in $c_{1}(L_{\alpha,\,i})$. Therefore, we see that
\begin{equation*}
\begin{split}
&\operatorname{ch}^{T}\left(\frac{1}{\det(1-(N_{\alpha}^{C})^{*})}\right)\left(\frac{\zeta}{k}\right)\\
&=\prod_{\{i\,:\,\lambda_{\alpha,\,i}\,\in\, C^{*}\}}\frac{k}{\lambda_{\alpha,\,i}(\zeta)\left(1+\frac{kc_{1}(L)}{\lambda_{\alpha,\,i}(\zeta)}\right)+
O(k)c_{1}(L_{\alpha,\,i})^{2}P_{1}+c_{1}(L_{\alpha,\,i})P_{2}+O(k^{-1})}\\
&=k^{\nu_{\alpha}^{C}}\prod_{\{i\,:\,\lambda_{\alpha,\,i}\,\in\, C^{*}\}}\frac{1}{\lambda_{\alpha,\,i}(\zeta)\left(1+\frac{kc_{1}(L)}{\lambda_{\alpha,\,i}(\zeta)}\right)+
O(k)c_{1}(L_{\alpha,\,i})^{2}P_{1}+c_{1}(L_{\alpha,\,i})P_{2}+O(k^{-1})}.
\end{split}
\end{equation*}
A similar argument also shows that
\begin{equation*}
\begin{split}
&\operatorname{ch}^{T}\left(\frac{1}{\det(1-N_{\alpha}^{-C})}\right)\left(\frac{\zeta}{k}\right)\\
&=(-k)^{n-n_{\alpha}-\nu_{\alpha}^{C}}\prod_{\{i\,:\,\lambda_{\alpha,\,i}\,\in\, -C^{*}\}}\frac{1}{\lambda_{\alpha,\,i}(\zeta)\left(1+\frac{kc_{1}(L_{\alpha,\,i})}{\lambda_{\alpha,\,i}(\zeta)}\right)+
O(k)c_{1}(L_{\alpha,\,i})^{2}Q_{1}+c_{1}(L_{\alpha,\,i})Q_{2}+O(k^{-1})}
\end{split}
\end{equation*}
for polynomials $Q_{1}$ and $Q_{2}$ in $c_{1}(L_{\alpha,\,i})$.

Finally,
\begin{equation*}
\begin{split}
\operatorname{ch}^{T}(\det(N_{\alpha}^{-C}))&=e^{c_{1}^{T}(\det(N_{\alpha}^{-C}))}\\
&=e^{c_{1}(\det(N_{\alpha}^{-C}))+\sum_{\{i\,:\,\lambda_{\alpha,\,i}\,\in\,-C^{*}\}}\lambda_{\alpha,\,i}},
\end{split}
\end{equation*}
so that
$$\operatorname{ch}^{T}(\det(N_{\alpha}^{-C}))\left(\frac{\zeta}{k}\right)=
e^{c_{1}(N_{\alpha}^{-C})+\frac{1}{k}\sum_{\{i\,:\,\lambda_{\alpha,\,i}\,\in\,-C^{*}\}}\lambda_{\alpha,\,i}(\zeta)}.$$

Putting all of the above observations together, we arrive at
\begin{equation*}
\begin{split}
&\frac{1}{k^{n}}\operatorname{char}H^{0}(M,\,\mathcal{O}(-kK_{M}))\left(\frac{\zeta}{k}\right)\\
&=\sum_{\alpha\in F}\frac{(-1)^{n-n_{\alpha}-\nu^{C}_{\alpha}}}{k^{n}}\int_{M_{\alpha}^{T}}\operatorname{ch}^{T}\left(\frac{-kK_{M}|_{M_{\alpha}^{T}}\otimes\det(N_{\alpha}^{-C})}
{\det(1-(N_{\alpha}^{C})^{*})\otimes\det(1-N_{\alpha}^{-C})}\right)
\operatorname{td}(M_{\alpha}^{T})\left(\frac{\zeta}{k}\right)\\
&=\sum_{\alpha\in F}\frac{1}{k^{n_{\alpha}}}
\int_{M_{\alpha}^{T}}e^{kc_{1}(-K_{M}|_{M_{\alpha}^{T}})}e^{\sum_{i=1}^{n-n_{\alpha}}\lambda_{\alpha,\,i}(\zeta)}
e^{c_{1}(N_{\alpha}^{-C})+\frac{1}{k}\sum_{\{i:\lambda_{\alpha,\,i}\in-C^{*}\}}\lambda_{\alpha,\,i}(\zeta)}\\
&\prod_{\{i\,:\,\lambda_{\alpha,\,i}\,\in\, C^{*}\}}\frac{1}{\lambda_{\alpha,\,i}(\zeta)\left(1+\frac{kc_{1}(L_{\alpha,\,i})}{\lambda_{\alpha,\,i}(\zeta)}\right)+
O(k)c_{1}(L_{\alpha,\,i})^{2}P_{1}+c_{1}(L_{\alpha,\,i})P_{2}+O(k^{-1})}\\
&\prod_{\{i\,:\,\lambda_{\alpha,\,i}\,\in\, -C^{*}\}}
\frac{1}{\lambda_{\alpha,\,i}(\zeta)\left(1+\frac{kc_{1}(L_{\alpha,\,i})}{\lambda_{\alpha,\,i}(\zeta)}\right)+
O(k)c_{1}(L_{\alpha,\,i})^{2}Q_{1}+c_{1}(L_{\alpha,\,i})Q_{2}+O(k^{-1})}\operatorname{td}(M_{\alpha}^{T})\\
&=\sum_{\alpha\in F}\frac{1}{k^{n_{\alpha}}}
\int_{M_{\alpha}^{T}}e^{kc_{1}(-K_{M}|_{M_{\alpha}^{T}})}e^{\sum_{i=1}^{n-n_{\alpha}}\lambda_{\alpha,\,i}(\zeta)}
e^{c_{1}(N_{\alpha}^{-C})+\frac{1}{k}\sum_{\{i:\lambda_{\alpha,\,i}\in-C^{*}\}}\lambda_{\alpha,\,i}(\zeta)}
\prod_{i=1}^{n-n_{\alpha}}\frac{1}{\lambda_{\alpha,\,i}(\zeta)}\\
&\prod_{\{i\,:\,\lambda_{\alpha,\,i}\,\in\, C^{*}\}}\frac{1}{1+\frac{kc_{1}(L_{\alpha,\,i})}{\lambda_{\alpha,\,i}(\zeta)}+O(k)c_{1}(L_{\alpha,\,i})^{2}P_{1}+c_{1}(L_{\alpha,\,i})P_{2}+O(k^{-1})}\\
&\prod_{\{i\,:\,\lambda_{\alpha,\,i}\,\in\, -C^{*}\}}\frac{1}{1+\frac{kc_{1}(L_{\alpha,\,i})}{\lambda_{\alpha,\,i}(\zeta)}+
O(k)c_{1}(L_{\alpha,\,i})^{2}Q_{1}+c_{1}(L_{\alpha,\,i})Q_{2}+O(k^{-1})}\left(1+\frac{1}{2}c_{1}(-K_{M^{T}_{\alpha}})+\ldots\right)\\
&\longrightarrow_{k\to\infty}\sum_{\alpha\in F}\int_{M^{T}_{\alpha}}
e^{c_{1}(-K_{M}|_{M_{\alpha}^{T}})}e^{\sum_{i=1}^{n-n_{\alpha}}\lambda_{\alpha,\,i}(\zeta)}
\prod_{i=1}^{n-n_{\alpha}}\frac{1}{\lambda_{\alpha,\,i}(\zeta)\left(1+\frac{c_{1}(L_{\alpha,\,i})}{\lambda_{\alpha,\,i}(\zeta)}\right)},
\end{split}
\end{equation*}
where, in taking the limit, we use the fact that any integrand involving terms not of the form $k^{j}\sigma_{j}$ for $\sigma_{j}$ a real $(j,\,j)$-form
vanishes. The result now follows from an application of Theorem \ref{dhthm} making use of assumption (ii) of the proposition.
\end{proof}

Given this proposition, it is tempting to define a notion of K-stability that characterises algebraically the existence of a shrinking gradient K\"ahler-Ricci soliton
on a complete anti-canonically polarised K\"ahler manifold $M$ endowed with a complete holomorphic vector field
following the strategy as implemented in the Fano case. For this purpose, we make the following definition.

\begin{definition}\label{positivee}
Let $M$ be a quasi-projective manifold endowed with the effective holomorphic action of a real torus $T$
whose fixed point set is compact. Denote by $\mathfrak{t}$ the Lie algebra of $T$, let
$\mathcal{O}_{M}(M)$ denote the global algebraic sections of the structure sheaf of $M$, and write $$\mathcal{O}_{M}(M)=\bigoplus_{\alpha\,\in\,\mathfrak{t}^{*}}\mathcal{H}_{\alpha}$$
for the weight decomposition under the action of $T$. Then we say that a vector field $Y\in\mathfrak{t}$ on $M$ is \emph{positive} if
$\alpha(Y)>0$ for all $\alpha\in\mathfrak{t}^{*}$ such that $\mathcal{H}_{\alpha}\neq\emptyset$ and $\alpha\neq0$.
\end{definition}

\begin{remark}
If $\pi:M\to C_{0}$ is a quasi-projective equivariant resolution of a K\"ahler cone $(C_{0},\,g_{0})$
with respect to the holomorphic isometric torus action on $(C_{0},\,g_{0})$
generated by the flow of the Reeb vector field of $g_{0}$, and $g$ is a K\"ahler metric on $M$ that is asymptotic to $g_{0}$ and
with respect to which the induced torus action on $(M,\,g)$ is isometric and Hamiltonian, then in the terminology just introduced,
the weighted volume functional $F$ for $(M,\,g)$ is defined on the open cone of positive vector fields in the Lie algebra of the torus
if this open cone is non-empty. This fact follows from Theorem \ref{dhthm} after noting Theorem \ref{fme}.
\end{remark}

Roughly speaking, one considers equivariant degenerations (or test configurations) of the pair $(M,\,X)$,
where $M$ is a quasi-projective anti-canonically polarised K\"ahler manifold with complex structure $J$ endowed with the holomorphic effective action of a real torus $T$ whose fixed point set is compact, and where $X$ is a vector field on $M$ with $JX$ a positive vector field lying in the Lie algebra of $T$. Then one defines a Futaki invariant in the usual manner as the derivative of the algebraic realisation of the weighted volume functional which is given by the right-hand side of \eqref{algebraic}. Of course, one must verify that this formula is well-defined in general. One subsequently defines $(M,\,X)$ as above to be K-stable if and only if the Futaki invariant is non-negative on all test configurations and positive if and only if the test configuration is non-trivial. This then allows one to make the following conjecture generalising the Yau-Tian-Donaldson conjecture for Fano manifolds.
\begin{conj}
Let $M$ be a quasi-projective anti-canonically polarised K\"ahler manifold endowed with the holomorphic effective action of a real torus $T$ whose fixed point
set is compact. Denote by $\mathfrak{t}$ the Lie algebra of $T$ and let $X$ be a vector field on $M$ such that $JX\in\mathfrak{t}$ is a positive vector field. Then $M$ admits a complete shrinking gradient K\"ahler-Ricci soliton with soliton vector field $X$ if and only if $(M,\,X)$ is K-stable.
\end{conj}
Thus, in light of this conjecture, one may view an anti-canonically polarised K\"ahler manifold as a
``non-compact Fano manifold''. (In a similar manner, one may also define a non-compact manifold of general type, etc.)
We expect that the well-developed machinery in the study of K\"ahler-Einstein metrics may be suitably adapted
to study this conjecture. We leave this for future work.

\subsection{Open problems}

There are also various other interesting open problems that we raise here.

\begin{enumerate}[label=\arabic{enumi}.]
\item Is a complete expanding or shrinking gradient K\"ahler-Ricci soliton necessarily algebraic (or quasi-projective)?
In particular, is the canonical ring of an expanding gradient K\"ahler-Ricci soliton finitely generated? Is the anti-canonical
ring of a shrinking gradient K\"ahler-Ricci soliton finitely generated?
What we can say here is that if the curvature tensor of a shrinking gradient K\"ahler-Ricci soliton decays quadratically, or if that of an expanding gradient K\"ahler-Ricci soliton decays quadratically with derivatives, then the soliton lives on a resolution of a K\"ahler cone by Theorem \ref{main}, hence is quasi-projective by Proposition \ref{quasi2}.
\item Is there at most one complete shrinking K\"ahler-Ricci soliton for a given holomorphic vector field
 on an anti-canonically polarised K\"ahler manifold up to automorphisms of the complex structure commuting with the flow of the vector field? More speculatively, is a complete shrinking K\"ahler-Ricci soliton on such a manifold unique up to automorphisms of the complex structure? A non-compact K\"ahler manifold may admit many non-isometric complete expanding gradient K\"ahler-Ricci solitons even for a fixed holomorphic soliton vector field, as demonstrated by \cite[Theorem A]{con-der}.
\item What are the constraints on a K\"ahler cone to appear as the
tangent cone of a complete shrinking gradient K\"ahler-Ricci soliton
with quadratic curvature decay? Is the underlying complex manifold of the shrinking soliton
then determined uniquely by its tangent cone? By Theorem \ref{main}, we know that the shrinking soliton must live on a resolution of its tangent cone
that is moreover an anti-canonically polarised K\"ahler manifold. For complete expanding gradient K\"ahler-Ricci solitons with quadratic curvature decay with derivatives, we know from Corollary \ref{unique}
that a K\"ahler cone appears as the tangent cone if and only if the K\"ahler cone has a smooth canonical model
(on which the soliton lives).

\item Related to the previous question, modulo automorphisms of the complex structure,
how many shrinking gradient K\"ahler-Ricci solitons with quadratic curvature decay have a given affine cone appearing as the underlying
complex space of the tangent cone? For $\mathbb{C}^{2}$, we have shown in Theorem \ref{classify} that the answer is two;
$\mathbb{C}^{2}$ only appears as the underlying complex space of the tangent cone
of the flat Gaussian shrinking soliton on $\mathbb{C}^{2}$ and of the
$U(2)$-invariant shrinking gradient K\"ahler-Ricci soliton of
Feldman-Ilmanen-Knopf on $\mathbb{C}^{2}$ blown up at a point \cite{FIK}. In general, we expect the answer to be finitely many for any given affine cone. By Corollary \ref{unique}, the answer to this question
for complete expanding gradient K\"ahler-Ricci solitons with quadratic curvature decay with derivatives is infinitely many for any given affine cone
admitting a smooth canonical model.
\item In Corollary \ref{unique}, we have seen that when the canonical model of a K\"ahler cone is smooth,
it admits a complete expanding gradient K\"ahler-Ricci soliton. Is this also true when the canonical model is singular?\footnote{We thank John Lott for raising this question.}
 \item Let $M$ be a complete quasi-projective K\"ahler manifold endowed with the holomorphic Hamiltonian action of a real torus $T$
 with Lie algebra $\mathfrak{t}$ whose fixed point set is compact. Are the elements of $\mathfrak{t}$ admitting Hamiltonian potentials that are proper and bounded below precisely those elements in $\mathfrak{t}$ that are positive in the sense of Definition \ref{positivee}? Equivalently, does the open cone of positive vector fields in $\mathfrak{t}$ coincide with the open cone $\operatorname{int}(\mathcal{C}(\mu(M))')\subset\mathfrak{t}$ of Proposition \ref{properr}?
 These questions have an affirmative answer in the setting of asymptotically conical K\"ahler manifolds; see Theorem \ref{fme} for a precise statement.
\item Does Theorem \ref{uunique} still hold true without the assumption of bounded Ricci curvature?
\item Given a complete shrinking gradient K\"ahler-Ricci soliton $(M,\,g,\,X)$, is the
zero set of $X$ always compact? As demonstrated in Lemma \ref{compactt},
this is the case if $g$ has bounded scalar curvature.
\item Let $M$ be a complete K\"ahler manifold endowed with the holomorphic Hamiltonian
action of a real torus $T$ with Lie algebra $\mathfrak{t}$ whose fixed point set is compact.
By the Duistermaat-Heckman theorem (Theorem \ref{dhthm}), the weighted volume functional $F$ is defined on the
open cone $\Lambda$ of elements of $\mathfrak{t}$ admitting Hamiltonian potentials that are proper and bounded below.
Is $F$ necessarily proper on $\Lambda$? If so, then it would have a unique minimiser on $\Lambda$.
Properness of the volume functional on the set of normalised Reeb vector fields of a Sasaki manifold
was shown in \cite[Proposition 3.3]{he}.
\end{enumerate}

\newpage
\appendix

\section{The Duistermaat-Heckman theorem}\label{heckman}

\subsection{Statement of the theorem}\label{a}
The material in this section has been taken verbatim from various sources in the literature including \cite{berline, heckman, Yau, wu}.
We begin with the definition of a moment map which is required for the statement of the Duistermaat-Heckman theorem.
\begin{definition}
Let $(M,\,\omega)$ be a symplectic manifold and let $T$ be a real torus acting by symplectomorphisms on $(M,\,\omega)$.
Denote by $\mathfrak{t}$ the Lie algebra of $T$ and by $\mathfrak{t}^{*}$ its dual. Then we say that the action of $T$ is \emph{Hamiltonian}
if there exists a smooth map $\mu:M\to\mathfrak{t}^{*}$ such that for all $\zeta\in\mathfrak{t}$,
$-\omega\lrcorner\zeta=du_{\zeta}$, where $u_{\zeta}(x)=\langle\mu(x),\,\zeta\rangle$ for all $\zeta\in\mathfrak{t}$ and $x\in M$. We call
$\mu$ the \emph{moment map} of the $T$-action and we call $u_{\zeta}$ the \emph{Hamiltonian (potential)} of $\zeta$.
\end{definition}
Notice that $u_{\zeta}$ is invariant under the flow of $\zeta$. Indeed, we have that
$$\mathcal{L}_{\zeta}u_{\zeta}=du_{\zeta}\lrcorner\zeta=-\omega(\zeta,\,\zeta)=0.$$
Consequently, each integral curve of $\zeta$ must be contained in a level set of $u_{\zeta}$.

Now consider the Hamiltonian action of a real torus $T$ of rank $s$ on a symplectic manifold $(M,\,\omega)$
of real dimension $2n$. Identify $T$ with $(S^{1})^{s}\subset(\mathbb{C}^{*})^{s}$ and
 introduce complex coordinates $(\phi_{1},\ldots,\phi_{s})$ on $T$ via this identification. This induces coordinates
$(\eta_{1},\ldots,\eta_{s})\in\mathbb{R}^{s}$ on the Lie algebra $\mathfrak{t}$ of $T$, where $(\eta_{1},\ldots,\eta_{s})$ corresponds to the vector
$\sum_{i=1}^{s}\eta_{i}\frac{\partial}{\partial\phi_{i}},$ each $\frac{\partial}{\partial\phi_{i}}$ the vector field on $M$ induced by
the coordinate $\phi_{i}$ on $T$. For $\zeta\in\mathfrak{t}$ with coordinates $(b_{1},\ldots,b_{s})$ say, the flow on $M$ generated by $\zeta$
will have a fixed point set $M_{0}(\zeta)$ corresponding to the zero set of the vector field $\zeta$. This set has the following properties.
\begin{prop}[{\cite[Proposition 7.12]{berline}}]\label{finished}
The connected components $\{F_{i}\}$ of $M_{0}(\zeta)$ are smooth submanifolds of $M$. The dimensions of different connected components do not have
to be the same. The normal bundle $\mathcal{E}_{i}$ of $F_{i}$ in $M$ are orientable vector bundles with even-dimensional fibres.
\end{prop}

For a disconnected component $F$ of $M_{0}(\zeta)$ of real codimension $2k$ in $M$, let $\iota:F\to M$ denote the inclusion.
Then $\iota^{*}\omega$ is a symplectic form on $F$
so that $F$ is a symplectic submanifold of $M$. The normal bundle $\mathcal{E}$ of $F$ in $M$ has the structure of a symplectic
vector bundle and will have real dimension $2k$. We denote this induced symplectic form on $\mathcal{E}$ by $\tau$.
The flow of $\zeta$ will generate a fibre-preserving linear action $L\zeta:\mathcal{E}\to\mathcal{E}$ on $\mathcal{E}$ which is
an automorphism of $\mathcal{E}$ leaving $\tau$ invariant in the infinitesimal sense. We introduce an almost complex structure $I:\mathcal{E}\to\mathcal{E}$, i.e.,
an automorphism of $\mathcal{E}$ such that $I^{2}=-\operatorname{id}$, commuting with $L\zeta$ and compatible with
$\tau$ in the sense that $\tau(I\cdot,\,\cdot)$ defines an inner product on $\mathcal{E}$. This gives $\mathcal{E}$ the structure of
a complex vector bundle over $F$ with $L\zeta$ an automorphism of $\mathcal{E}$ preserving the complex structure.

Next, denote by $u_{1},\ldots,u_{R}\in\mathbb{Z}^{s}\subset\mathfrak{t}^{*}$ the weights of the induced representation of $\mathfrak{t}$ on $\mathcal{E}$.
Then we have a direct sum decomposition of vector bundles
$$\mathcal{E}=\bigoplus_{\lambda=1}^{R}\mathcal{E}_{\lambda},$$
where
$$\mathcal{E}_{\lambda}:=\{v\in\mathcal{E}:(e^{i\eta_{1}},\ldots,e^{i\eta_{s}})\cdot v=u_{\lambda}(\eta_{1},\ldots,\eta_{s})Iv\quad\textrm{for all $(\eta_{1},\ldots,\eta_{n})\in\mathfrak{t}$}\}.$$
Each $\mathcal{E}_{\lambda}$ is a vector bundle of rank $2n_{\lambda}$ say. Clearly we must have $k=\sum_{\lambda=1}^{R}n_{\lambda}$.
Consider now the complex vector bundle $\mathcal{E}^{1,\,0}$ of complex dimension $k$ endowed with the action of $L\zeta$ extended by $\mathbb{C}$-linearity.
Then we have an induced decomposition
$$\mathcal{E}^{1,\,0}=\bigoplus_{\lambda=1}^{R}\mathcal{E}^{1,\,0}_{\lambda},$$
where $L\zeta$ acts on the $\lambda$th factor by $iu_{\lambda}(b_{1},\ldots,b_{n})$, and so the action of $L\zeta$ on $\mathcal{E}^{1,\,0}$ will take the form
$$L\zeta=i\operatorname{diag}(1_{n_{1}}u_{1}(b),\ldots,1_{n_{R}}u_{R}(b)),$$
where $1_{n_{\lambda}}$ denotes the $n_{\lambda}\times n_{\lambda}$ identity matrix and $b=(b_{1},\ldots,b_{s})$ are the
coordinates of $\zeta$. Thus,
$$\det\left(\frac{L\zeta}{i}\right)=\prod_{\lambda=1}^{R}u_{\lambda}(b)^{n_{\lambda}}.$$
Note that this is homogeneous of degree $k$ in $b$.

We next choose any $L\zeta$-invariant connection on $\mathcal{E}^{1,\,0}$ with curvature matrix $\Omega$.
Finally, for a polyhedral set $U$ of a vector space $V$, we define the \emph{asymptotic cone}
$$\mathcal{C}(U):=\{v\in V:\textrm{there is a $v_{0}\in V$ such that $v_{0}+tv\in U$ for $t>0$ sufficiently large}\},$$
and for a subset $W$ of $V$, we define the \emph{dual cone}
$W^{\prime}:=\{\alpha\in V^{*}:\alpha(W)\subseteq\mathbb{R}_{\geq0}\}$. Now we can state the Duistermaat-Heckman theorem.
\begin{theorem}[{The Duistermaat-Heckman theorem \cite[Theorem 2.2]{wu}}]\label{dhthm}
Let $(M,\,\omega)$ be a (possibly non-compact) symplectic manifold of real dimension $2n$ with symplectic form $\omega$ on which there is a Hamiltonian
action of a real torus $T$ with moment map $\mu:M\to\mathfrak{t}^{*}$, where $\mathfrak{t}$ is the Lie algebra of $T$ and $\mathfrak{t}^{*}$
its dual. Assume that the fixed point set of $T$ is compact. If there exists
an element $\zeta_{0}\in\mathfrak{t}$ such that the component of the moment map $u_{\zeta_{0}}=\langle\mu,\,\zeta_{0}\rangle$ is proper and bounded below, then
\begin{equation}\label{dh}
\int_{M}e^{-\langle\mu,\,\zeta\rangle}\frac{\omega^{n}}{n!}=\sum_{F\in M_{0}(\zeta)}\int_{F}\frac{e^{-\iota^{*}\langle\mu,\,\zeta\rangle}e^{\iota^{*}\omega}}
{\det\left(\frac{L\zeta-\Omega}{2\pi i}\right)}
\end{equation}
for all $\zeta$ in the open cone $\operatorname{int}(\mathcal{C}(\mu(M))^{\prime})\subset\mathfrak{t}$, where the sum on the right-hand side
is taken over the connected components $F$ of the zero set $M_{0}(\zeta)$ of $\zeta$.
\end{theorem}

Under the assumptions on the moment map $\mu$ as in the theorem, $\mu(M)$ is a proper polyhedral set in $\mathfrak{t}^{*}$
and the elements of $\operatorname{int}(\mathcal{C}(\mu(M))^{\prime})\subset\mathfrak{t}$ are characterised as follows.
\begin{prop}[{\cite[Proposition 1.4]{wu}}]\label{properr}
Under the assumptions on $T$ and $\mu$ as in Theorem \ref{dhthm}, $u_{\zeta}=\langle\mu,\,\zeta\rangle$ is proper if
 and only if $\zeta\in\pm\operatorname{int}(\mathcal{C}(\mu(M))^{\prime})
\subset\mathfrak{t}$. Moreover, if $\zeta\in\operatorname{int}(\mathcal{C}(\mu(M))^{\prime})
\subset\mathfrak{t}$, then $u_{\zeta}(M)=[m_{\zeta},\,+\infty)$ for a suitable $m_{\zeta}\in\mathbb{R}$.
\end{prop}
\noindent That is, elements of $\operatorname{int}(\mathcal{C}(\mu(M))^{\prime})$ are precisely those elements of
$\mathfrak{t}$ whose Hamiltonian is proper and bounded below. Notice that this cone is non-empty because
it contains $\zeta_{0}$ by assumption. Then for each $\zeta\in\operatorname{int}(\mathcal{C}(\mu(M))^{\prime})$,
each connected component of the zero set $M_{0}(\zeta)$ of $\zeta$ must be compact because $u_{\zeta}=\langle\mu,\,\zeta\rangle$ is proper, and
moreover, it must contain a fixed point of the torus action by \cite[Proposition 1.2]{wu}. Hence, since the fixed point set of $T$ is
assumed to be compact in Theorem \ref{dhthm}, the sum on the right-hand side of \eqref{dh} is over a finite set
and so is itself finite for all such $\zeta$.

Now, the sum on the right-hand side of \eqref{dh} is over each connected component $F$ of the zero set $M_{0}(\zeta)$ of $\zeta$. The determinant
is a $k\times k$ determinant and should be expanded formally into a differential form of
mixed degree. Moreover, the inverse is understood to mean one should expand this
formally in a Taylor series, as is standard in index theory. We next study the right-hand side of \eqref{dh} in more detail.

Under the decomposition $$\mathcal{E}^{1,\,0}=\bigoplus_{\lambda=1}^{R}\mathcal{E}^{1,\,0}_{\lambda},$$
let $\Omega_{\lambda}$ be the component of the curvature matrix of $\Omega_{\mathcal{E}}$ corresponding to $\mathcal{E}_{\lambda}$. Then
we have that
\begin{equation*}
\begin{split}
\det\left(\frac{L\zeta-\Omega}{2\pi i}\right)&=\det\left(\frac{L\zeta}{2\pi i}\right)\det(1-(L\zeta)^{-1}\Omega)\\
&=\det\left(\frac{L\zeta}{2\pi i}\right)\prod_{\lambda=1}^{R}\det(1-(L\zeta)^{-1}\Omega_{\lambda}).\\
\end{split}
\end{equation*}
Fix one of the bundles $\mathcal{E}_{\lambda}$. Then
$$\det(1-(L\zeta)^{-1}\Omega_{\lambda})=\det\left(1+wi\Omega_{\lambda}\right)=\sum_{a\geq0}c_{a}(\mathcal{E}_{\lambda})w^{a}\in H^{*}(F,\,\mathbb{R}),$$
where $w=\frac{1}{u_{\lambda}(b)}$ and $c_{a}(\mathcal{E}_{\lambda})$ are the Chern classes of $\mathcal{E}_{\lambda}$ for $0\leq a\leq n_{\lambda}$
with $c_{0}=1$. Thus,
$$\det\left(\frac{L\zeta-\Omega}{2\pi i}\right)=\prod_{\lambda=1}^{R}u_{\lambda}(b)^{n_{\lambda}}\left(\sum_{a\geq0}c_{a}(\mathcal{E}_{\lambda})w^{a}\right).$$

In particular, if $F$ is an isolated fixed point, in which case $k=n$ and $\mathcal{E}$ is the trivial
bundle, then we may write the $n$ (possibly indistinct) weights as $u_{1},\ldots,u_{n}$. The Chern classes
and the measure $e^{\iota^{*}\omega}$ contribute non-trivially, and we arrive at the
contribution
$$e^{-\iota^{*}\langle\mu,\,\zeta\rangle}\prod_{\lambda=1}^{R}\frac{1}{u_{\lambda}(b)^{n_{\lambda}}}$$
of an isolated fixed point to the Duistermaat-Heckman formula.

We next wish to sketch the proof of Theorem \ref{dhthm}. Before we do so however, we must first discuss invariant forms on a symplectic manifold.

\subsection{Invariant forms}\label{b}

Consider a symplectic manifold $(M,\,\omega)$ of real dimension $2n$ endowed with the Hamiltonian action of a real torus $T$.
For $\zeta$ in the Lie algebra $\mathfrak{t}$ of $T$, denote by
$\Omega^{k}_{\zeta}(M)$ the space of smooth $k$-forms on $M$ which are invariant under the flow of $\zeta$, i.e.,
$\alpha\in\Omega^{k}_{\zeta}(M)$ if and only if $\mathcal{L}_{\zeta}\alpha=0$.
The wedge product of two invariant forms is also invariant, therefore we have an
algebra $\Omega^{*}_{\zeta}(M)$ of invariant forms on $M$. We define the \emph{equivariant derivative} $d_{\zeta}$ on
$\Omega^{*}_{\zeta}(M)$ by $$d_{\zeta}\alpha=d\alpha-\alpha\lrcorner\zeta.$$
This derivative has the properties that $d_{\zeta}^{2}=0$ and
$$d_{\zeta}(\alpha\wedge\beta)=d_{\zeta}\alpha\wedge\beta+(-1)^{p}\alpha\wedge d_{\zeta}\beta$$
for $\alpha$ a $p$-form and $\beta$ another differential form.

For $\alpha\in\Omega^{*}_{T}(M)$, we can write $$\alpha=\alpha_{[0]}+\alpha_{[1]}+\ldots+\alpha_{[2n]}$$
with $\alpha_{[i]}$ a differential form of degree $i$ in $\Omega^{*}_{T}(M)$. Then integration of invariant forms is defined by integrating over
the highest degree part of the form, i.e.,
$$\int:\Omega^{*}_{T}(M)\to\mathbb{R},\qquad\int\alpha:=\int_{M}\alpha_{[2n]}.$$
This leads to a version of Stokes' theorem for invariant forms: if an invariant form $\alpha$ is $d_{\zeta}$-exact, i.e.,
if $\alpha=d_{\zeta}\beta$ for another form $\beta$, then $\alpha_{[2n]}=d\beta_{[2n-1]}$, since contracting with $\zeta$ decreases the degree of a form.
We then have that
$$\int_{M}\alpha:=\int_{M}\alpha_{[2n]}=\int_{M}d\beta_{[2n-1]}=\int_{\partial M}\beta_{[2n-1]}.$$

Recall that $M_{0}(\zeta)$ denotes the zero locus on $M$ of $\zeta\in\mathfrak{t}$.
\begin{lemma}\label{exact}
Let $\alpha\in\Omega_{\zeta}^{*}(M)$ be $d_{\zeta}$-closed. Then $\alpha_{[2n]}$ is exact
on $M\setminus M_{0}(\zeta)$.
\end{lemma}

\begin{proof}
Let $\theta$ be a one-form on $M\setminus M_{0}(\zeta)$ such that
\begin{equation}\label{hooray}
\mathcal{L}_{\zeta}\theta=0\quad\textrm{and}\quad\theta\lrcorner\zeta\neq0.
\end{equation}
Such a one-form can be constructed explicitly. Indeed, let $g$ be a $T$-invariant Riemannian metric on $M$, let $\tilde{\zeta}$ be
any non-zero positive smooth function times $\zeta$, and define
$$\theta(v)=g(\tilde{\zeta},\,v)\quad\textrm{for any vector field $v$ on $M$.}$$
This is well-defined on $M\setminus M_{0}(\zeta)$ as $\zeta$, hence $\tilde{\zeta}$, are non-zero on this set, and is easily seen to satisfy \eqref{hooray}.
We can then invert $d_{\zeta}\theta$ on $M\setminus M_{0}(\zeta)$ using a geometric series:
$$(d_{\zeta}\theta)^{-1}=\frac{1}{(d\theta-\theta\lrcorner\zeta)}=\frac{1}{(\theta\lrcorner\zeta)((\theta\lrcorner\zeta)^{-1}d\theta-1)}
=-(\theta\lrcorner\zeta)^{-1}-(\theta\lrcorner\zeta)^{-2}d\theta-(\theta\lrcorner\zeta)^{-3}(d\theta)^{2}-\cdots$$
Note that this geometric series is finite because the $(d\theta)^{k}$ vanish if $2k>2n=\dim M$, and we have that
$$d_{\zeta}\theta\wedge(d_{\zeta}\theta)^{-1}=1.$$
Applying $d_{\zeta}$ to this yields
$$d_{\zeta}\theta\wedge d_{\zeta}((d_{\zeta}\theta)^{-1})=0.$$
Further taking the wedge product with $(d_{\zeta}\theta)^{-1}$ on the left then leaves us with
$$d_{\zeta}[(d_{\zeta}\theta)^{-1}]=0.$$
Define $\nu$ by $$\nu:=\theta\wedge(d_{\zeta}\theta)^{-1}\wedge\alpha.$$
Then since $d_{\zeta}\alpha=0$ by assumption, we have that
\begin{equation*}
\begin{split}
d_{\zeta}\nu&=d_{\zeta}\theta\wedge(d_{\zeta}\theta)^{-1}\wedge\alpha=\alpha.
\end{split}
\end{equation*}
Taking the highest degree part of each side of this equality, we obtain the result.
\end{proof}

\subsection{Sketch of the proof of Theorem \ref{dhthm}}\label{c}
Since the left-hand side of \eqref{dh} is analytic on $\operatorname{int}(\mathcal{C}(\mu(M))^{\prime})\subset\mathfrak{t}$,
it suffices to prove \eqref{dh} for rational elements in this open cone. So let $\zeta\in\operatorname{int}(\mathcal{C}(\mu(M))^{\prime})$ be rational
and recall that the zero set $M_{0}(\zeta)$ of $\zeta$ is compact because the fixed point set of $T$ is compact by assumption; see the discussion after
Proposition \ref{properr}. Write $\zeta=t\eta$ for some integral point
$\eta\in\operatorname{int}(\mathcal{C}(\mu(M))^{\prime})$ and some $t>0$ and let $H:=\langle\mu,\,\eta\rangle$ denote the
Hamiltonian of $\eta$ which serves as a moment map of the induced $S^{1}$-action of $\{e^{i\eta}\}$ on $M$. Recall that
$H$ is a proper function bounded from below and so must tend to infinity as $x\to\infty$ in $M$.

Next, observe that $$e^{\omega-tH}=e^{-tH}\left(1+\omega+\ldots+\frac{\omega^{n}}{n!}\right)\in\Omega^{*}_{\zeta}(M)$$
and
$$d_{\zeta}e^{\omega-tH}=e^{\omega-tH}(d_{\zeta}(\omega-tH))=e^{\omega-tH}(-d(tH)-\omega\lrcorner\zeta)
=te^{\omega-tH}(-dH-\omega\lrcorner\eta)=0$$
so that $e^{\omega-tH}$ is $d_{\zeta}$-closed. An immediate consequence of Lemma \ref{exact} is
therefore that $e^{-tH}\frac{\omega^{n}}{n!}$ is exact off of the zero set $M_{0}(\zeta)$ of $\zeta$. Indeed, fix a $T$-invariant metric $g$ on $M$. Then
tracing through the proof of Lemma \ref{exact}, we see that $$e^{-tH}\frac{\omega^{n}}{n!}=d\nu_{[2n-1]}\quad\textrm{where
$\nu=\theta\wedge(d_{\zeta}\theta)^{-1}\wedge e^{\omega-tH}$ and $\theta=g(\tilde{\zeta},\,\cdot)$},$$
$\tilde{\zeta}$ here denoting any non-zero positive function times $\zeta$. We take $\tilde{\zeta}=\frac{\eta}{g(\eta,\,\eta)}$ in what follows.

Let $F$ denote each of the connected components of $M_{0}(\zeta)$ and recall that each is a smooth submanifold of $M$. Using the exponential map
of the $T$-invariant metric $g$ on $M$, we obtain a diffeomorphism $\psi$
from a neighbourhood $U$ of the zero section of the normal bundle $\mathcal{E}$ of $F$ in $E$ onto a neighbourhood $\psi(U)$ of
$F$ in $M$. For $\epsilon>0$, denote by $B_{\epsilon}$ the $\epsilon$-ball bundle in $\mathcal{E}$
and by $S_{\epsilon}$ its boundary. Since $M_{0}(\zeta)$
is compact and
$H(x)\to+\infty$ as $x\to\infty$ in $M$, we have by Stokes' theorem that
\begin{equation}\label{stokers}
\begin{split}
\int_{M}e^{-tH}\frac{\omega^{n}}{n!}&=\lim_{a\to+\infty}
\lim_{\epsilon\to 0}\int_{H^{-1}((-\infty,\,a])\setminus\cup_{F\in M_{0}(\zeta)}\psi(B_{\epsilon})}e^{-tH}\frac{\omega^{n}}{n!}\\
&=\lim_{a\to+\infty}\lim_{\epsilon\to 0}\int_{H^{-1}((-\infty,\,a])\setminus\cup_{F\in M_{0}(\zeta)}\psi(B_{\epsilon})}d\nu_{[2n-1]}\\
&=\lim_{\epsilon\to 0}\sum_{F\in M_{0}(\zeta)}\int_{\psi(S_{\epsilon})}\nu_{[2n-1]}+\lim_{a\to+\infty}\int_{H^{-1}(a)}\nu_{[2n-1]},\\
\end{split}
\end{equation}
where we recall the fact that $H$ is proper and that $a$ is a regular value of $H$ for all $a$ sufficiently large by \cite[Proposition 1.2]{wu}
because the fixed point set of the torus action is compact, so that $H^{-1}(a)$ is a smooth compact submanifold of $M$
for all such values of $a$.

Now, we have that
\begin{equation*}
\begin{split}
\nu_{[2n-1]}&=-e^{-tH}\theta\wedge\sum_{j=0}^{n-1}\frac{(d\theta)^{j}}{(\theta\lrcorner\zeta)^{j+1}}\wedge\frac{\omega^{n-1-j}}{(n-1-j)!}\\
&=-e^{-tH}\theta\wedge\sum_{j=0}^{n-1}\frac{(d\theta)^{j}}{t^{j+1}}\wedge\frac{\omega^{n-1-j}}{(n-1-j)!},
\end{split}
\end{equation*}
where again $\theta=\frac{g(\eta,\,\cdot)}{g(\eta,\,\eta)}$. For one connected component $F\in M_{0}(\zeta)$ of codimension $k$ say,
as in the proof of the Duistermaat-Heckman formula in the compact case \cite{heckman}, the only summand contributing to $\int_{\psi(S_{\epsilon})}\nu_{[2n-1]}$ in
the limit as $\epsilon\to0$ is the one with $j=k-1$. Therefore, computing as in \cite{heckman}, one sees that
\begin{equation}\label{amen}
\begin{split}
\lim_{\epsilon\to 0}\int_{\psi(S_{\epsilon})}\nu_{[2n-1]}
&=-\lim_{\epsilon\to 0}\int_{F}e^{-tH}\theta\wedge\frac{(d\theta)^{k-1}}{t^{k}}\wedge\frac{\omega^{n-k}}{(n-k)!}\\
&=-\lim_{\epsilon\to 0}\int_{F}e^{-tH}\tilde{\theta}\wedge(d\tilde{\theta})^{k-1}\wedge\frac{\omega^{n-k}}{(n-k)!}=\frac{e^{-\iota^{*}(tH)}e^{\iota^{*}\omega}}
{\det\left(\frac{L\zeta-\Omega}{2\pi i}\right)},
\end{split}
\end{equation}
where $\tilde{\theta}=\frac{g(\zeta,\,\cdot)}{g(\zeta,\,\zeta)}$ on the second line.

We finally deal with the term $\lim_{a\to+\infty}\int_{H^{-1}(a)}\nu_{[2n-1]}$. Since
for all $a$ sufficiently large $a$ is a regular value of $H$, the moment map of the $S^{1}$-action $\{e^{i\eta}\}$, the set
$H^{-1}(a)$ is a connected compact submanifold of $M$ on which the $S^{1}$-action is locally free. Let $M_{a}=H^{-1}(a)/S^{1}$ be the symplectic quotient
with canonical symplectic form $\omega_{a}$. The preimage $H^{-1}(a)\to M_{a}$ then has the structure of a orbi-bundle over $M_{a}$.
Moreover, since $a$ is a regular value of $H$, there exists a number $\delta>0$ such that $H^{-1}((a-\delta,\,a+\delta))$ is diffeomorphic to $H^{-1}(a)\times(-\delta,\,\delta)$. With respect to this diffeomorphism, the symplectic form $\omega$ on $H^{-1}(a)\times(-\delta,\,\delta)$ is, up
to an exact form, equal to $$\alpha\wedge dH-(H-a)F_{a}+\omega_{a}$$
for one (and hence any) connection $1$-form $\alpha$ on the orbibundle $H^{-1}(a)\to M_{a}$ with curvature
$F_{a}$. Now, when restricted to $H^{-1}(a)$, one can verify that $\omega|_{H^{-1}(a)}=\omega_{a}$,
$\theta|_{H^{-1}(a)}=:\alpha$ defines a
connection $1$-form, $d\theta|_{H^{-1}(a)}=:F_{a}$ is the curvature form of $\alpha$
 so that $d_{\zeta}\theta|_{H^{-1}(a)}=F_{a}-t$. And so we have that
\begin{equation*}
\begin{split}
\int_{H^{-1}(a)}\nu_{[2n-1]}&=\int_{H^{-1}(a)}\theta\wedge(d_{\zeta}\theta)^{-1}\wedge e^{\omega-tH}=\int_{H^{-1}(a)}\alpha\wedge(F_{a}-t)^{-1}\wedge e^{\omega_{a}-ta}\\
&=-\frac{e^{-ta}}{t}\int_{M_{a}}\left(1-\frac{F_{a}}{t}\right)^{-1}\wedge e^{\omega_{a}}=-\sum_{j=1}^{n-1}\frac{e^{-ta}}{t^{j+1}}\int_{M_{a}}\frac{\omega_{a}^{n-1-j}}{(n-1-j)!}\wedge F_{a}^{j}.
\end{split}
\end{equation*}
As $a\to+\infty$, the cohomology class of $\omega_{a}$ depends linearly on $a$ \cite{heckman, wu2}, whereas that of $F_{a}$ remains
fixed since the topology of the bundle $H^{-1}(a)\to M_{a}$ does not change as $a$ runs through a set of regular values.
So the integral over $M_{a}$ here is a polynomial in $a$. Consequently, $\int_{H^{-1}(a)}\nu_{[2n-1]}\to\nolinebreak 0$ exponentially as
$a\to+\infty$. Thus, combining this fact with \eqref{stokers} and \eqref{amen}, and noting that $tH=\langle\mu,\,\zeta\rangle$,
we arrive at the desired conclusion.

\subsection{Examples}\label{d}
We next consider some simple examples and see what formula \eqref{dh} yields for the weighted volume functional $F$.
\begin{example}\label{one}
Let $M=\mathbb{C}^{n}$ and consider the action of
the maximal torus $T=\{\operatorname{diag}(e^{i\eta_{1}},\ldots,e^{i\eta_{n}}):\eta_{i}\in\mathbb{R}\}$
in $GL(n,\,\mathbb{C})$ acting on $M$ with induced
coordinates $(\eta_{1},\ldots,\eta_{n})$ on the Lie algebra $\mathfrak{t}$ of $T$, where $(1,\,0,\ldots,0)\in\mathfrak{t}$ generates the vector field
$\operatorname{Im}(z_{1}\partial_{z_{1}})$ on $M$, etc. The fixed point set of $T$ is clearly compact.

For any $Y\in\{(\eta_{1},\ldots,\eta_{n})\in\mathfrak{t}:\eta_{i}>0\}$ and
for any $T$-invariant complete shrinking gradient K\"ahler-Ricci soliton $(M,\,\omega,\,X)$ with $X=\nabla^g f$ for $f:M\to\mathbb{R}$ smooth,
let $u_{Y}$ be the Hamiltonian potential of $Y$ normalised as in Definition \ref{moments} so that in particular, $\Delta_{\omega}u_{Y}+u_{Y}+\frac{1}{2}(JY)\cdot f=0$. Then
$$-u_{Y}(0)=(\Delta_{\omega}u_{Y})(0)+\underbrace{\frac{1}{2}((JY)\cdot f)(0)}_{=0}=\operatorname{div}(Y)=\sum_{j}\eta_{j},$$
and so the Duistermaat-Heckman theorem yields
\begin{equation*}
\begin{split}
F(\eta_{1},\ldots,\eta_{n})&=\int_{M}e^{-u_{Y}}\omega^{n}=\Pi_{j}\eta^{-1}_{j}\cdot e^{\sum_{j}\eta_{j}}.
\end{split}
\end{equation*}
Since this function is symmetric in its components, its unique critical point must be of the form $\lambda(1,\,\ldots,1)$ for some $\lambda>0$. It is
then easy to show that $\lambda=1$. The corresponding shrinking gradient K\"ahler-Ricci soliton is the flat Gaussian shrinking soliton on $\mathbb{C}^{n}$.
\end{example}

\begin{example}\label{two}
Let $M$ be $\mathbb{C}^{2}$ blown up at the origin and again consider the action of
the maximal torus $T=\{\operatorname{diag}(e^{i\eta_{1}},\,e^{i\eta_{2}}):\eta_{1},\,\eta_{2}\in\mathbb{R}\}$
in $GL(2,\,\mathbb{C})$ acting on $M$, with induced
coordinates $(\eta_{1},\,\eta_{2})$ on the Lie algebra $\mathfrak{t}$ of $T$, where $(1,\,0)\in\mathfrak{t}$ generates the vector field
$\operatorname{Im}(z_{1}\partial_{z_{1}})$ on $M$, etc. In this case, the weighted volume functional is given by
$$F:\{(\eta_{1},\,\eta_{2})\in\mathfrak{t}:\eta_{1},\,\eta_{2}>0\}\rightarrow\mathbb{R}_{>0},\quad
F(\eta_{1},\,\eta_{2})= \left\{
\begin{array}{rl}
\frac{e^{\eta_{1}}}{(\eta_{1}-\eta_{2})\eta_{2}}+\frac{e^{\eta_{2}}}{(\eta_{2}-\eta_{1})\eta_{1}} & \text{if } \eta_{1}\neq\eta_{2},\\
e^{\eta_{1}}(\eta_{1}^{-1}+\eta_{1}^{-2}) & \text{if } \eta_{1}=\eta_{2}.\\
\end{array} \right.$$
Again by symmetry, the unique critical point of $F$ here must have $\eta_{1}=\eta_{2}$, and a computation shows that $\eta_{1}=\eta_{2}=\sqrt{2}$ in this
case. The corresponding shrinking gradient K\"ahler-Ricci soliton is that of Feldman-Ilmanen-Knopf \cite{FIK} on this space.
\end{example}

\begin{example}\label{three}
More generally, let $M$ be the total space of the line bundle $\mathcal{O}(-k)$ over $\mathbb{P}^{n-1}$
for $0<k<n$ and consider the induced action of
the maximal torus $T=\{\operatorname{diag}(e^{i\eta_{1}},\ldots,e^{i\eta_{n}}):\eta_{i}\in\mathbb{R}\}$
in $GL(n,\,\mathbb{C})$ acting on $M$, with induced
coordinates $(\eta_{1},\ldots,\eta_{n})$ on the Lie algebra $\mathfrak{t}$ of $T$, where $(1,\,0,\ldots,0)\in\mathfrak{t}$ generates the vector field
$\operatorname{Im}(z_{1}\partial_{z_{1}})$ on $M$, etc. In this case, the weighted volume functional is given by
$$F:\{(\eta_{1},\ldots,\eta_{n})\in\mathfrak{t}:\eta_{i}>0\}\rightarrow\mathbb{R}_{>0},$$
\begin{equation*}
\begin{split}
F(\eta_{1},\ldots,\eta_{n})&=\sum_{i=1}^{n}\frac{e^{(k+1-n)\eta_{i}+\sum_{j\neq i}\eta_{j}}}{k\eta_{i}\Pi_{j\neq i}(\eta_{j}-\eta_{i})}\\
&=\frac{\sum_{i=1}^{n}(-1)^{i+1}\Pi_{j\neq i}\eta_{j}\Pi_{\substack{k,\,l\neq i\\ k>l}}(\eta_{k}-\eta_{l})
e^{(k+1-n)\eta_{i}+\sum_{j\neq i}\eta_{j}}}{k\Pi_{i=1}^{n}\eta_{i}\Pi_{i<j}(\eta_{i}-\eta_{j})}\quad\textrm{if $\eta_{k}\neq\eta_{l}$ for $k\neq l$.}
\end{split}
\end{equation*}
Again, by symmetry, the unique critical point of $F$ here must satisfy $\eta_{1}=\ldots=\eta_{n}$. By taking limits, one can write
down an expression for $F$ when this is the case. Differentiating the resulting expression and setting it equal to zero,
one obtains the polynomials  of \cite[equation (36)]{FIK}. For example, in low dimensions, when  $\eta_{1}=\ldots=\eta_{n}=:\eta$, we obtain the following formulae for $F$:
$$\begin{tabular}{|c|c|}\hline
Line bundle		& $F(\eta)$ 		\\\hline\hline
$\mathcal{O}(-1)\to\mathbb{P}^{1}$ 	& $\frac{(\eta+1)e^{\eta}}{\eta^{2}}$		\\\hline
$\mathcal{O}(-1)\to\mathbb{P}^{2}$ & $\frac{(2\eta^{2}+2\eta+1)e^{\eta}}{\eta^{3}}$\\\hline
$\mathcal{O}(-2)\to\mathbb{P}^{2}$ & $\frac{(\eta^{2}+2\eta+2)e^{2\eta}}{\eta^{3}}$\\\hline
$\mathcal{O}(-1)\to\mathbb{P}^{3}$ & $\frac{(9\eta^{3}+9\eta^{2}+6\eta+2)e^{\eta}}{\eta^{4}}$\\\hline
$\mathcal{O}(-2)\to\mathbb{P}^{3}$ & $\frac{(4\eta^{3}+6\eta^{2}+6\eta+3)e^{2\eta}}{\eta^{4}}$\\\hline
$\mathcal{O}(-3)\to\mathbb{P}^{3}$ & $\frac{(\eta^{3}+3\eta^{2}+6\eta+6)e^{3\eta}}{\eta^{4}}$\\\hline
\end{tabular}$$
The corresponding shrinking gradient K\"ahler-Ricci solitons are those of Feldman-Ilmanen-Knopf \cite{FIK} on these spaces.
\end{example}

\begin{example}
Let $L$ be the total space of a negative holomorphic line bundle over a Fano manifold $D$ of complex dimension $n$. By adjunction, in order for $L$ to admit
a shrinking gradient K\"ahler-Ricci soliton, we must have $c_{1}(-K_{D}\otimes L)>0$. Assuming that this is the case,
consider the action of the torus $T$ given by rotating the fibres of $L$. We have an induced
coordinate $w$ on the Lie algebra $\mathfrak{t}$ of $T$, where $1\in\mathfrak{t}$ will generate the vector field
$\operatorname{Im}(z_{i}\partial_{z_{i}})$ in a local trivialising chart of $L$. The zero set of every element of $\mathfrak{t}$
will be $D$, the zero section of $L$, and in this case the weighted volume functional $F$ on the domain
$\{\eta\in\mathfrak{t}:\eta>0\}$ is given by
\begin{equation}\label{lovely}
\begin{split}
F(\eta)&=\int_{D^{n}}\frac{e^{\eta}e^{\iota^{*}\omega}}{\eta\left(1+\frac{c_{1}(L)}{\eta}\right)}\\
&=\frac{e^{\eta}}{\eta}\int_{D^{n}}e^{\iota^{*}\omega}\left(1+\frac{c_{1}(L)}{\eta}\right)^{-1}\\
&=\frac{e^{\eta}}{\eta}\int_{D^{n}}e^{\iota^{*}\omega}\left(1-\frac{c_{1}(L^{*})}{\eta}\right)^{-1}\\
&=\frac{e^{\eta}}{\eta}\int_{D^{n}}\left(1+\omega+\frac{\omega^{2}}{2!}+\ldots\right)\left(1+\frac{c_{1}(L^{*})}{\eta}+\frac{c_{1}(L^{*})^{2}}{\eta^{2}}+\ldots\right)\\
&=\frac{e^{\eta}}{\eta}\int_{D^{n}}\sum_{i=0}^{n}\frac{\omega^{i}}{i!}\wedge\left(\frac{c_{1}(L^{*})^{n-i}}{\eta^{n-i}}\right)\\
&=\frac{e^{\eta}}{\eta}\sum_{i=0}^{n}\frac{1}{\eta^{n-i}i!}\int_{D^{n}}\omega^{i}\wedge c_{1}(L^{*})^{n-i}\\
&=\frac{e^{\eta}}{\eta^{n+1}}\sum_{i=0}^{n}\frac{\eta^{i}}{i!}\int_{D^{n}}c_{1}(K_{D}^{-1}\otimes L)^{i}\wedge c_{1}(L^{*})^{n-i}.\\
\end{split}
\end{equation}

This formula in particular applies to the total space of the line bundle $\mathcal{O}(-k)$ over $\mathbb{P}^{n-1}$ for
$0<k<n$. Its relationship to the formulae of Example \ref{three} is as follows. On the total space of
$\mathcal{O}(-k)$, we have two torus actions, one given by the standard action of a torus $T_{1}$ rotating the fibres of
$\mathcal{O}(-k)$, and another given by the action of a torus $T_{2}$ induced from the standard torus action on $\mathcal{O}(-1)$
that rotates the fibres. The formulae of Example \ref{three} with $\eta_{1}=\ldots=\eta_{n}$ are given with respect to the action of $T_{2}$, whereas formula \eqref{lovely} is with respect to $T_{1}$. Consequently, the formulae of Example \ref{three} with $\eta_{1}=\ldots=\eta_{n}$ are given by $F(k\eta)$, where $F$ is as in \eqref{lovely}.
\end{example}

\subsection{The domain of definition of the weighted volume functional}\label{e}

By Theorem \ref{dhthm} and Proposition \ref{properr}, we see that the weighted volume functional $F$ is defined on the open cone $\Lambda$ of elements of the Lie algebra of the torus admitting Hamiltonian potentials that are proper and bounded below if $\Lambda$ is non-empty. In this subsection, we characterise $\Lambda$ algebraically in the setting of asymptotically conical K\"ahler manifolds.

Our precise set-up is as follows. Let $(C_{0},\,g_{0})$ be a K\"ahler cone with apex $o$, complex structure $J_{0}$, and radial function $r$ so that $g_{0}=dr^{2}+r^{2}g_{S}$
for a Riemannian metric $g_{S}$ on the link $S=\{r=1\}$ of $C_{0}$. Let $\pi:M\to C_{0}$ be a quasi-projective resolution of $C_{0}$ that is equivariant with respect to the holomorphic isometric action on $C_{0}$ of the torus $T$ with Lie algebra $\mathfrak{t}$ generated by $\xi:=J_{0}r\partial_{r}$, and let $g$ be a K\"ahler metric on $M$ with
\begin{equation}\label{honeybunch}
|\pi_{*}g-g_{0}|_{g_{0}}=O(r^{-2})
\end{equation}
with respect to which $T$ acts isometrically in a Hamiltonian fashion with moment map
$\mu:M\to\mathfrak{t}^{*}$. Write $u_{Y}(x):=\langle\mu(x),\,Y\rangle, x\in M,$ for the Hamiltonian potential
of $Y\in\mathfrak{t}$ so that $du_{Y}=-\omega\lrcorner Y$, $\omega$ here the K\"ahler form of $g$, and set
$$\Lambda:=\{Y\in\mathfrak{t}:\textrm{$u_{Y}$ is proper and bounded below}\}.$$
Next, let $\mathcal{O}_{M}(M)$ (respectively $\mathcal{O}_{C_{0}}(C_{0})$) denote the global algebraic sections of the structure sheaf of $M$ (resp.~of $C_{0}$), and write $$\mathcal{O}_{M}(M)=\bigoplus_{\alpha\,\in\,\mathfrak{t}^{*}}\mathcal{H}_{\alpha}$$
for the weight decomposition under the action of $T$. Then we have:
\begin{theorem}\label{fme}
In the above situation,
\begin{equation*}
\Lambda=\{Y\in\mathfrak{t}:\textrm{$\alpha(Y)>0$ for all $\alpha\in\mathfrak{t}^{*}$ such that $\mathcal{H}_{\alpha}\neq\emptyset$ and $\alpha\neq0$}\}.
\end{equation*}
\end{theorem}

\subsection*{Proof of Theorem \ref{fme}}

Let $E$ denote the exceptional set of the resolution $\pi:M\to C_{0}$. In what follows, we will identify $M\setminus E$ with
$C_{0}\setminus\{0\}$ via $\pi$. Let us begin by making some useful observations. Let $X$ be the unique vector field on $M$ such that $d\pi(X)=r\partial_{r}$. Then $d\pi(JX)=\xi$, where $J$ denotes the complex structure on $M$, so that $JX\in\mathfrak{t}$ and $[X,\,Y]=0$ for every $Y\in\mathfrak{t}$.
Then we have:
\begin{lemma}\label{popp}
Let $Y\in\mathfrak{t}$ so that $Y$ defines a real holomorphic $g$-Killing vector field on $M$ with $[X,\,Y]=0$.
Then $Y$ is tangent to the level sets of $r$ on $C_{0}\setminus\{o\}$.
\end{lemma}

\begin{proof}
Since $T$ acts isometrically on $g$ and $g_{0}$, $Y$ will define a
holomorphic $g_{0}$-Killing vector field on $C_{0}$. We claim that such a vector field is tangent to the
level sets of $r$. Indeed, simply note that
$$0=\mathcal{L}_{Y}g_{0}=d(Y\cdot r)\otimes dr+dr\otimes d(Y\cdot r)+2rdr(Y)g_{S}+r^{2}\mathcal{L}_{Y}g_{S}.$$
Then plugging $\xi$ into both arguments on the right-hand side and observing that
$$[Y,\,\xi]=[Y,\,Jr\partial_{r}]=[Y,\,JX]=J[Y,\,X]=0$$
along the end of $C_{0}$ since $Y$ is holomorphic, we arrive at the fact that
\begin{equation*}
\begin{split}
-dr(Y)&=\frac{r}{2}(\mathcal{L}_{Y}g_{S})(\xi,\,\xi)=\frac{r}{2}\left((\mathcal{L}_{Y}(g_{S}(\xi,\,\xi))-2g_{S}([Y,\,\xi],\,\xi)\right)=0,\\
\end{split}
\end{equation*}
as required.
\end{proof}

We now demonstrate that
$$\Lambda\subseteq\{Y\in\mathfrak{t}:\textrm{$\alpha(Y)>0$ for all $\alpha\in\mathfrak{t}^{*}$ such that $\mathcal{H}_{\alpha}\neq\emptyset$ and $\alpha\neq0$}\}.$$
To this end, let $Y\in\Lambda$ so that the Hamiltonian potential $u_{Y}$ of $Y$ is proper and bounded below,
let $f$ be a non-constant holomorphic function on $M$ on which $Y$ acts with weight $\lambda$
so that $JY(f)=-\lambda f$, let $x\in M\setminus E$ be a point where $f(x)\neq0$,
and denote by $\gamma_{x}(t)$ the flow line of $-JY$ with $\gamma_{x}(0)=x$. Then
$$\frac{d}{dt}f(\gamma_{x}(t))=\lambda f(\gamma_{x}(t))$$
so that
\begin{equation}\label{star}
f(\gamma_{x}(t))=f(x)e^{-\lambda t}\quad\textrm{for all $t<0$}.
\end{equation}
Now, by definition, we have that $-JY=\nabla^{g}u_{Y}$ and so from
Proposition \ref{alix} we deduce that there is a sequence $t_{i}\to-\infty$ as $i\to+\infty$ such that $(\gamma_{x}(t_i))_i$ converges to a point $x_{\infty}\in M$ satisfying $\nabla^g u(x_{\infty})=0$. Since the fixed point set of $T$ is contained in $E$, we must have that $x_{\infty}\in E$. Let $x_{i}:=\gamma_{x}(t_{i})$. Then plugging $t_{i}$ into \eqref{star} yields the fact that
$$|f(x_{i})|=|f(x)|e^{-\lambda t_{i}}\to_{i\to\infty}\left\{
\begin{array}{rl}
+\infty & \text{if $\lambda < 0$},\\
0 & \text{if $\lambda > 0$.}
\end{array} \right.$$
Since $x_{i}\to x_{\infty}\in E$ as $i\to\infty$, we conclude from the maximum principle that $\lambda>0$ as required.

Next we show that
\begin{equation}\label{work}
\{Y\in\mathfrak{t}:\textrm{$\alpha(Y)>0$ for all $\alpha\in\mathfrak{t}^{*}$ such that $\mathcal{H}_{\alpha}\neq\emptyset$ and $\alpha\neq0$}\}\subseteq\Lambda.
\end{equation}
By \cite[Lemma 2.15]{Conlon}, $M$ is $1$-convex. By construction then, $\pi:M\to C_{0}$ will be the Remmert reduction
of $M$. In particular, we have that $\pi^{*}\mathcal{O}_{C_{0}}(C_{0})=\mathcal{O}_{M}(M)$ by the properties of the Remmert reduction.
Since $\pi:M\to C_{0}$ is equivariant with respect to the action of $T$, we thus see that
$Y\in\mathfrak{t}$ acts with weight $\lambda$ on $f\in\mathcal{O}_{C_{0}}(C_{0})$ if and only if
it acts with weight $\lambda$ on the unique lift $\pi^{*}f$ of $f$ to $\mathcal{O}_{M}(M)$. Applying
\cite[Proposition 2.7]{collinss}, we therefore deduce that
$$\{Y\in\mathfrak{t}:g_{S}(Y,\,\xi)>0\}=\{Y\in\mathfrak{t}:\textrm{$\alpha(Y)>0$ for all $\alpha\in\mathfrak{t}^{*}$ such that $\mathcal{H}_{\alpha}\neq\emptyset$ and $\alpha\neq0$}\}.$$
Consequently, in order to prove the inclusion \eqref{work}, it suffices to show that
\begin{equation*}
\{Y\in\mathfrak{t}:g_{S}(Y,\,\xi)>0\}\subseteq\Lambda.
\end{equation*}
This inclusion is established by the following proposition.
\begin{prop}
Let $Y\in\{Z\in\mathfrak{t}:g_{S}(Z,\,\xi)(x)>0\quad\textrm{for all $x\in S$}\}$ with Hamiltonian potential $u_{Y}$.
Then $u_{Y}\geq cr^{2}$ along the end of $C_{0}$ for some $c>0$. In particular, $u_{Y}$ is proper and bounded below.
\end{prop}

\begin{proof}
Let $x\in\{r=1\}$ and let $\gamma_{x}(t)$ denote the integral curve of $X$, the vector field on $M$ with $d\pi(X)=r\partial_{r}$, with $\gamma_{x}(0)=x$. Then we have that
\begin{equation}\label{honn}
\begin{split}
u_{Y}(\gamma_{x}(t))&=u_{Y}(\gamma_{x}(0))+\int_{0}^{t}du_{Y}(\dot{\gamma}_{x}(s))\,ds\\
&=u_{Y}(x)+\int_{0}^{t}g(-JY,\,X)(\gamma_{x}(s))\,ds\\
&=u_{Y}(x)+\int_{0}^{t}g(Y,\,JX)(\gamma_{x}(s))\,ds.\\
\end{split}
\end{equation}
Next observe that since $Y\in\mathfrak{t}$, $Y$ is tangent to the level sets of $r$ by Lemma \ref{popp}. Hence, the asymptotics \eqref{honeybunch}
give us that
\begin{equation*}
\begin{split}
g(Y,\,JX)(\gamma_{x}(s))&=g_{0}(Y,\,JX)+\underbrace{O(r^{-2})|Y|_{g_{0}}|JX|_{g_{0}}}_{=\,O(1)}\\
&=O(1)+r(\gamma_{x}(s))^{2}g_{S}(Y,\,\xi)\\
&=O(1)+r(x)^{2}e^{2s}g_{S}(Y,\,\xi),
\end{split}
\end{equation*}
where the final equality follows from the fact that $r(\gamma_{x}(s))=r(x)e^{s}$ because
$$\frac{\partial}{\partial s}(r(\gamma_{x}(s)))=r(\gamma_{x}(s))\quad\textrm{and $\gamma_{x}(0)=x$}.$$
Plugging this into \eqref{honn} yields
\begin{equation*}
\begin{split}
u_{Y}(\gamma_{x}(t))&=u_{Y}(x)+\frac{1}{2}r(x)^{2}g_{S}(Y,\,\xi)(e^{2t}-1)+O(t)\\
&=u_{Y}(x)-\frac{1}{2}g_{S}(Y,\,\xi)r(x)^{2}+\frac{1}{2}g_{S}(Y,\,\xi)r(\gamma_{x}(t))^{2}+O(\ln r(\gamma_{x}(t)))\\
&\geq cr(\gamma_{x}(t))^{2}
\end{split}
\end{equation*}
along the end of $C_{0}$ for some $c>0$, since $g_{S}(Y,\,\xi)>0$. From this, the assertion follows.
\end{proof}

\subsection{Coercive estimates on Hamiltonian potentials}\label{f}
The goal of this subsection is to prove sharp positive bounds on the growth of the Hamiltonian potential
of a real holomorphic Killing vector field on a complete shrinking gradient
K\"ahler-Ricci soliton $(M,\,g,\,X)$ that commutes with the soliton vector field $X$ under certain conditions.
Since $H^{1}(M)=0$ by \cite{wyliee}, such a vector field always admits a Hamiltonian potential.
Let $\omega$ denote the K\"ahler form of $g$ and recall that for each real holomorphic Killing vector field $Y$ on $M$ commuting with $X$, the Hamiltonian $u_{Y}$ of $Y$ is normalised so that $\Delta_{\omega}u_Y+u_Y+\frac{1}{2}JY\cdot f=0$. Since $du_Y=-\omega\lrcorner Y$ by definition, one sees that
\begin{eqnarray}
\Delta_{\omega}u_Y+u_Y=-\frac{1}{2}g(JY,X)=\frac{1}{2}g(\nabla u_Y,X)=\frac{1}{2}X\cdot u_Y,\label{evo-eqn-pot}
\end{eqnarray}
an identity that shall prove useful in what follows. We will prove:
\begin{prop}\label{prop-growth-pot}
Let $(M,\,g,\,X)$ be a complete shrinking gradient K\"ahler-Ricci soliton with bounded Ricci curvature
with soliton vector field $X=\nabla^{g}f$ for a smooth real-valued function $f:M\to\mathbb{R}$.
Let $Y$ be a real holomorphic Killing vector field on $M$ commuting with $X$ and assume that the Hamiltonian potential $u_{Y}$ of $Y$
is proper and bounded below. Then there exist positive constants $c_1$ and $c_2$ such that $c_{1}f\leq u_Y\leq c_{2}f$ outside of a compact set.
\end{prop}

As a consequence of this proposition, we see, without appealing to the Duistermaat-Heckman theorem, that the weighted volume functional is defined on a complete shrinking gradient K\"ahler-Ricci soliton $(M,\,g,\,X)$ with bounded Ricci curvature on those elements
admitting a Hamiltonian potential that is proper and bounded below in the Lie algebra of any torus
that acts in a holomorphic Hamiltonian fashion on $M$ and contains the flow of $JX$, $J$ here denoting the complex structure of $M$.

\begin{proof}[Proof of Proposition \ref{prop-growth-pot}]
Let $|\,\cdot\,|_{g}$ denote the norm with respect to $g$. By Proposition \ref{keywest}, there exist positive constants $c$ and $c_0$ such that $|Y|^2_g(x)\leq cf(x)$ if $f(x)\geq c_0$. Since $|\nabla u_Y|_g=|JY|_g=|Y|_g$, one obtains the expected growth on $u_Y$ by integrating the previous estimate on the norm of $Y$.

We next prove the lower bound on $u_Y$. First notice that since $u_Y$ is proper and bounded below, $u_Y$ is strictly positive outside a sufficiently large compact set of the form $\{f\leq c_0\}$. Recall from Remark \ref{rk-sol-id} that the normalisation of $f$ is determined by the soliton identities which in this case yield $\Delta_{g}f-X\cdot f=-2f$. Using $f$ as a barrier function together with \eqref{evo-eqn-pot}, we compute the weighted Laplacian of the difference of the inverses of $u_Y$ and $f$ on the region where this difference makes sense. We have:
\begin{eqnarray}
\left(\Delta_{\omega}-\frac{1}{2}X\cdot\right)\left(\frac{1}{u_Y}-\frac{C}{f}\right)=\frac{1}{u_Y}-\frac{C}{f}+2\frac{|\nabla u_Y|^2}{u_Y^3}-2C\frac{|\nabla f|^2}{f^3},\label{evo-eqn-aux-fct}
\end{eqnarray}
where $C$ is a positive constant to be specified later. Assume that the function $u_Y^{-1}-Cf^{-1}$ attains its maximum at an interior point $x_0$ of a domain of the form $\{r_1\leq f\leq r_2\}$ with $r_1$ sufficiently large so that both $u_Y$ and $f$ are strictly positive. At such a point $x_0$, one sees that $\nabla(u_Y^{-1}-Cf^{-1})(x_0)=0$ and from the maximum principle that
\begin{eqnarray*}
0\geq\left(\Delta_{\omega}-\frac{1}{2}X\cdot\right)\left(\frac{1}{u_Y}-\frac{C}{f}\right)(x_0).
\end{eqnarray*}
This information, together with (\ref{evo-eqn-aux-fct}), implies that at $x_0$,
\begin{eqnarray*}
0&\geq&\frac{1}{u_Y}-\frac{C}{f}+2\frac{|\nabla u_Y|^2}{u_Y^3}-2C\frac{|\nabla f|^2}{f^3}\\
&=&\frac{1}{u_Y}-\frac{C}{f}+2u_Y|\nabla u_Y^{-1}|^2-2Cf|\nabla f^{-1}|^2\\
&=&\frac{1}{u_Y}-\frac{C}{f}+2u_YC^2|\nabla f^{-1}|^2-2Cf|\nabla f^{-1}|^2\\
&=&\left(\frac{1}{u_Y}-\frac{C}{f}\right)\left(1-2Cu_Yf|\nabla f^{-1}|^2\right)\\
&=&\left(\frac{1}{u_Y}-\frac{C}{f}\right)\left(1-2C\frac{u_Y}{f^3}|\nabla f|^2\right).
\end{eqnarray*}
Next, using the fact that $|\nabla f|^2$ grows quadratically by the soliton identities,
we see from the upper bound on $u_Y$ that on $M$,
\begin{eqnarray*}
2C\frac{u_Y}{f^3}|\nabla f|^2\leq \frac{2Cd}{f}
\end{eqnarray*}
for some positive constant $d$ uniform in $r_1$ and $r_2$. In particular, the term $(1-2Cu_Yf^{-3}|\nabla f|^2)$ is positive on $\{r_1\leq f\leq r_2\}$ as long as $2Cf^{-1}d$ is strictly less than $1$, or equivalently, as long as $C< (2d)^{-1}r_1.$

In summary, for any heights $r_1<r_2$ and any constant $C$ such that $C< (2d)^{-1}r_1$, we have that
\begin{equation*}
\max_{r_1\,\leq\,f\,\leq\,r_2}\left(\frac{1}{u_Y}-\frac{C}{f}\right)\leq \max\left\{0,\max_{ f\,=\,r_1}\left(\frac{1}{u_Y}-\frac{C}{f}\right),\max_{ f\,=\, r_2}\left(\frac{1}{u_Y}-\frac{C}{f}\right)\right\}.
\end{equation*}
Since $u_Y$ (and $f$) tend to $+\infty$ as $f$ approaches $+\infty$, one sees, by letting $r_2$ tend to $+\infty$, that
\begin{eqnarray}\label{wtf}
\max_{r_1\,\leq\, f}\left(\frac{1}{u_Y}-\frac{C}{f}\right)\leq\max\left\{0,\max_{ f\,=\, r_1}\left(\frac{1}{u_Y}-\frac{C}{f}\right)\right\}.
\end{eqnarray}
We choose $C$ and $r_1$ such that the right-hand side of \eqref{wtf} is non-positive and such that $C< (2d)^{-1}r_1$. Indeed, since $u_Y$ is proper, there exists some positive height $r_1$ such that $\min_{f=r_1}u_Y\geq 4d$. Thus, if $C:=(4d)^{-1}r_1$, then $C<(2d)^{-1}r_1$ and
\begin{eqnarray*}
\max_{ f\,=\, r_1}\left(\frac{1}{u_Y}-\frac{C}{f}\right)\leq 0,
\end{eqnarray*}
as required. This completes the proof of the proposition.
\end{proof}

\newpage
\bibliographystyle{amsalpha}

\bibliography{ref2}

\end{document}